\documentclass[11pt]{article}
\usepackage[top=2.5cm, bottom=2.5cm, left=2.5cm, right=2.5cm] {geometry}
\usepackage{lineno}
\usepackage{lmodern}
\usepackage{comment}
\usepackage{amsmath,amssymb,theorem,color}
\numberwithin{equation}{section} 
\usepackage{amssymb}
\usepackage{color}
\usepackage{amsmath,amssymb,theorem,color,amsfonts,mathrsfs}
\usepackage{graphicx}
\usepackage{amsfonts}%
\usepackage{xcolor}
\usepackage{soul}
\usepackage{marginnote}

\usepackage{accents}

{\Large }

\newcommand{\R}{\ensuremath{\mathbb{R}}}

\newcommand{\N}{\ensuremath{\mathbb{N}}}
\newcommand{\eps}{\varepsilon}

\newcommand{\B}{\mathbb{B}}
\newcommand{\cA}{\mathscr{A}}
\newcommand{\cC}{\mathscr{C}}

\newcommand{\bP}{\mathbb{P}}

\newcommand{\mP}{\mathbb{P}}

\newcommand{\X}{\mathbb{X}}

\newcommand{\E}{\mathbb{E}}

\newcommand{\ta}{\mathfrak{a}}
\newcommand{\tA}{\mathfrak{A}}

\newcommand{\cD}{\mathscr{D}}
\newcommand{\cS}{\mathcal{S}}

\newcommand{\cL}{\mathcal{L}}  

\newcommand{\bx}{\mathbf x}
\newcommand{\bX}{\mathbf X}
\newcommand{\bz}{\ta^\Delta}  

\newcommand{\bw}{\mathbf w}

\newcommand{\ty}{\bar y}
\newcommand{\tz}{\bar z}
\newcommand{\tphi}{\bar{\phi}}
\newcommand{\tpsi}{\bar{\psi}}

\newcommand{\teta}{\bar{\eta}}

\newcommand{\tp}{p{\rm -var}}
\newcommand{\tq}{q{\rm -var}}

{\Large }

\newcommand{\ltn}{\ensuremath{\left| \! \left| \! \left|}}
\newcommand{\rtn}{\ensuremath{\right| \! \right| \! \right|}}

\newtheorem{theorem}{Theorem}[section]
{ \theorembodyfont{\normalfont} 
	\newtheorem{example}[theorem]{Example}
	\newtheorem{remark}[theorem]{Remark}
}

\newtheorem{definition}[theorem]{Definition}
\newtheorem{lemma}[theorem]{Lemma}
\newtheorem{corollary}[theorem]{Corollary}
\newtheorem{proposition}[theorem]{Proposition}

\newcounter{enumctr}

\begin{document}
	
	\title{Stability criteria for rough systems}
	\author{
		Luu Hoang Duc\thanks{Department of Mathematics, University of Klagenfurt, Austria \& Max-Planck-Institute for Mathematics in the Sciences, Leipzig, Germany \& Institute of Mathematics, Vietnam Academy of Science and Technology, Vietnam. {\it E-mail:  duc.luu@aau.at,
        duc.luu@mis.mpg.de, lhduc@math.ac.vn}}
        , $\;$ Phan Thanh Hong \thanks{Thang Long University, Hanoi, Vietnam {\it E-mail: hongpt@thanglong.edu.vn }}, $\;$ Nguyen Dinh Cong\thanks{Institute of Mathematics, Vietnam Academy of Science and Technology, Vietnam {\it E-mail: ndcong@math.ac.vn}}}
	\date{{\it submitted draft}}
	\maketitle

	\begin{abstract}
		We propose a quantitative direct method to prove the local stability of a stationary solution for a rough differential equation and its regular discretization scheme. Using Doss-Sussmann technique and stopping time analysis, we provide stability criteria for a stationary solution of the continuous system to be exponentially stable, provided the diffusion term is bounded and its derivatives exhibit small growth. The same conclusions hold for the regular discretization scheme with a sufficiently small step size, but one needs to apply the sewing lemma and stopping times for the discrete time set. Our stability criteria are based on the linearization of the drift and require only information about the bound and growth rates of the diffusion, making them data-driven criteria.
	\end{abstract}

	
	
	\section{Introduction}
	
	This paper deals with the local asymptotic stability criteria for rough differential equations on $\R^d$ of the form 
	\begin{equation}\label{RDE1}
		dy = f(y)dt + g(y)dx.
	\end{equation}
Equation \eqref{RDE1} can be viewed as a controlled differential equation driven by rough path $x \in C^{\nu}([0,T],\R^m)$ in the sense of Lyons \cite{lyons98}, \cite{lyonsetal07} or of Friz-Victoir \cite{friz}. It can also be interpreted as a rough intergral equation for controlled rough paths in the sense of Gubinelli \cite{gubinelli}. As such, system \eqref{RDE1} appears as a pathwise approach to solve a stochastic differential equation which is driven by a certain H\"older noise $X_t$ (e.g. fractional Brownian motions).
	
	Under certain assumptions imposed on its coefficient functions, a rough differential equation will have the property of
	the local existence, uniqueness and continuity of  solution given initial conditions, see e.g. \cite{gubinelli} or \cite{frizhairer} for a version without drift coefficient function, and \cite{riedelScheutzow}, \cite{duckloeden} for a full version using variation or H\"older norms. Moreover, with a stronger assumption on the coefficient functions a rough differential equation will have global solutions which exist throughout the time interval $[0,\infty)$ (see e.g. \cite{bailleul} for non-explosion criteria of the solutions). This gives rise to  the recent interest in investigation of qualitative properties of rough differential equation with a view to the parallel in the well developed and well known qualitative theory of ordinary differential equations.

    The topic of random attractors and its upper semi-continuity has been studied for the random dynamical system generated from \eqref{RDE1} and its discretization, which applies Lyapunov methods in many works, see e.g. \cite{GAKLBSch2010}, \cite{duc21}, \cite{duckloeden}, \cite{congduchong23} and the references therein.  In particular, it is proved in \cite{duc21} that for strictly dissipative $f: \R^d \to \R^d$ and a sufficiently small $g \in C^3_b$ there exists a singleton random attractor to which every trajectory converges in the pathwise sense and with an exponential rate. Moreover for globally Lipschitz continuous $f$, there exists a numerical attractor of the Euler discretization scheme of \eqref{RDE1} which converges to the continuous attractor as the step size converges to zero. However, the existence of a global attractor often does not guarantee the local stability of any reference solution even if it stays inside the attractor. 
		
	This gives a motivation for studying the asymptotic stability for path-wise solution of \eqref{RDE1}. The topic of stability for stochastic differential equation with H\"older noises dated back from our previous work \cite{ducGANSch18, duchongcong18} using both the semigroup method and the Lyapunov function method, with an improved version in \cite{ducSD22} for the stability of the trivial solution. More precisely, the semigroup method in \cite{GAKLBSch2010} is developed further in \cite{ducGANSch18, duchongcong18,ducSD22}, to deal with fractional Brownian noise with small intensity. The semigroup technique is also used to study local stability of the trivial solution in \cite[Theorem 18]{GABSch18} on a small neighborhood $B(0,r)$, using the cutoff technique and fractional calculus, and under the assumption that $g(x)$ is flat, i.e. $g(0) =0,\  Dg(0) = 0,\ D^2g(0)= 0$. Later, this condition is improved in \cite{ducSD22}, where it applies the semigroup technique to derive a global stability criterion that requires the linearized part $Df(0) \in \R^{d \times d}$ to have eigenvalues of negative real parts and the non-linear part $F(y) =f(y)-Df(0)y$ to be globally Lipschitz w.r.t. Lipschitz constant $C_F$  such that $C_F$ and $\|Dg\|_\infty \vee \|D^2g\|_\infty \vee \|D^3g\|_\infty $ are sufficiently small. 
 Recently, the semigroup technique is also used in \cite{riedel25} to deal with the Lyapunov exponents and the related topics regarding to the local dynamics in the vicinity of a stationary solution, such as exponential stability and invariant manifolds. 

However, we have a different vision in developing results on exponential stability. Since our previous works \cite{ducGANSch18}, \cite{ducSD22}, our ideas have stemmed from the observation that the information on the linearization of the whole system is not available in practice, as we often have only the drift $f$ and a trajectory of a reference stationary solution and not much information on the diffusion $g$. For this reason, it is often too costly, if not impossible, to obtain the information on the linearization along the stationary solution of the entire system \eqref{RDE1} and to investigate its Lyapunov spectrum, as suggested recently in \cite{riedel25}. Therefore in our studies we often consider and impose conditions on only the linearization of the drift $Df$ rather than of the entire system to meet the reality. On the other hand, we would like to develop the Lyapunov function method in studying stability for rough systems, as done in our previous result in \cite{duchongcong18} for nonautonomous Young equations.

In this paper we revisit the Lyapunov asymptotic stability of the rough differential equation \eqref{RDE1} near a stationary solution which can be, in many cases, time dependent, e.g. a singleton random attractor. In that scenarios, one can test stability of a time dependent reference solution $\ta_t$ of \eqref{RDE1} by checking the dynamics of the unperturbed deterministic system $\dot{y} = f(y)$ along a neighborhood of $\ta_t$, by considering e.g. the linearization $\dot{\xi} = Df(\ta_t)\xi$ whose data is often available. Roughly speaking, our main results, Theorem \ref{local_diss_attractor} and Theorem \ref{stable1}, for rough systems can be combined and reformulated as below.
	\begin{theorem}[Main results] Assume conditions for the drift $f \in C^1$, the diffusion $g$ and the driving noise $X$ so that there exists pathwise solution for the equation \eqref{RDE1} and the generated Wiener shift is ergodic. Let $\ta(\cdot)$ be a stationary solution of \eqref{RDE1} that is $L^\rho$-integrable. If the linearization of the unperturbed system $\dot{\xi} = f(\xi)$ along $\ta(\cdot)$ (i.e. the system $\dot{\xi} = Df(\ta_t)\xi$) is exponentially stable, then $\ta_t$ is exponentially stable almost surely for the perturbed system \eqref{RDE1}, as long as $C_g$ defined in \eqref{gcond.new} is sufficiently small and satisfies a stability criterion. If $f$ is globally Lipschitz continuous, the same conclusion holds for the regular discretization scheme with sufficiently small $L_g$ in \eqref{Lgdis} and step size $\Delta$.
    \end{theorem}   
	 Note that the linearized system $\dot{\xi} = Df(\ta_t)\xi$ is exponentially stable under the sufficient condition $\E \ell(f,\ta(\cdot)) <0$ where $\ell(f,y):=\sup \limits_{\|h\| =1} \langle h, Df(y)h \rangle$. This is due to an intuitive estimate 
    \[
    \frac{1}{t} \Big(\log \|\xi_t\|-\log \|\xi_0\|\Big) = \frac{1}{t} \int_0^t \langle \frac{\xi_s}{\|\xi_s\|}, Df(\ta(\theta_s \omega))\frac{\xi_s}{\|\xi_s\|} \rangle ds \leq \frac{1}{t} \int_0^t  \ell(f,\ta(\theta_s \omega))ds \to \E \ell(f,\ta(\cdot)).
    \]
    Theorem 1.1 is applicable to the class of equations \eqref{RDE1} with $f$ globally Lipschitz continuous and strictly dissipative and smooth $g$ with $C_g$ small enough, because one can prove the existence of a singleton pullback attractor, which is integrable and converges to the unique fixed point $\ta^*$ (also the global attractor) of the unperturbed system $\dot{u} = f(u)$ as $C_g$ tends to zero (see \cite[Theorem 3.3]{duc21}). In that case, condition $\E \ell(f,\ta(\cdot)) \approx \ell(f,\ta^*)<0$ is satisfied for sufficiently small $C_g$. For more concrete examples, see Example \ref{expitchfork}, \ref{FHNex}.\\
	Our method employs the direct method of Lyapunov, which aims to estimate the norm growth (or a Lyapunov-type function) of the solution in discrete time intervals. For the continuous time set, this technique is feasible thanks to the so-called {\it Doss-Sussmann technique} to transform the rough differential equation on each stopping time interval to an ordinary differential equation, that can be seen as a non-autonomous perturbation of the unperturbed ODE. The tricky part is proving stability, which means to control the norm difference between a solution and the reference stationary solution to be less than a parameter $\epsilon$ on each stopping time interval. This can be done in an indirect way by comparing the norm difference between the two images after the Doss-Sussmann transformation, thanks to Proposition \ref{solest} and Proposition \ref{solestdiff} which are the improved results in \cite{duc21} for $C_g$ defined in \eqref{gcond.new}. Exponential attractivity is then an indirect consequence of the Birkhoff's ergodic theorem (via the crucial Lemma \ref{infimumD}), provided that the generated {\it Wiener shift} is ergodic (this holds for fractional Brownian motions, see Lemma \ref{ergodicity} in Appendix \ref{appen}). The results in Theorem \ref{local_diss_attractor} and Theorem \ref{main1}, with stability criteria depending on the parameter $C_g$ in \eqref{gcond.new}, are significantly stronger and more practical than the ones in  \cite[Theorem 18]{GABSch18} or recently in \cite{riedel25} for general stationary solutions. Moreover, no further information on the diffusion $g$ or its derivatives is required, hence many complicated tasks when working with the linearization of \eqref{RDE1} can be avoided.
	
	The situation is rather complicated for the discretized system of \eqref{RDE1}, because the Doss-Sussmann transformation fails to control the solution norm to be sufficiently small in the discrete time set, and in fact the solution can exit the local regime of the trivial solution right after one discrete time step. To overcome this difficulty, in Section \ref{Sectime} we develop further the stopping time technique in our previous work \cite{congduchong23} for discrete time sets, to estimate the norm growth of the solution on each stopping time interval, see e.g. Proposition \ref{hnew}. Another important step is to prove Proposition \ref{case1}, where we apply the discrete sewing lemma \cite{davie}, \cite{lejay} to estimate the difference between the solution starting at a stopping time $\tau_0^\Delta$ from an initial value $y_{\tau^\Delta_0}$ inside a random neighborhood of $\ta(\omega)$ and the reference solution $\ta(\theta_\cdot \omega)$ at the next stopping time $\tau_1^\Delta$. This crucial estimate and the induction principle then help control the difference between two solutions to be sufficiently small, and finally lead to exponential stability. The stability criteria in Theorem \ref{stable1}, Theorem \ref{disccridiss} and Theorem \ref{main1_new} are new, where the choice of the parameter $L_g$ defined in \eqref{Lgdis} and the step size $\Delta$ are independent and pathwise free. 
	
	When the assumption on ergodicity of the generated Wiener shift is relaxed, the ergodic Birkhoff theorem is still applicable in estimating stopping times, but results in random variable limits. In this case, all conclusions in the main theorems still hold almost surely, but all the stability criteria and parameters would be path dependent. In addition, our method still works for a lower regularity coefficient $\nu \leq \frac{1}{3}$, although the computations would be rather complicated. Moreover, it could also be applied in case the diffusion part $g$ is linear, or in case $f$ also depends on time, i.e. $f(y)$ is replaced by $f(t,y)$. 
    
	We close the introduction part with the note that our stability criteria for stationary solutions agree with the previous ones in \cite{ducSD22} that $C_g$ in \eqref{gcond.new} and $L_g$ in \eqref{Lgdis} can be sufficiently small by choosing $\|Dg\|_\infty \vee \|D^2g\|_\infty \vee \|D^3g\|_\infty $ to be sufficiently small, thus in general $\|g\|_\infty$ can be very large. The following counter example shows that choosing large $C_g$ and $L_g$ might break stability.
	
	\begin{example}
		Consider the It\^o stochastic differential equation
		\begin{equation}\label{ex1} 
			d\Big(\begin{array}{cc} y_1 \\y_2\end{array}\Big) = \Big(\begin{array}{cc} y_1(\mu-y_1^2-y_2^2) \\ y_2(\mu-y_1^2-y_2^2) \end{array}\Big) dt + \Big(\begin{array}{ccc} 0 & \sigma  \\ -\sigma & 0 \end{array}\Big) \Big(\begin{array}{cc} y_1 \\ y_2 \end{array}\Big) dB_t
		\end{equation}
		or in short $dy= f(y)dt + g(y)dB_t$, where $B$ is a scalar standard Brownian motion with all realization in $C^\nu$ for $\nu \in (\frac{1}{3},\frac{1}{2})$; and $\mu,\sigma$ are real constants. We omit the issue of existence and uniqueness of the solution of equation \eqref{ex1}, noting that it can be solved either in the It\^o sense using polar coordinates $y_1 =r \sin \alpha, y_2 = r \cos \alpha$ or in the pathwise sense as a rough differential equation (see e.g. \cite[Theorem 2.1]{duc21} and \cite[Theorem 5.1]{duckloeden}). In particular, one can apply It\^o formula to check that
		\[
		d\|y\|^2 = \Big[2 \langle y, f(y) \rangle + \|g(y)\|^2\Big]dt + 2 \langle y, g(y) \rangle dB_t = \Big( 2 \mu + \sigma^2  - 2\|y\|^2 \Big) \|y\|^2  dt.
		\]
		In other words, $\eta =\|y\|^2$ is the solution of the ordinary differential equation
		\begin{equation}\label{ex2} 
			\frac{d}{dt}{\eta} = \eta ( 2 \mu + \sigma^2  - 2\eta).
		\end{equation}
		Clearly, zero is the trivial solution of \eqref{ex1}. For $\sigma = 0$ and $\mu <0$, system \eqref{ex1} is an ordinary differential equation, which admits zero as an globally asymptotically stable solution. However, when $\sigma >0$ large enough such that $\sigma^2 + 2\mu >0$, then the zero solution of \eqref{ex2} becomes unstable. In other words, the zero solution of \eqref{ex1} will break its stability when perturbed by the linear noise for $\sigma$ large enough.  
	\end{example} 
	
	 The paper is organized as follows. In Section 2 we present assumptions for \eqref{RDE1} and recall Doss-Sussmann transformations used for the associate pure rough differential equation of \eqref{RDE1}. Section 3 discusses the construction of stopping times for both continuous and discrete time sets. In Section 4, we establish in Theorem \ref{local_diss_attractor} an exponential stability criterion of stationary solution of  \eqref{RDE1}. Section 5 is devoted to the study of discrete systems, where a similar stability criterion is proved in Theorem \ref{stable1}. We present in the appendix some basic definitions on rough paths, rough integrals and technical proofs for the auxiliary results.
    
	\section{Rough differential equations, Doss-Sussmann technique}\label{existunique}
Throughout the paper, we will assume the following global conditions on the equation \eqref{RDE1}.\\

(${\textbf H}_{f}$) 
 $f$ is locally Lipschitz continuous and of one-sided linear growth
\begin{equation}\label{onesideLipschitz}
\exists C>0:	\langle y,f(y) \rangle \leq C(1+ \|y\|^2),\quad \forall y \in \R^d; 
\end{equation}
in addition $f$ is of linear growth in the perpendicular direction, i.e. there exists $C_f>0$ such that
\begin{equation}\label{lineargrowth}
\Big\| f(y) - \frac{\langle f(y), y \rangle }{\|y\|^2}y\Big\|\leq C_f \Big(1+\|y\|\Big),\quad \forall y \ne 0.
\end{equation}

(${\textbf H}_g$) $g$ is in $C^3_b(\R^d,\cL(\R^m,\R^d))$ where we define
\begin{eqnarray}\label{gcond.new}
	C_g &:=& \max \Big\{\|Dg\|_\infty,\sqrt{\|Dg\|_\infty \|g\|_\infty},\sqrt{\|D^2g\|_\infty \|g\|_\infty},\notag\\ 
    &&\qquad \qquad \sqrt{\|D^3g\|_\infty \|g\|_\infty},\Big( \|D^2g\|_\infty\|g\|^2_\infty \Big)^{\frac{1}{3}}, \Big(\|D^3g\|_\infty\|g\|^2_\infty\Big)^{\frac{1}{3}}\Big\}.
\end{eqnarray}
	
	(${\textbf H}_X$) for a given $\nu \in(\frac{1}{3},\frac{1}{2}]$, $x \in C^{\nu}(\R, \R^m)$ - the space of all H\"older continuous paths such that $x$ is a realization of a stochastic process $X_t(\omega)$ with stationary increments and that $x$ can be lifted into a realized component $\bx = (x,\X)$ of a stochastic process $(x_\cdot(\omega),\X_{\cdot,\cdot}(\omega))$ with stationary increments. Moreover the estimate
	\begin{equation}\label{conditionx}
		E \Big(\|x_{s,t} \|^p +\|\X_{s,t}\|^{q}\Big)\leq C_{T,\nu} |t-s|^{p \nu },\qquad\forall s,t \in [0,T],
	\end{equation}
	holds for any $[0,T]\subset [0,\infty)$, with $p\nu \geq 1, q = \frac{p}{2}$ and some constant $C_{T,\nu}$. \\
	
	Concerning the assumption (${\textbf H}_X$), such a stochastic process, in particular, can be a fractional Brownian motion $B^H$  with Hurst exponent $H \in (\frac{1}{3},1)$.
	
	The existence and uniqueness theorem for system \eqref{RDE1} is first proved in \cite{riedelScheutzow}, where the solution is understood in the sense of Friz \&Victoir \cite{friz}. Using rough path integrals \cite{gubinelli}, we interpret the rough differential equation \eqref{RDE1} by writing it in the integral form 
	\begin{equation}\label{RDEintegral}
		y_t = y_a + \int_a^t f(y_s)ds + \int_a^t g(y_s)dx_s, \quad \forall t \in [a,b], 
	\end{equation}
	for any interval $[a,b]$ and an initial value $y_a \in \R^d$. Then we search for a solution in the Gubinelli sense, and solve for a path $y$ which is controlled by $x$. We refer to \cite{duc20} and Appendix \ref{appen} for definitions of variation and H\"older norms, Gubinelli rough integrals for controlled rough paths. Assumptions (${\textbf H}_f$), (${\textbf H}_g$), (${\textbf H}_X$) are sufficient to prove the existence and uniqueness of the solution of \eqref{RDE1} defined for any initial value $y_0\in\R^d$,
    as well as the continuity of the solution semi-flow and the generation of a continuous random dynamical system, see e.g. \cite[Theorem 4.3]{riedelScheutzow}, \cite{BRSch17}, \cite{duc21} and \cite{duckloeden}. 
	 
	Note that from \cite[Theorem 3.4]{duc20}, the global solution $\phi_\cdot(\bx,\phi_a)$ of the {\it pure} rough differential equation 
	\begin{equation}\label{pure}
		d \phi_u = g(\phi_u)dx_u,\quad u \in [a,b], \phi_a \in \R^d 
	\end{equation}
	is $C^1$ w.r.t. $\phi_a$, and $\frac{\partial \phi }{\partial \phi_a}(\cdot,\bx,\phi_a)$ is the solution of the linearized system 
	\begin{equation}\label{linearized}
		d\xi_u = Dg(\phi_u(\bx,\phi_s)) \xi_u dx_u,\quad u \in [a,b], \xi_a = Id, 
	\end{equation}
	where $Id \in \R^{d\times d}$ denotes the identity matrix. The idea is then to prove the existence and uniqueness of the global  solution on each small interval $[\tau_k,\tau_{k+1}]$ between two consecutive stopping times,
	and then concatenate to obtain the conclusion on any interval.
	The Doss-Sussmann technique used in \cite[Theorem 3.7]{duc20}  and \cite{riedelScheutzow} ensures 
	that, by a transformation $y_t = \phi_t(\bx,z_t)$ there is an one-to-one correspondence between a solution $y_t$ of \eqref{RDE1} on a certain interval $[0,\tau]$ 
	and a  solution $z_t$ of the associate ordinary differential equation
	\begin{equation}\label{ascoODE}
		\dot{z}_t = \Big[\frac{\partial \phi }{\partial z}(t,\bx,z_t)\Big]^{-1} f( \phi_t(\bx,z_t)),\quad t \in [0,\tau],\ z_0 = y_0.
	\end{equation}
	The parameter $\tau >0$ can be chosen such that 
	\begin{equation*}
		8 C_p C_g \ltn \bx \rtn_{\tp,[0,\tau]} \leq \lambda^*,\quad \text{for some}\quad \lambda^* \in (0,1), 
	\end{equation*}
	where $\ltn \bx\rtn_{\tp,[a,b]}$ is defined in \eqref{pvarnorm}, $C_g$ is defined in \eqref{gcond.new},  and $C_p$ is defined in \eqref{roughpvar}.
	
	The following result is an improvement of \cite[Proposition 1]{duc21} and \cite[Proposition 2.3]{duckloeden}, which shows solution norm estimates for equation \eqref{pure} (see the proofs in Appendix \ref{proofs}). Note that the estimates are dependent on $C_g$ from \eqref{gcond.new} and thus dependent on all the derivatives of $g$. 
\begin{proposition}\label{solest}
			Let $\phi_t$ be a solution of \eqref{pure}. Assume the condition (${\textbf H}_g$). Suppose further that $\lambda :=C_p C_g \ltn \bx \rtn_{\tp,[a,b]} \leq 1/8$.
            
(i) The following inequalities hold
\begin{equation}
\ltn \phi \rtn_{\tp,[a,b]} \leq 2(\|g\|_\infty \ltn x \rtn_{\tp, [a,b]}+\lambda^2),\quad 	\ltn R^{\phi} \rtn_{\tp,[a,b]} \leq 2\lambda^2; \label{solest1a}
\end{equation}
and in case the condition $g(0)=0$ is satisfied then
			\begin{equation}\label{solest1}
				\ltn \phi \rtn_{\tp,[a,b]} \leq 2\lambda\|\phi_a\|,\quad 	\ltn R^{\phi} \rtn_{\tp,[a,b]} \leq 2\lambda^2 \|\phi_a\|.  
			\end{equation}
            
(ii) Furthermore, it holds for $t\in [a,b]$ that
\begin{equation}
\Big\| \frac{\partial \phi }{\partial \phi_a}(t,\bx,\phi_a) - I \Big\|, \Big\| \Big[\frac{\partial \phi }{\partial \phi_a}(t,\bx,\phi_a)\Big]^{-1} - I\Big\|\leq 4\lambda.\label{solest3}
\end{equation}           
\end{proposition}
We would also like to modify the result in \cite[Proposition 4]{duc21} as follows (see the proof in Appendix \ref{proofs}). 

\begin{proposition}\label{solestdiff}
Consider the Doss-Sussmann transformations  $y_t = \phi_t(\bx,z_t)$ and  $\bar{y}_t = \phi_t(\bx,\bar{z}_t)$. Assign
		\begin{equation*}
			\eta_t = y_t - z_t;\quad \bar{\eta}_t = \ty_t - \tz_t;\quad \psi_t =\Big[\frac{\partial \phi }{\partial z}(t,\bx,z_t)\Big]^{-1} - Id;\quad \tpsi_t =\Big[\frac{\partial \phi }{\partial z}(t,\bx,\tz_t)\Big]^{-1} - Id;
		\end{equation*}
	Given $\lambda=C_pC_g \ltn \bx \rtn_{\tp,[a,b]} \leq \frac{1}{8}$, the following estimates hold
		\begin{equation}\label{difference}
			\|\bar{\eta}_t- \eta_t\| \leq 4\lambda  \|\tz_t-z_t\|; \quad
			\|\tpsi_t - \psi_t\| \leq  256\lambda \|\tz_t-z_t\|,\quad \forall t\in [a,b]. 
		\end{equation}   
\end{proposition}

    Note that all the above arguments might fail to be applied if we work directly with the local solution for $y_0 \in B(0,\epsilon_0)$. This is because at stopping time interval $[\tau_k,\tau_{k+1}]$ the Doss-Sussmann transformation $y_t = \phi_t(\bx,z_t)$ that starts from $y_{\tau_k} \in B(0,\epsilon_0)$ might exit $B(0,\epsilon_0)$ soon before $\tau_{k+1}$. Hence, in general, it would be rather technical to estimate the exit time of the local solution from $B(0,\epsilon_0)$.

		\section{Stopping time analysis}\label{Sectime}
		\subsection{Stopping times for the continuous time case}
		The construction of a {\it greedy sequence of stopping times} in \cite{cassetal} is now used in many recent results, see e.g. \cite{congduchong17,ducGANSch18,duc20,duc21,riedelScheutzow}. Namely, for any fixed $\gamma \in (0,1)$ the sequence of stopping times $\{\tau_i(\gamma,\bx,[0,\infty))\}_{i \in \N}$ is defined by
		\begin{equation}\label{greedytime}
			\tau_0 = 0,\quad \tau_{i+1}:= \inf\Big\{t>\tau_i:  \ltn \bx \rtn_{\tp, [\tau_i,t]} = \gamma \Big\}.
		\end{equation}
		For a fixed closed interval $I\subset [0,\infty)$, we define another sequence of stopping times $\{\tau^*_i(\gamma,\bx,I)\}_{i \in \N}$  by
		\begin{equation}\label{greedytime*}
			\tau^*_0 = \min{I},\quad \tau^*_{i+1}:= \inf\Big\{t>\tau^*_i:  \ltn \bx \rtn_{\tp, [\tau_i,t]}  = \gamma \Big\} \wedge \max{I}.
		\end{equation}
		Define $N^*(\gamma,\bx,I) :=\sup \{i \in \N: \tau^*_i < \max{I}\}+1$. It is easy to show a rough estimate 
		\begin{equation}\label{Nest}
			N^*(\gamma,\bx ,I) \leq 1 +\frac{1}{\gamma^p}\ltn \bx \rtn^p_{\tp,I}.
		\end{equation}
		In fact, it is proved in \cite{cassetal} that $e^{N^*(\gamma,\bx,I)}$ is integrable for Gaussian rough paths.
		
		Denote $\theta$ the Wiener-type shift in the probability space $\Omega:=\cC_0^{0,\alpha}(\R,T^2_1(\R^m))$  (see \eqref{shift} in Appendix \ref{probset}). Throughout this paper, we will assume that $\theta$ is ergodic. 
        
		The following lemma is reformulated from \cite[Theorem 14]{ducGANSch18} for the stopping times defined in the $p$-variation norm. To make the presentation self-contained, we are going to give a short and direct proof here.
		\begin{lemma}\label{infimumD}
			Given the greedy sequence of stopping times \eqref{greedytime}, the followings estimates
			\begin{equation}\label{infD}
				\liminf \limits_{n \to \infty} \frac{\tau_n}{n} \geq \frac{1}{\E N^*(\gamma,\bx ,[0,1])}
			\end{equation}
			holds almost surely.
		\end{lemma}
		\begin{proof}
			For each $j$ denote by $N(\gamma,\bx, [a,b))$ the number of stopping times $\tau_k =\tau_k(\gamma,\bx,[0,\infty))$ in $[a, b)$. Since the minimal stopping time $\tau$ in $I$ is bigger than or equal $\min{I} = \tau^*_0(\gamma,\bx,I)$ and the maximal stopping time $\tau$ in $I$ is less than or equal $\max{I} = \tau^*_{N^*(\gamma,\bx,I)}(\gamma,\bx,I)$, it follows that
			\begin{equation}\label{Ncompare} 
				N^*(\gamma,\bx,I) \geq N(\gamma,\bx,[\min I, \max I)).
			\end{equation}
			
			Denote by $\lfloor{\tau_k}\rfloor$ the integer part of $\tau_k$, then $\tau_k < \lfloor{\tau_k}\rfloor +1$. As a result, the number of positive stopping times in the interval $[0,\tau_k]$ (which is $k$) is less than or equal to the one in the interval $[0,\lfloor{\tau_k}\rfloor +1)$. As a result,
			\[
			\frac{\tau_k}{k} \geq \frac{\lfloor{\tau_k}\rfloor}{N\left(\gamma,\bx,[0,\lfloor{\tau_k}\rfloor +1)\right)}.
			\]
			On the other hand, by definition of $N$ and $N^*$ and inequality \eqref{Ncompare},
			\begin{eqnarray*}
				N(\gamma,\bx,[0,\lfloor{\tau_k}\rfloor +1)) &\leq& \sum_{j=0}^{\lfloor{\tau_k}\rfloor } N(\gamma,\bx,[j,j+1)) \\
				&\leq& \sum_{j=0}^{\lfloor{\tau_k}\rfloor } N^*(\gamma,\bx,[j,j+1]) 
				= \sum_{j=0}^{\lfloor{\tau_k}\rfloor } N^*(\gamma,\bx(\theta_j \omega),[0,1]),
			\end{eqnarray*}
			where the last equality is due to \eqref{roughshift}. Hence,
			\begin{equation}\label{infinequality}
				\frac{\tau_k}{k}  \geq \frac{\lfloor{\tau_k}\rfloor}{\sum_{j=0}^{\lfloor{\tau_k}\rfloor } N^*(\gamma,\bx(\theta_j \omega),[0,1])} \geq \frac{\frac{\lfloor{\tau_k}\rfloor}{\lfloor{\tau_k}\rfloor +1}}{\frac{1}{\lfloor{\tau_k}\rfloor+1}\sum_{j=0}^{\lfloor{\tau_k}\rfloor } N^*(\gamma,\bx(\theta_j \omega),[0,1])}.
			\end{equation}
			By applying the Birkhoff ergodic theorem to the last right hand side of \eqref{infinequality}, where $\tau_k \to \infty$ as $k \to \infty$, the numerator tends to one while the denominator converges to $\E N^*(\gamma,\bx(\cdot),[0,1])$. This proves \eqref{infD}.
		\end{proof}	
		
		\begin{lemma}\label{infN}
			For every $a, b \in \R, a\leq b$, the following estimate holds 
			\begin{equation}\label{infN2a}
				\gamma  N^*(\gamma,\bx(\cdot),[a,b]) \geq  \ltn \bx(\cdot) \rtn_{\tp,[a,b]},\quad \forall \gamma >0.
			\end{equation} 
			As a consequence,
			\begin{equation}\label{infN2}
				\gamma \E N^*(\gamma,\bx(\cdot),[a,b]) \geq \E \ltn \bx(\cdot) \rtn_{\tp,[a,b]},\quad \forall \gamma >0.
			\end{equation}
		\end{lemma}	
		
		\begin{proof}
			Indeed, it follows from \cite[Lemma 2.1]{congduchong17} that
			\begin{eqnarray*}
				\ltn \bx \rtn^p_{\tp,[a,b]} &=& \ltn \bx \rtn^p_{\tp,[\tau^*_0,\tau^*_{N^*(\gamma,\bx,[a,b])}]}\\
				&\leq& \big[N^*(\gamma,\bx,[a,b])\big]^{p-1} \sum_{j = 1}^{N^*(\gamma,\bx,[a,b])}  \ltn \bx \rtn^p_{\tp,[\tau^*_{j-1},\tau^*_j]} \\
				&\leq& \big[N^*(\gamma,\bx,[a,b])\big]^{p-1} \big[N^*(\gamma,\bx,[a,b])\big] \gamma^p\\
				&\leq& \big[N^*(\gamma,\bx,[a,b])\big]^p \gamma^p
			\end{eqnarray*}
			or equivalently
			\[
			\gamma N^*(\gamma,\bx,[a,b]) \geq 	\ltn \bx \rtn_{\tp,[a,b]}.
			\]
			Taking the expectation to both sides of the above inequality, we obtain \eqref{infN2}.
		\end{proof}	
		
		\begin{lemma}\label{Nsumin}
			The following estimate holds for any $n\geq 1$ and any sequence $t_0 < t_1 < \ldots < t_n$
			\begin{equation}\label{Nsum}
				\sum_{j = 0}^{n-1} N^*(\gamma,\bx,[t_j,t_{j+1}]) \leq N^*(\gamma,\bx,[t_0,t_{n}]) + n.			
			\end{equation}
		\end{lemma}
		\begin{proof}
			By definition, on $[t_j,t_{j+1}]$ there are at least $N^*(\gamma,\bx,[t_j,t_{j+1}])-1$ consecutive stopping time intervals on which the $p$-variation norm of $\bx$ is $\gamma$, each of which contains at least one stopping time in the sequence constructed on $[t_0,t_n]$. As a result, on the interval $[t_0,t_n]$ there are at least $\sum_{j = 0}^{n-1} \Big(N^*(\gamma,\bx,[t_j,t_{j+1}]) -1\Big)$  disjoint stopping time intervals on which the $p$-variation norm of $\bx$ is $\gamma$. This proves \eqref{Nsum}. \\	
		\end{proof} 		
        {\bf Stopping times for a control sequence.}
		 Note that one can replace $\ltn \bx\rtn_{\tp}$ in \eqref{greedytime}, \eqref{greedytime*} by a finite sequence of controls $\bw_{\cdot,\cdot}\in \mathcal{S}$ associated with parameters $\beta_\bw \in (0,1]$, to construct the sequences $\tau_k(\gamma,\cS,[0,\infty))$ on $[0,\infty)$ and $ \tau^{*}_k(\gamma,\cS,I)$ on a given closed interval $I\subset[0,\infty)$ in a similar manner. Denote by $N^*(\gamma,\mathcal{S} ,I)$ the number of $ \tau^{*}_k(\gamma,\cS,I)$ on $I$. Similar to \eqref{Nest} we obtain (see for instance \cite[Lemma 2.6]{congduchong17})
		\begin{equation}\label{N*}
			N^*(\gamma,\mathcal{S},I) < 1 + \frac{1}{\gamma^{\frac{1}{\beta}}} |S|^{\frac{1}{\beta}-1} \sum_{\bw\in \mathcal{S}}(\bw_{\min I,\max I})^{\frac{\beta_\bw}{\beta}}
		\end{equation}
    in which $\beta= \min_{\bw\in \mathcal{S}}{\beta_\bw}$.
		Furthermore, similar to Lemma \ref{infimumD}, we have
		\begin{equation}\label{infDS}
			\liminf \limits_{n \to \infty} \frac{\tau_n(\gamma,\cS,[0,\infty))}{n} \geq \frac{1}{\E N^*(\gamma,\cS ,[0,1])}
		\end{equation}
		almost surely.
        
		\subsection{Stopping times for discrete time sets}\label{dis_greedy}
		For investigation of discrete time systems arising from Euler scheme applied to the rough equation \eqref{RDE1} we modify the stopping time technique to the discrete framework. In the discrete time setting we will use more complicated control than the control function $\ltn x \rtn^p_{\tp,[s,t]}$ used in the continuous-time setting above. More precisely, 
		given a finite sequence of controls $\bw_{\cdot,\cdot}\in \mathcal{S}$ associated with parameters $\beta_\bw \in (0,1]$, we would like to construct a version of greedy sequence of (discrete-time) stopping times similar to that in \cite{congduchong23}.
			
	For given $\Delta>0$, consider the sequence  $t_k =k\Delta, \ k\in \N$. Given a fixed $\gamma >0$, assign the starting time $\tau^\Delta_0 (\gamma,\cS,[0,\infty))= 0$. For each $n\in \N$, assume $\tau^\Delta_n(\gamma,\cS,[0,\infty)) =t_k$ is determined. Then one can define $\tau^\Delta_{n+1}(\gamma,\cS)$ by the following rule:
		\begin{itemize}
			\item if $\displaystyle\sum_{\bw \in \mathcal{S}} \bw^{\beta_\bw}_{t_k,t_{k+1}} > \gamma$ then set $\tau^\Delta_{n+1}(\gamma,\cS,[0,\infty)) := t_{k+1}$;
			\item else set $\tau^\Delta_{n+1}(\gamma,\cS,[0,\infty)):= \sup \{t_l >t_k: \sum_{\bw \in \mathcal{S}} \bw^{\beta_\bw}_{t_k,t_l} \leq \gamma \}$. 	
		\end{itemize}

        The following result show the relation between discrete and continuous stopping times.
		\begin{lemma}
        For given $\gamma>0$ and set of controls $\cS$,
			\begin{equation}
				\liminf \limits_{n \to \infty} \frac{\tau^\Delta_n(\gamma,\cS,[0,\infty))}{n} \geq \liminf \limits_{n \to \infty} \frac{1}{2}\frac{\tau_n(\gamma,\cS,[0,\infty))}{n}. \label{disinfD}
			\end{equation}
		\end{lemma}
		\begin{proof}
			Observe that, between two consecutive stopping times $\tau_i(\gamma,\cS,[0,\infty)),\tau_{i+1}(\gamma,\cS,[0,\infty))$ there are at most two stopping times $\tau^\Delta_n(\gamma,\cS,[0,\infty))$. As a result, 	
				 \[
             \tau^\Delta_{2m}(\gamma,\cS,[0,\infty)), \tau^\Delta_{2m+1} (\gamma,\cS,[0,\infty))\geq \tau_m(\gamma,\cS,[0,\infty)). 
             \]
             Therefore,
			\begin{equation*}
				\liminf \limits_{n \to \infty} \frac{\tau^\Delta_n(\gamma,\cS,[0,\infty))}{n} \geq \liminf \limits_{n \to \infty} \frac{1}{2}\frac{\tau_n(\gamma,\cS,[0,\infty))}{n}. 
			\end{equation*}
		\end{proof}

Similar to the continuous case, we construct sequence $\hat{\tau}_n:=\tau^\Delta_n(\gamma, \mathcal{S}, I)$ on a given closed interval $I$. The number of $\tau^\Delta_n(\gamma, \mathcal{S}, I)$ denoted by $\hat{N}=\hat{N}(\gamma, \mathcal{S}, I)$ can be estimated by (see Subsection 2.4 in \cite{congduchong23})
     \begin{equation}\label{N_hat}
		\hat{N} <  2 + \frac{2}{\gamma^{\frac{1}{\beta}}} |S|^{\frac{1}{\beta}-1} \sum_{\bw\in \mathcal{S}}(\bw_{\min I,\max I})^{\frac{\beta_\bw}{\beta}}
		\end{equation}
    in which $\beta= \min_{\bw\in \mathcal{S}}{\beta_\bw}$.
             
We emphasize here that Lemma \ref{infimumD} holds only almost surely w.r.t. $\omega \in \Omega$ under the assumption that the Wiener-shift $\theta$ in Subsection \ref{probset} is ergodic. From now on, we will only work with a realization $x \in C^\nu$ of the stochastic process $X_t$ satisfying assumption (${\textbf H}_X$), such that $x$ can be lifted into a rough path $\bx$. 
		For a little abuse of notation, we only mention the dependence of $\bx$ in the proof, without addressing that $\bx = \bx(\omega) = (x(\omega),\X(\omega))$ for almost surely $\omega \in \Omega$.
	
		\section{Lyapunov stability of rough differential equations}\label{contRDE}
Given the probabilistic setting in Appendix \ref{probset}, the RDE \eqref{RDE1} generates a random dynamical system $\varphi$ over the Wiener space $(\Omega,{\cal F},{\mathbb P})$ and the Wiener-type shift $\theta_t$, thus we may use the tools and results in the theory of random dynamical systems to get the similar ones for the RDE \eqref{RDE1}. For a (possibly random) point $y^*_0\in \R^d$ the solution of \eqref{RDE1} starting from $y^*_0$ can be represented as $\varphi(t,\omega)y^*_0$. For the details on RDS setting for \eqref{RDE1} see \cite{BRSch17} and \cite{duc21}. 

In \cite{duc21} we prove that there exists a global random pullback attractor for the RDS \eqref{RDE1} provided the drift $f$ is globally dissipative. 
If in addition the drift is globally strictly dissipative, then the random attractor is a singleton. In this section we would like to generalize these results to a local situation and for $C_g$, which is then automatically satisfied in the global picture. First we need to recall the concept of dissipativity from the global and local sense.

\begin{definition}[Dissipativity]
A function $f$ is called:
\begin{itemize}
    \item globally dissipative, if there exist constants $D_1,D_2 >0$ such that
    \begin{equation}\label{globaldiss}
        \langle y, f(y) \rangle \leq  D_1- D_2\|y\|^2,\quad \forall y \in \R^d;
    \end{equation}
   \item globally strictly dissipative, if there exists a constant $D >0$ such that
    \begin{equation}\label{strictdissipative}
			\langle y_1 - y_2, f(y_1) - f(y_2) \rangle \leq  - D \|y_1-y_2\|^2,\quad \forall y_1, y_2 \in \R^d;
		\end{equation}
    \item locally strictly dissipative at $\hat{y} \in \R^d$, if there exist a radius $r>0$ and a constant $D >0$ such that    
    \begin{equation}\label{localstrictdissipative}
			\langle y - \hat{y}, f(y) - f(\hat{y}) \rangle \leq  - D \|y-\hat{y}\|^2,\quad \forall \|y-\hat{y}\| <r.
		\end{equation}
\end{itemize}
\end{definition}

In this paper we are interested in the stability of a stationary solution of \eqref{RDE1} which stays inside our global random attractor, thus we give here the definition of stationary solution.

\begin{definition}[Stationary solution]
A random point
$\ta(\cdot):\Omega \to \R^d$ is called a {\em stationary solution} of \eqref{RDE1} if $\ta(\theta_t\omega), t\in\R_+,$ is the solution of \eqref{RDE1} starting from $\ta(\omega)\in\R^d$ at time 0, i.e. $\varphi(t,\omega)\ta(\omega) = \ta(\theta_t \omega)$, for all $t\in \R_+, \omega \in \Omega$.
\end{definition}

 Since in this section we investigate the stability of nontrivial solutions it is natural to impose global smooth and growth conditions (${\textbf H}_f$), (${\textbf H}_g$) on the functions $f,g$. We may of course impose smooth and growth conditions on the functions $f$ and $g$ only in a neighborhood of the stationary solution $\ta$ as well. But notice that $\ta$ is a random point hence its neighborhood is a random open set whereas the functions $f, g$ are defined as nonrandom functions. In the special case that $\ta(\omega)$ is essentially bounded, $\sup\|\ta(\omega)\|<R$ for some non-random $R>0$, we may then impose smooth and growth condition on $f$ and $g$ only on the ball of radius $R$ in $\R^d$.
 
We give the definition  on (Lyapunov) asymptotic/exponential stability of a solution (cf. \cite{chicone} and \cite{perko} for the ODE case, and \cite[Definition 8]{GABSch18} for the rough differential equation case).
\begin{definition}\label{Defstability}
			 (A) Stability: a solution $y^*$ of \eqref{RDE1} is called  (Lyapunov) {\em stable} if for any $\varepsilon >0$ there exists a positive random variable $r =r(\omega)>0$ such that for any initial value $y_0\in \R^d$ satisfying $\|y_0-y^*_0\| < r(\omega)$ the solution $y_t$ of  \eqref{RDE1} starting from $y_0$ exists on the whole half line $t\in [0,\infty)$ and  the following inequality holds 
			\[
			\sup_{t\geq 0}\|y_t-y^*_t\| < \varepsilon.
			\]
			(B) Attractivity:  a solution $y^*$ is called attractive, if  there exists a positive random variable $r(\omega) >0$ such that any solution $y_t$ of  \eqref{RDE1} with  $\|y_0-y^*_0\| < r(\omega)$ exists on the whole half line $t\in [0,\infty)$ and satisfies
			\begin{equation}\label{attractivity} 
				\lim \limits_{t \to \infty} \|y_t-y^*_t\| = 0.   
			\end{equation}
			(C) Asymptotic stability: a  solution $y^*$ of equation \eqref{RDE1} is called asymptotically stable, if it is stable and attractive.\\
			(D) Exponential stability: The stationary solution $y^*$ of equation \eqref{RDE1} is called exponentially  stable, if it is stable and there exist two positive random variables $r(\omega)>0$ and $\alpha(\omega)>0$ and a positive constant $\mu>0$  such that for any initial value $y_0$ satisfying $\|y_0-y^*_0\| < r(\omega)$ the solution of equation \eqref{RDE1} starting from $y_0$ exists on the whole half line $0\leq t<\infty$ and the following inequality
			\[
			\|y_t-y^*_t\| \leq \alpha(\omega) e^{-\mu t}
			\]
			holds for all $t\geq 0$.
		\end{definition}
		
		It is easily seen that, like the case of ordinary differential equations, exponential stability implies asymptotic stability, and the asymptotic stability implies stability; but the inverse direction is not true.

Because stability is a local property, we would like to find a  local condition for the Lyapunov exponential stability of an arbitrary stationary solution of the RDE \eqref{RDE1}. A closer look at \eqref{strictdissipative} motivates the following definition on the so-called {\it locally one-sided Lipschitz constant}: assign for each fixed $\hat{y} \in \R^d$ the quantities
\begin{eqnarray}\label{dfy}
    \ell(f,\hat{y},r) &:=& \sup \limits_{0<\|y-\hat{y}\| < r} \langle \frac{y-\hat{y}}{\|y-\hat{y}\|},\frac{f(y)-f(\hat{y})}{\|y-\hat{y}\|}\rangle; \notag\\
    \ell(f,\hat{y}) &:=& \lim \limits_{r \to 0} \ell(f,\hat{y},r).
\end{eqnarray}
Notice that $\ell(f,\hat{y},r)$ is well-defined and finite due to the locally Lipschitz continuity of $f$ at $\hat{y}$. In addition, $\ell(f,\hat{y},r)$ is a non-decreasing function of $r$, thus the limit in the definition of $\ell(f,\hat{y})$ exists, i.e. $\ell(f,\hat{y})$ is well-defined and finite for every $\hat{y} \in \R^d$. In fact we can derive an explicit computation for $\ell(f,\hat{y})$ as follows.
\begin{proposition}\label{propl}
 Under the assumption $f \in C^1$,
    \begin{equation}\label{dfyhat}
        \ell(f,\hat{y}) = \sup \limits_{\|h\| =1} \langle h, Df(\hat{y})h \rangle,\quad \forall \hat{y} \in \R^d.
    \end{equation}
    As a result, $\ell(f,\hat{y})$ is continuous at $\hat{y}$.
\end{proposition}
\begin{proof}
    The proof follows directly from the fact that for $0<\|y-\hat{y}\|< r$,
    \begin{eqnarray}\label{Dfdiff}
       &&\Big| \langle \frac{y-\hat{y}}{\|y-\hat{y}\|},\frac{f(y)-f(\hat{y})}{\|y-\hat{y}\|}\rangle - \langle \frac{y-\hat{y}}{\|y-\hat{y}\|}, Df(\hat{y})\frac{(y-\hat{y})}{\|y-\hat{y}\|} \rangle \Big| \notag\\
       &\leq& \int_0^1 \|Df(\hat{y}+ \mu(y-\hat{y})) -Df(\hat{y})\|d\mu \notag\\
       &\leq& \sup \limits_{\|h\|\leq r} \|Df(\hat{y}+h)-Df(\hat{y})\| =:osc(Df(\hat{y}))_r,
    \end{eqnarray}
    where it can be proved that $osc(Df(\hat{y}))_r$ in the right hand side of \eqref{Dfdiff} is a continuous function of $\hat{y}$ and $r$, and that $osc(Df(\hat{y}))_r$ tends to zero as $r \to 0$. As a result,
    \begin{equation}\label{Dfdiff2}
    |\ell(f,\hat{y},r) - \sup \limits_{\|h\| =1} \langle h, Df(\hat{y})h \rangle| \leq osc(Df(\hat{y}))_r.   
    \end{equation}
    Moreover,
    \[
    \sup \limits_{\|h\| =1} \langle h, Df(\hat{y})h \rangle = \sup \limits_{\|h\| =1} \langle h, \frac{1}{2}\Big[Df(\hat{y})+Df(\hat{y})^{\rm{T}}\Big] h \rangle
    \]
    which is the maximal eigenvalue of $\frac{1}{2}\Big[Df(\hat{y})+Df(\hat{y})^{\rm{T}}\Big]$. This prove the continuity of $\ell(f,\hat{y})$ at $\hat{y}$.
\end{proof}

A direct consequence of \eqref{Dfdiff2} is that for any $\|y-\hat{y}\| < r$,
\begin{equation}\label{Mf}
\ell(f,\hat{y},r)  \leq \sup \limits_{\|h\| =1} \langle h, Df(\hat{y})h \rangle + osc(Df(\hat{y}))_r =: M(Df,\hat{y},r)
\end{equation}
so that
\begin{equation}\label{fdiff}
    \langle y - \hat{y}, f(y)-f(\hat{y}) \rangle \leq M(Df,\hat{y},r) \|y-\hat{y}\|^2,\quad \forall \|y-\hat{y}\| < r. 
\end{equation}
Since $osc(Df(\hat{y}))_r$ is continuous in $\hat{y}$ and $r$, so is $M(Df,\hat{y},r)$. Moreover, the definition of $M$ in \eqref{Mf} yields
\begin{equation}\label{Mdf}
   | M(Df,\hat{y},r) |\leq 3 \|Df\|_{\infty,B(y,r)}.
\end{equation}
Because of \eqref{localstrictdissipative}, we are interested in the points $\hat{y}$ that makes $M(Df,\hat{y},r)$ negative for certain $r>0$, i.e. $f$ is locally strictly dissipative at $\hat{y}$. However for a random variable $\ta(\omega)$, it might happen that $M(Df,\ta(\omega),r) >0$ for a positive probability.  

Motivated by Proposition \ref{propl}, we shall assume an additional condition for $f$ that\\

(${\textbf H}^+_{f}$) $f\in C^1$ and there exist constants $C_{f}>0,\rho \geq 1$ such that
    \begin{equation}\label{f_power_rate}
   \|f(y)\|, \|Df(y)\| \leq C_{f}(1+\|y\|^\rho), \quad \forall y \in \R^d.
    \end{equation}
Our first main result of this paper derives a local stability criterion for a stationary solution $\ta(\omega)$. First we introduce the function
      \begin{equation}\label{funkappa}
	\kappa(\lambda,y,r):=M(Df,y,r) +256\lambda\|f(y)\| +64 \lambda\|Df\|_{\infty,B(y,r)},
    \end{equation}
for all $y\in\R^d$ and $\lambda, r \in \R_+$. 
	Since $f \in C^1$, it is easy to check that $\kappa$ is continuous in $y, \lambda,r$ with
    \[
      \kappa(\lambda,y,0) = \ell (f,y)+  256\lambda\|f(y)\| +  64 \lambda\|Df(y)\|,\quad \forall y\in\R^d, \lambda \in \R_+.
    \]     
\begin{theorem} \label{local_diss_attractor}
	Assume (${\textbf H}_f$), (${\textbf H}^+_f$), (${\textbf H}_g$),(${\textbf H}_X$). Assume further that $\ta(\omega)$ is a stationary solution of the RDE \eqref{RDE1} with $\|\ta(\cdot)\|^\rho \in L^1$, where $\rho$ is defined in \eqref{f_power_rate}, and there exists a $\lambda \in (0,\frac{1}{8})$ such that  
    \begin{equation}\label{Anegdef}
      -\E \Big[ \ell(f,\ta) + 256\lambda\|f(\ta)\| + 64 \lambda \|Df(\ta)\|  \Big]  >4\lambda \E N^*(\frac{\lambda}{C_pC_g},\bx(\cdot),[0,1]).
    \end{equation} 
Then the solution $\ta(\theta_t\omega)$, for $t \in \R_+$, of \eqref{RDE1} is almost surely exponentially stable.
\end{theorem}
		\begin{proof}
 First, observe that by definition of $\kappa$ given in \eqref{funkappa}, \eqref{Anegdef} means
\begin{equation}\label{Anegdef_new}
       -\E \kappa(\lambda,\ta(\cdot),0)  >4\lambda \E N^*(\frac{\lambda}{C_pC_g},\bx(\cdot),[0,1]).
    \end{equation} 
Write in short $\ta_t := \ta(\theta_t \omega) = \varphi(t,\omega)\ta(\omega)$ to be the stationary solution of \eqref{RDE1}. First observe that $osc(Df(\ta(\omega)))_r \to 0$ as $r \to 0$ almost surely, and that 
\[
osc(Df(\ta(\cdot)))_r \leq 2 \|Df(\ta(\cdot)+\cdot)\|_{\infty,B(0,r)} \leq 2^{1+\rho} C_{f}\Big(1+\|\ta(\cdot)\|^\rho +r^\rho\Big)\in L^1.
\]
 Thus by the dominated convergence theorem $\lim \limits_{r\to 0} \E \Big[osc(Df(\ta(\cdot)))_r\Big] =0$. As a result,  
 \[
 \lim \limits_{r\to 0} \E M(Df,\ta(\cdot),r) = \E \ell(f,\ta(\cdot)).
 \] 
 Similarly, it follows from \eqref{f_power_rate}, $\|\ta(\cdot)\|^\rho\in L^1$ and the dominated convergence theorem that 
 \[
 \|f(\ta(\cdot))\|, \|Df\|_{B(\ta(\cdot),r)} \in L^1 \quad \text{and} \quad \lim \limits_{r \to 0}\E\|Df\|_{B(\ta(\cdot),r)} = \E\|Df(\ta(\cdot))\|. 
 \]
Hence $\lim \limits_{r \to 0}\E\kappa(\lambda, \ta(\cdot),r) = \E\kappa(\lambda, \ta(\cdot),0)$. Due to \eqref{Anegdef}, there exists $r_0$ small enough such that
 \begin{equation}\label{pointattractorcriterion}
	-\E\kappa(\lambda, \ta(\cdot),r) >4 \lambda \E N^*(\frac{\lambda}{C_pC_g},\bx(\cdot),[0,1]),\quad \forall r \leq r_0.     
 \end{equation}    
	We fix an arbitrary radius $r \leq r_0$ from now on.
    For given $\tau >0$ as in Section \ref{existunique} such that $C_pC_g \ltn \bx \rtn_{\tp,[0,\tau]} = \lambda \leq \frac{1}{8}$, let $y_0\in\R^d$ be close to $\ta(\omega)$ such that $\|y_0-\ta(\omega)\| \leq r $ and consider the solution $y_t$ of \eqref{RDE1} starting from $y_0$. Let  $z_t$ be a
 the solution of the associate ODE \eqref{ascoODE} starting at $y_0$
 \begin{equation*}
		\dot{z}_t = \Big[\frac{\partial \phi }{\partial z}(t,\bx(\omega),z_t)\Big]^{-1} f( \phi_t(\bx(\omega),z_t)),\quad t \in [0,\tau],\ z_0 = y_0,
	\end{equation*}
    and similarly let $\tz_t$ be the solution of this ODE starting at $\tz_0=\ta(\omega)$ where $\ta_t = \phi(t,\bx(\omega),\bar{z}_t) = \bar{z}_t+ \bar{\eta}_t$.
Consider the difference $\gamma_t := z_t- \bar{z}_t$ and
 \begin{equation*}
            \eta_t = y_t - z_t;\  \bar{\eta}_t = \ta_t - \tz_t;\  \psi_t =\Big[\frac{\partial \phi }{\partial z}(t,\bx(\omega),z_t)\Big]^{-1} - Id;\  \tpsi_t =\Big[\frac{\partial \phi }{\partial z}(t,\bx(\omega),\tz_t)\Big]^{-1} - Id.
		\end{equation*}
According to Proposition \ref{solestdiff},
\begin{eqnarray*}			&&\|\bar{\eta}_t- \eta_t\| \leq 4\lambda  \|\tz_t-z_t\| \leq 4 \lambda \|\gamma_t\|; \quad
			\|\tpsi_t - \psi_t\| \leq  256\lambda \|\tz_t-z_t\| \leq  256\lambda\|\gamma_t\|; \notag\\
            &&\|y_t-\ta_t\| \leq \|\tz_t-z_t\| + \|\bar{\eta}_t- \eta_t\| \leq (1+4\lambda)\|\gamma_t\|;\notag\\
            &&\|y_t-\ta_t\| \geq \|\tz_t-z_t\| - \|\bar{\eta}_t- \eta_t\| \geq (1-4\lambda)\|\gamma_t\|;\quad \forall t\in [0,\tau]. 
		\end{eqnarray*}
Then $\gamma$ evolves with its equation
			\begin{eqnarray*}
				\dot{\gamma}_t &=&  \dot{z_t}-\dot{\bar{z}}_t \\
                &=& (Id + \psi_t) f(y_t)-(Id + \tpsi_t) f(\ta_t) \\
                &=& (Id + \psi_t) f(z_t + \eta_t)-(Id + \tpsi_t) f(\tz_t + \bar{\eta}_t) \\
				&=& [f(y_t)-f(\ta_t)] + \psi_t [f(y_t)-f(\ta_t)] +  (\psi_t-\tpsi_t)f(\ta_t).
			\end{eqnarray*}
It follows from Lagrange's mean value theorem and \eqref{fdiff} that for all $t \in [0,\tau]$ and $\|y_t -\ta_t\| <r$,
\begin{eqnarray*}
    \frac{d}{dt} \|\gamma_t\|^2 
    &=& 2 \langle z_t-\bar{z}_t, 
  \dot{z}_t -  \dot{\bar{z}}_t \rangle \\
	&\leq& 2 \Big\langle  y_t -\ta_t+(\bar\eta_t-\eta_t), f(y_t)-f(\ta_t) \Big\rangle +2 \|\gamma_t\|\|\psi_t-\tpsi_t\| \|f(\ta_t)\| \\
		&&+ 2 \| \gamma_t\| \|\psi_t\| \|f(y_t)-f(\ta_t)\| \\
    &\leq& 2 M(Df,\ta_t,r) \| y_t-\ta_t\|^2 +  2 \|\gamma_t\|\| \psi_t-\tpsi_t\| \|f(\ta_t)\| \\
    &&+2 \Big(\|\eta_t-\bar{\eta}_t \|  +   \|\gamma_t\|\|\psi_t\| \Big) \|Df\|_{\infty,B(\ta_t,r)} \|y_t - \ta_t\|.
\end{eqnarray*}
One can estimate the first term on the right hand side of the last inequality as below
\begin{eqnarray*}
    M(Df,\ta_t,r) \| y_t-\ta_t\|^2 &\leq& \begin{cases}
        M(Df,\ta_t,r) \|\gamma_t\|^2 (1- 4\lambda )^2 & \text{if}\quad M(Df,\ta_t,r) \leq 0 \\
        M(Df,\ta_t,r) \|\gamma_t\|^2 (1+ 4\lambda )^2 & \text{if} \quad M(Df,\ta_t,r) \geq 0 
    \end{cases}\\
    &\leq& \Big[ M(Df,\ta_t,r)  (1+ 16\lambda^2) + 8 \lambda | M(Df,\ta_t,r)|\Big] \|\gamma_t\|^2.
\end{eqnarray*}
As a result, since $\lambda <1/8$, we have
\begin{eqnarray*}
    \frac{d}{dt} \|\gamma_t\|^2 
&\leq& 2\Big[ M(Df,\ta_t,r)  (1+ 16\lambda^2) + 8 \lambda | M(Df,\ta_t,r)|\Big] \|\gamma_t\|^2 +2\|f(\ta_t)\|256\lambda   \|\gamma_t\|^2\\
    &&+ 2(4\lambda +4\lambda)(1+4\lambda)  \|Df\|_{\infty,B(\ta_t,r)} \|\gamma_t\|^2 \\
	&\leq& 2\kappa(\lambda,\ta_t,r)\|\gamma_t\|^2,
\end{eqnarray*}
where we use \eqref{solest3} that $\|\psi_t\| \leq 4\lambda$. Hence
    \[
    \|\gamma_t\| \leq \|\gamma_0\| \exp\Big\{ \int_0^t \kappa(\lambda,\ta_s,r)ds \Big\}, 
    \]
    which yields
			\begin{align}\label{ytau}
				\|y_t-\ta_t\| &\leq \big(1+ 4\lambda\big)\exp\Big\{ \int_0^t \kappa(\lambda,\ta_s,r)ds \Big\}\|y_0-\ta(\omega) \|\notag\\
				&\leq\exp\Big\{4 \lambda +\int_0^t \kappa(\lambda,\ta_s,r)ds \Big\} \| y_0-\ta(\omega)\|,\quad \forall t \in [0,\tau].
			\end{align}
			In other words, in order for \eqref{ytau} to hold for all $t\in [0,\tau]$, it is important that
			\begin{equation}\label{localstab} 
			\|y_0 - \ta(\omega)\| \leq \inf_{t \in [0,\tau]} r\exp\Big\{ -4\lambda -\int_0^t \kappa(\lambda,\ta_s,r)ds \Big\},	
            \end{equation}
			where the right hand side of \eqref{localstab} is positive due to the fact that the function under taking infimum is also continuous and positive definite in time. Now constructing the sequence of stopping times $\{\tau_i(\frac{\lambda}{C_pC_g},\bx(\omega),[0,\infty))\}$, we can prove by induction that 
			\begin{align}\label{Delta}
				\|y_t-\ta_t\| &\leq \big(1+ 4\lambda\big)\exp\Big\{ \int_{\tau_n}^t \kappa(\lambda,\ta_s,r)ds \Big\}\|y_{\tau_n}-\ta_{\tau_n} \|\notag\\
				&\leq\exp\Big\{ 4\lambda +\int_{\tau_n}^t \kappa(\lambda,\ta_s,r)ds \Big\} \|y_{\tau_n}-\ta_{\tau_n} \|,\quad \forall t \in [\tau_n,\tau_{n+1}].
			\end{align}
			provided that
			\begin{equation}\label{localstab2} 
				\|y_{\tau_n}-\ta_{\tau_n} \| \leq \inf_{t \in  [\tau_n,\tau_{n+1}]} r\exp\Big\{ -4\lambda -\int_{\tau_n}^t \kappa(\lambda,\ta_s,r)ds\Big\}.
			\end{equation}
	By induction,
	\begin{equation}\label{attractivity} 
		\|y_{\tau_n}-\ta_{\tau_n} \| \leq \exp\Big\{ 4\lambda n+\int_0^{\tau_n} \kappa(\lambda,\ta_s,r)ds \Big\} \| y_0-\ta(\omega)\|,\quad \forall n \geq 0.
	\end{equation}
	Thus in order for $y_{\tau_n}$ to satisfy \eqref{localstab2}, it suffices to choose $y_0$ such that
	\[
	 \| y_0-\ta(\omega)\| \leq  \exp\Big\{ -4\lambda n -\int_0^{\tau_n} \kappa(\lambda,\ta_s,r)ds \Big\} \inf_{t \in  [\tau_n,\tau_{n+1}]} r\exp\Big\{ -4\lambda -\int_{\tau_n}^t \kappa(\lambda,\ta_s,r)ds\Big\},\quad \forall n \geq 0,
	\]
	or $\| y_0-\ta(\omega)\| \leq R(\omega)$ where
	\begin{eqnarray}\label{Romega} 
		R(\omega) &:=& r\inf_{n \geq 0} \Big[\exp\Big\{ -4\lambda n-\int_0^{\tau_n} \kappa(\lambda,\ta_s,r)ds \Big\} \inf_{t \in  [\tau_n,\tau_{n+1}]} \exp\Big\{ -4\lambda -\int_{\tau_n}^t \kappa(\lambda,\ta_s,r)ds\Big\}\Big]\notag\\
        &=& r \inf_{n \geq 0} \Big[\inf_{t \in  [\tau_n,\tau_{n+1}]} \exp\Big\{ -4\lambda (n+1)-\int_0^t \kappa(\lambda,\ta_s,r)ds\Big\}\Big].
	\end{eqnarray}
 Because of the stationarity, $\kappa(\lambda,\ta_s,r) =\kappa(\lambda,\ta(\theta_s \omega),r)$. Assign
			\[
			-\mu:=  \E \kappa(\lambda,\ta(\cdot),r) + \epsilon+4\lambda \E N^*(\frac{\lambda}{C_pC_g},\bx(\cdot),[0,1]).
			\]
Due to \eqref{pointattractorcriterion}, $\E \kappa(\lambda,\ta(\cdot),r) + \epsilon<-\mu<0$ for sufficiently small $\epsilon >0$. 
With such $\epsilon$, it follows from Birkhorff's ergodic theorem and \eqref{infD} that for $n \geq n(\omega)$ large enough,
  \allowdisplaybreaks
\begin{eqnarray*}
&&    r e^{-4\lambda} \inf_{n \geq n(\omega)} \Big[\inf_{t \in  [\tau_n,\tau_{n+1}]} \exp\Big\{ -4\lambda n + t \Big(\frac{1}{t}\int_0^t -\kappa(\lambda, \ta(\theta_s \omega),r)ds\Big)\Big\}\Big] \\
&\geq& r e^{-4\lambda} \inf_{n \geq n(\omega)} \Big[\inf_{t \in  [\tau_n,\tau_{n+1}]} \exp\Big\{ -4\lambda n + t \E\Big( -\kappa(\lambda,\ta(\cdot),r) \Big)- \epsilon \Big\}\Big]\\
&\geq& r e^{-4\lambda} \inf_{n \geq n(\omega)} \exp\Big\{ \tau_n\Big[\E \Big(-\kappa(\lambda,\ta(\cdot),r)\Big) - \epsilon -4\lambda \frac{n}{\tau_n} \Big]\Big\}\\
&\geq& r e^{-4\lambda} \inf_{n \geq n(\omega)} \exp\Big\{ \tau_n\Big[\E\Big( -\kappa(\lambda,\ta(\cdot),r) \Big)- \epsilon -4\lambda \E N^*(\frac{\lambda}{C_pC_g},\bx(\cdot),[0,1]) \Big]\Big\}\\
&\geq& re^{-4\lambda} \inf_{n \geq n(\omega)} \exp\big\{ -\mu\tau_n \big\} \geq re^{-4\lambda}.
\end{eqnarray*}
Thus 
\[
R(\omega) \geq r e^{-4\lambda}\inf_{n \leq n(\omega)} \Big[\inf_{t \in  [\tau_n,\tau_{n+1}]} \exp\Big\{ -4\lambda n -\int_0^t \kappa(\lambda,\ta(\theta_s \omega),r)ds\Big\}\Big]
\]
which is positive. All in all, there exists a random neighborhood $B(\ta(\omega),R(\omega))$ of $a$ such that whenever $\|y_0-\ta(\omega)\| \leq R(\omega)$ then \eqref{localstab2} and \eqref{attractivity} hold for all $n\in \N$, thus $\|y_t-\ta_t\| \leq r$ for all $t \geq 0$ due to \eqref{Delta}. Since the conclusion holds for any fixed $r \leq r_0$, this proves stability. Moreover, it follows from \eqref{Delta} and \eqref{attractivity} that for $t \geq \tau_{n(\omega)}$,
\allowdisplaybreaks
			\begin{eqnarray*}
				\|y_t-a_t\| 
                &\leq& \|y_0-\ta(\omega)\| \exp \Big \{4\lambda (n+1)+ \int_0^t \kappa(\lambda,\ta(\theta_s \omega),r)ds\Big \} \\
                &\leq&  R(\omega)e^{4\lambda} \exp \Big\{\Big[ \frac{1}{t}\int_0^t \kappa(\lambda,\ta(\theta_s \omega),r)ds\Big] t + 4\lambda n \Big\}\\
                &\leq&  R(\omega)e^{4\lambda} \exp \Big\{\Big[\E \kappa(\lambda,\ta(\cdot),r) + \epsilon\Big] t +4 \lambda n \Big\}\\
                 &\leq&  R(\omega)e^{4\lambda} \exp \Big\{\Big[\E \kappa(\lambda,\ta(\cdot),r) + \epsilon\Big] \tau_n + 4\lambda n\Big\} \exp \Big\{\Big[\E \kappa(\lambda,\ta(\cdot),r) + \epsilon\Big] (t-\tau_n) \Big\}\\
                 &\leq&  R(\omega)e^{4\lambda} \exp \Big\{\tau_n\Big[\E \kappa(\lambda,\ta(\cdot),r) + \epsilon+ 4\lambda \frac{n}{\tau_n} \Big] \Big\} \exp \Big\{\Big[\E \kappa(\lambda,\ta(\cdot),r) + \epsilon\Big] (t-\tau_n) \Big\}\\
                 &\leq&  R(\omega)e^{4\lambda} \exp \Big\{\tau_n\Big[\E \kappa(\lambda,\ta(\cdot),r) + \epsilon+4\lambda \E N^*(\frac{\lambda}{C_pC_g},\bx(\cdot),[0,1]) \Big] \Big\} \times \\
                 && \times \exp \Big\{\Big[\E \kappa(\lambda,\ta(\cdot),r) + \epsilon\Big] (t-\tau_n) \Big\}.
			\end{eqnarray*}
		As a result, for all $t\in [\tau_n,\tau_{n+1}]$ and $n\geq n(\omega)$,
			\[
			\|y_t-\ta_t\| \leq R(\omega)e^{4\lambda} e^{\big[\E \kappa(\lambda,\ta(\cdot),r) + \epsilon\big](t-\tau_n) }e^{-\mu \tau_n} \leq R(\omega)e^{4\lambda} e^{-\mu(t-\tau_n) }e^{-\mu \tau_n} \leq 
			R(\omega)e^{4\lambda}  e^{-\mu t}.
			\]
			Since $\|y_t-\ta_t\|\leq r$ for all $t\geq 0 $ whenever $\|y_0-\ta(\omega)\|\leq R(\omega)$, we obtain
			\[
			\|y_t-\ta_t\| \leq \alpha(\omega) \exp (- \mu t) \quad \hbox{for all}\quad t\geq 0,
			\]
			where
			\[
			\alpha(\omega) := \max\Big\{
			r\exp \{\mu \tau_{n(\omega)}\}, R(\omega) e^{4\lambda}\Big\} >0
			\]
			is a positive random variable. This proves the exponential stability of the stationary solution $\ta(\cdot)$ of \eqref{RDE1}.
The proof is complete.
    
	\end{proof}				

\begin{remark}\label{stabcrirem}
\begin{enumerate}
\item It follows from \eqref{infN2} that the necessary condition for the criterion \eqref{Anegdef} to occur is 
\begin{equation}\label{necscri}
-\E \ell (f,\ta(\cdot)) >4 C_p C_g \E \ltn \bx(\cdot)\rtn_{\tp,[0,1]}.    
\end{equation}
Conditions \eqref{Anegdef} and \eqref{necscri} are independent of $g$ in the sense that only $C_g$ is needed, provided that the information about the unperturbed drift $f$ and its derivative $Df$ is given. This is consistent with classical stochastic stability criteria w.r.t. the white noises; see e.g. in \cite{khasminskii} or \cite{duchongcong18} and the references therein. The result in Theorem \ref{local_diss_attractor} is therefore more advantageous in practice, as it often appears that only the information on the trajectories of the particular stationary solution is provided together with the information of the unperturbed autonomous drift $f$, and very little if any can be guessed from the diffusion part $g$. In these scenarios, the expectations on the left-hand sides of \eqref{Anegdef} and \eqref{necscri} are explicitly computable in the pathwise sense as the time averages of functions $\kappa$ and $\ell$ of stationary trajectories $\ta$. In contrast, other approaches might need more information on the linearized system along the stationary solution in an attempt to compute Lyapunov exponents, thereby requiring an unrealistic assumption that the diffusion part is known. 
    
    \item In most cases, we can choose $\lambda$ in \eqref{pointattractorcriterion} to be $C_g \in (0,\frac{1}{8})$, so that the criterion \eqref{Anegdef} or \eqref{Anegdef_new} becomes
    \begin{equation}\label{pointcriterionCg}
-\E\kappa(C_g,\ta(\cdot),0) >4C_g \E N^*(\frac{1}{C_p},\bx(\cdot),[0,1]).        
    \end{equation}
    From \eqref{Anegdef_new}, $-\E\kappa(\lambda,\ta(\cdot),0)$ is a continuous function of $\lambda$, hence if $\E\|\ta(\cdot)\|^\rho = \E\|\ta(C_g,\cdot)\|^\rho$ depends continuously on $C_g$ then the left-hand side of \eqref{pointcriterionCg} is also a continuous function of $C_g$ which is equal to $-\E \ell(f,\ta(0,\cdot),0) >0$ at $C_g =0$, i.e. $Df(\ta(0,\cdot))$ is negative definite in the expectation. In these scenarios, there exists a $C_0 < \frac{1}{8}$ small enough such that \eqref{pointcriterionCg} holds for all $C_g < C_0$. We will apply such a choice $\lambda := C_g$ in Corollary \ref{global_diss_attractor} and in treating the case of trivial solution at the end of this section. 

 \item  It is important to note that the criterion \eqref{Anegdef} can only be checked once the stationary solution $\ta(\omega)$ is given. This solution exists in the case where $f$ is globally dissipative and $g \in C^3_b$ as proved in \cite[Theorem 3.3]{duc21}, because there exists a global random attractor $\tA$ in these scenarios and one can take any stationary solution $\ta\in \tA$ to check the condition \eqref{Anegdef}. We will see from Example \ref{expitchfork} below that not all stationary solutions pass the test \eqref{Anegdef}, thus having a global attractor does not guarantee the local dynamics and stability. 
\end{enumerate}
\end{remark}
As a special case, we show below that if $f$ satisfies the globally strict dissipativity condition, then there exists a unique singleton attractor that is also exponentially stable.
\begin{corollary} \label{global_diss_attractor}
Assume (${\textbf H}_f$),(${\textbf H}_g$),(${\textbf H}_X$) and further that $f$ is globally strictly dissipative in the sense of \eqref{strictdissipative}. Then there exists a positive constant $C_0>0$ such that if  
 $C_g\leq C_0$ then the RDE \eqref{RDE1}  has a stationary solution $\ta(\omega)$ such that\\
    (i)  $\{\ta(\omega)\}$ is a singleton attractor of the RDS $\varphi$, which is both forward and pullback attractor.\\
    (ii) The solution $\ta(\theta_t\omega), t\in\R$ of the RDE \eqref{RDE1} is Lyapunov exponentially stable for arbitrarily large radius $r(\omega)$, for almost all $\omega\in\Omega$.\\
    (iii) $\|\ta(\cdot)\|\in {\cal L}^\rho(\Omega)$ for any $\rho\geq 1$.
		\end{corollary}
		\begin{proof}
Part (i) and (iii) is a direct consequence of \cite[Theorem 3.3, Theorem 3.1]{duc21}. To prove part (ii), observe from \eqref{strictdissipative} and \eqref{dfyhat} that $- \ell(f,y,0)\geq D >0$ for all $y\in \R^d$. We then follows Remark \ref{stabcrirem} (2) and choose $\lambda := C_g < C_0$ for sufficiently small $C_0$. 
       \end{proof} 
    
\begin{example}\label{expitchfork}
	Consider the pitchfork bifurcation system of the form
	\begin{equation}\label{pitchfork} 
		dy = (\alpha y - y^3) dt + \sigma y dB^H,
	\end{equation}
where $B^H$ is a fractional Brownian motion with the Hurst index $H \in (\frac{1}{3},\frac{1}{2}]$. In that scenarios, it is easy to check that 
\begin{equation}\label{pitchforkdissipative}
\begin{split}
\langle y,f(y)\rangle &= y^2 (\alpha - y^2) \leq \alpha^2 - |\alpha| y^2;\\
\langle y - \hat{y},f(y)-f(\hat{y})\rangle &= (y-\hat{y})^2 \Big[\alpha - y^2 -y\hat{y}-(\hat{y})^2\Big];     
\end{split}
\end{equation}
hence the drift $f$ is globally dissipative. As a result, it follows from \cite[Theorem 5.1]{duckloeden} that there exists for $\sigma$ small enough a global attractor for the stochastic system \eqref{pitchfork}. Our interest is to consider the local stability of points in the global attractor. 

Similar to the situation with stochastic Stratonovich noise from \cite[Subsection 9.3.2, pp. 480]{arnold},  one can solve the stochastic differential equation \eqref{pitchfork} explicitly in the pathwise sense for almost all realizations $B^H_t(\omega) \in C^{H-}$ with $\omega \in \Omega$ as a rough differential equation, by using the rough path technique 
to solve the pure rough equation and then applying the Doss-Sussmann transformation. The rough solution then has the form
\[
\varphi(t,\omega)y = y e^{\alpha t + \sigma B^H_t(\omega)} \Big(1+2y^2 \int_0^t e^{2(\alpha s + \sigma B^H_s(\omega))}ds \Big)^{-\frac{1}{2}},\quad \forall t\geq 0.
\]
When $\alpha <0$, $f$ becomes globally strictly dissipative, thus it follows from \cite[Theorem 3.3]{duc21} that the trivial solution is the unique global attractor of system \eqref{pitchfork} and thus is exponentially stable. When $\alpha >0$,  it follows from \eqref{pitchforkdissipative} that $f$ is no longer globally strictly dissipative, but only locally strictly dissipative around the vicinity 
\[
\mathcal{D} := \{\hat{y} \in \R: \alpha < 3(\hat{y})^2 \}.
\]
The origin then changes its stability as $\alpha$ turns positive, and there appear two more stationary solutions which can be computed explicitly as 
\[
\pm c(\omega) = \pm \Big(2\int_{-\infty}^0 \exp \{2\alpha t +2 \sigma B^H_t(\omega)\}dt \Big)^{-\frac{1}{2}}. 
\]
and it is easy to check that $\E c^2 = \alpha$. Indeed, since $Df(y) = \alpha - 3 y^2$ and $f\in C^2$, $Df(\pm c(\omega))$ might go and leave the negative region $\mathcal{D}$ at any time due to its randomness. However, one can check condition \eqref{Anegdef} as
\[
\E \ell(f,\pm c(\cdot)) = \E Df(\pm c(\cdot)) = \alpha -3 \E c(\cdot)^2   = -2 \alpha <0.    
\]
Thus the two stationary solutions $\pm c(\omega)$ become locally exponentially stable by choosing sufficiently small $\sigma$. One can actually proves that the global attractor of system \eqref{pitchfork} is $\tA(\omega) = [-c(\omega),c(\omega)]$. 
\end{example}

\begin{example}\label{FHNex}
For a subtle example, by following \cite[Section 2.2]{ducjost25}, we consider Fitzhugh-Nagumo system perturbed by any bounded diffusion part $g$ satisfying (${\textbf H}_g$) with a fractional Brownian noise 
\begin{equation}\label{eq:dFHN}
		dy = f(y)dt  + g(y)dB^H_t;\quad f(y) = \left(\begin{matrix}
			 v_t-\frac{v_t^3}{3}-w_t+I\\
			\varepsilon (v_t-\mu w_t + J)
		\end{matrix}\right),\quad y = (v,w)^{\rm T}.
	\end{equation}	
For positive parameter $\varepsilon,\mu,I,J >0$, such that there exists a unique fixed point which is also the global attractor $\cA_0 = \{\ta^*=(v^*,w^*)\}$ of the unperturbed system $\dot{u} = f(u)$ (for instance we can follow \cite[p. 513]{ducjostdatmarius} to choose $I = 0.265, \mu = 0.75, J = 0.7, \varepsilon = 0.08$). Moreover, it is easy to check that $\ta^*$ is locally exponentially stable, thus $\ell (f,\ta^*) <0$. As proved in \cite[Theorem 7]{ducjost25}, the Fitzhugh-Nagumo drift admits a so-called strong Lyapunov function of the form
\[
V(y):=\Big(1+ v^4 + \frac{6}{\varepsilon \mu} w^2\Big)^{\frac{1}{4}}.
\]
Hence it follows from \cite[Theorem 14, Theorem 16]{ducjost25} that system \eqref{eq:dFHN} admits a global random pullback attractor $\cA^g(\omega) \in \cL^\rho$ for any $\rho \geq 1$ (which might contain more than a random point), such that 
\begin{equation}\label{FHNattractor}
\lim \limits_{C_g \to 0} d_H(\cA^g(\omega)|\ta^*) =0\quad \text{a.s.}; \quad  \lim \limits_{C_g \to 0} \E d_H(\cA^g(\cdot)|\ta^*) =0.
\end{equation} 
This implies $[d_H(\cA^g(\cdot)|\ta^*)]^\rho \to 0$ a.s. and also bounded by an integrable function (see the proof in \cite[Theorem 16]{ducjost25}), thus $\E [d_H(\cA^g(\cdot)|\ta^*)]^\rho \to 0$ as $C_g \to 0$ by Lebesgue dominated convergence theorem. Now for any stationary solution $\ta^g(\omega)=(v^g,w^g) \in \cA^g(\omega)$, it is easy to check that $\ta^g \to \ta^*$ both in the almost sure sense and in the $\cL^\rho$ sense. In fact, a direct computation shows that 
\begin{eqnarray*}
|\E \ell (f,\ta^g) - \ell (f,\ta^*)| &\leq& \E |\ell (f,\ta^g)-\ell (f,\ta^*)| \leq \E \|Df(\ta^g)-Df(\ta^*)\|     \\
&\leq& 
\E|(v^g)^2 - (v^*)^2|\\
&\leq& 
 \E \Big[\|\ta^g- \ta^*\|(\|\ta^g-\ta^*\| + 2 \|\ta^*\| )\Big] \\
&\leq& \E \|\ta^g- \ta^*\|^2 + 2 \|\ta^*\|\; \E \|\ta^g- \ta^*\| \\
&\leq & \E \Big[d_H(\cA^g(\cdot)|\ta^*)\Big]^2 +2 \|\ta^*\|\; \E \Big[d_H(\cA^g(\cdot)|\ta^*)\Big] \to 0 \quad \text{as}\quad C_g \to 0. 
\end{eqnarray*}

Therefore, there exists $C_0\in (0,\frac{1}{8})$ such that for all $0 < C_g <C_0$ we have $\E\ell(f,\ta^g)<\frac{1}{2}\ell(f,\ta^*)<0$, and criterion \eqref{Anegdef} in Theorem \ref{local_diss_attractor} is satisfied by choosing $\lambda := C_g$, leading to local exponential stability for $\ta^g$. Note from \cite{ducjostdatmarius} that the noise might trigger a spike, and the trajectory $\{\ta^g(\theta_t\omega)\}_{t\in \R_+}$ might experience a tour out of the vicinity of the fixed point $\ta^*$ into a larger regime before coming back close to $\ta^*$. Although \eqref{FHNattractor} implies that this phenomenon is seldom as $C_g$ decreases to zero, the spiking part of the trajectory when happening could make the computation of Lyapunov exponents more complex than just for the linearization around $\ta^*$, due to the uncertainty of $g$ in practice.

This example can be generalized to any drift $f$ which is locally Lipschitz continuous and admits a singleton global attractor, such that there exists a strong Lyapunov function. 
\end{example}

 \subsection*{Special case: the trivial solution}\label{sec.stab.equilibrium}      
 Of particular interest is the stability for the trivial solution, for this we need to assume $f(0) = 0, g(0) =0$ so that system \eqref{RDE1} admits the trivial solution to be an equilibrium. Since the stability is a local property of the equation, it is natural to restrict our investigation and impose  conditions on the equation only in a neighborhood of the origin. In this subsection, instead of the global conditions (${\textbf H}_{f}$)-(${\textbf H}_{g}$) we will assume that there exist a positive constant $\epsilon_0>0$ such that in the ball $B(0,\epsilon_0):=\{y\in \R^d : \|y\| \leq \epsilon_0\} \subset \R^d$ the following (local) conditions are satisfied.
	
	(${\textbf H}^\eps_{f}$) in the ball $B(0,\epsilon_0)$ the coefficient function $f:\R^d \to \R^d$ of \eqref{RDE1} continuously differentiable and 
 the matrix   $Df(0) \in \R^{d \times d}$ is negative definite, i.e.
	\begin{equation}\label{lambdaf}
			\exists \lambda_f >0:\quad  \langle y, Df(0) y \rangle \leq - \lambda_f \|y\|^2,\quad \forall y \in \R^d.
		\end{equation}
        
	(${\textbf H}^\eps_g$) in the ball $B(0,\epsilon_0)$ the coefficient function $g: B(0,\epsilon_0)\to \cL(\R^m,\R^d)$ of the equation \eqref{RDE1}  belongs to $C^3(B(0,\epsilon_0),\cL(\R^m,\R^d))$; we denote 
	\begin{eqnarray}
C^*_g(\epsilon_0) &:=& \max \Bigg\{\|Dg\|_{\infty,B(0,\epsilon_0)},
\sqrt{\|g\|_{\infty,B(0,\epsilon_0)} \cdot\|Dg\|_{\infty,B(0,\epsilon_0)}},
\notag\\
&&\qquad \qquad \sqrt{\|g\|_{\infty,B(0,\epsilon_0)}\cdot\|D^2g\|_{\infty,B(0,\epsilon_0)}},
\sqrt{\|g\|_{\infty,B(0,\epsilon_0)}\cdot\|D^3g\|_{\infty,B(0,\epsilon_0)}},
\notag\\
&&\qquad \qquad \Big( \|D^2g\|_{\infty,B(0,\epsilon_0)}\|g\|^2_{\infty,B(0,\epsilon_0)} \Big)^{\frac{1}{3}}, \Big( \|D^3g\|_{\infty,B(0,\epsilon_0)}\|g\|^2_{\infty,B(0,\epsilon_0)} \Big)^{\frac{1}{3}} \Bigg\}.\label{gcond*}
	\end{eqnarray}

Note that hypothesis (${\textbf H}^\eps_{f}$) can also cover the case in which $Df(0)$ admits all eigenvalues of negative real parts, by upto a linear transformation (see e.g. \cite[Remark 3.6 (iii)]{duchongcong18}).

Our aim is to obtain the stability of the system \eqref{RDE1} in the neighborhood of the origin (an equilibrium). For this we need to show that the solutions starting near to 0 are not exploded, i.e. it can be extended to all the time $t>0$, and the requirement of the stability are met. 

The solution of \eqref{RDE1} with the initial value $y_0\in B(0,\epsilon_0)$ is understood in the pathwise sense, under the assumptions (${\textbf H}^\eps_f$)-(${\textbf H}^\eps_g$). To apply the arguments in \cite{duc20}, we provide an indirect argument to first extend the local domain $B(0,\epsilon_0)$ to the whole $\R^d$, and then to apply available results for global solutions. To do that, we first need to recall a result on extension of differentiable functions on $\R^d$. Namely, we have the following lemma which is a direct corollary of a theorem by C. Fefferman~\cite[Theorem 1]{fefferman}.
	\begin{lemma}\label{fefferman}
		There exists a positive constant $c_i^* \geq 1$, $i\in\N$, depending only on the dimensions $m$ and $d$ (independent from $\epsilon_0$) such that any function 
		$g: B(0,\epsilon_0) \to \cL(\R^m,\R^d)$ which is in the class $C_b^i(B(0,\epsilon_0))$ can be extended to a function $g^* : \R^d \to \cL(\R^m,\R^d)$ of the class $C_b^i(\R^d)$ with bounded derivatives up to order i, and the following inequality holds
		\[
		\|g^*\|_{C^i_b(\R^d)} \leq c_i^*\|g\|_{C^i_b(B(0,\epsilon_0))}.
		\]
	\end{lemma}
	Put
	\begin{equation}\label{constantc*} 
		c^*:= \max\{c_1^*, c_3^*\} \geq 1,
	\end{equation} 
	to be the universal constant that can serve for  estimations of smooth  extensions above. Hence, there are new functions
	\[
	f^*:  \R^d\to \R^d, \qquad 
	g^* : \R^d \to \cL(\R^m,\R^d),
	\]
	such that $f^*, g^*$ are the extensions of $f,g$ from $B(0,\epsilon_0)$ to $\R^d$ provided by Lemma~\ref{fefferman}, i.e. $f^*, g^*$ coincide with $f,g$ in $B(0,\epsilon_0)$ respectively. 
		Hence, we have
		\begin{equation}\label{fefferman2}
    \|Dg\|_{\infty,B(0,\epsilon_0)}
             \leq C^*_g(\epsilon_0)	\leq C_{g^*} \leq c^* C^*_g(\epsilon_0),\quad 
     \|f^*\|_{C^1(\R^d)} \leq  c^* \|f\|_{C^1(B(0,\epsilon_0))}.
		\end{equation}
		Consider the equation
		\begin{equation}\label{RDE1new}
			dy_t = f^*(y_t)dt + g^*(y_t)dx_t,
		\end{equation}
		It is easily seen that the functions $f^*, g^*$ satisfy the strong global assumptions (${\textbf H}_{f^*}$)-(${\textbf H}_{g^*}$) due to \eqref{fefferman2}, as well as the original local assumptions (${\textbf H}^\eps_{f}$)-(${\textbf H}^\eps_{g}$). 
		Therefore, there exists a unique global solution for equation \eqref{RDE1new} due to the choice of $f^*, g^*$. In particular, for any given solution $y_t(\bx,y_0)$ with $y_0\in B(0,\epsilon_0)$ there exists a time 
		\begin{equation}\label{localtime}
			\tau(\bx,y_0) := \sup \{t >0: y_s(\bx,y_0) \in B(0,\epsilon_0)\  \forall s \in [0,t] \} >0,
		\end{equation}
		such that all solution norm estimates can be computed via $f, g$ (instead of $f^*,g^*$) during the time interval $[0,\tau(\bx,y_0))$.

    As a direct consequence of Theorem \ref{local_diss_attractor}, the following result is significantly stronger than \cite[Theorem 17]{GABSch18}.
		
     \begin{theorem}\label{main1}
			Assume that $f(0) = 0$ and $g(0) =0$ so that zero is the trivial solution of \eqref{RDE1}, and the conditions (${\textbf H}^\eps_f$), (${\textbf H}^\eps_g$), (${\textbf H}_X$) hold. Then there exists $C_0>0$ depending only on $f$ such that  if $0<\|Dg(0)\|<C_0$ the trivial solution of \eqref{RDE1} is exponentially stable almost surely. 
		\end{theorem}
		\begin{proof}
			Given the equation \eqref{RDE1} on $B(0,\epsilon_0)$, we follow the extension process provided by Lemma~\ref{fefferman} to extend $f,g$ on $B(0,\epsilon_0)$ to $f^*,g^*$ on $\R^d$. Then consider equation \eqref{RDE1new}
			\begin{equation*}
				dy_t = f^*(y_t)dt + g^*(y_t)dx_t,
			\end{equation*}
			It is easily seen that the functions $f^*, g^*$ satisfy the strong assumptions (${\textbf H}_{f^*}$)-(${\textbf H}_{g^*}$) with \eqref{fefferman2}, thus the solutions of the equation \eqref{RDE1} exist on the whole half line $0\leq t <\infty$, hence \eqref{RDE1} generate a global random dynamical system, see e.g. \cite[Theorem 4.3]{riedelScheutzow}, \cite{duc21} and \cite{duckloeden}.
            Note that \eqref{fefferman2} implies \eqref{f_power_rate}, condition \eqref{Anegdef} can be verified by using \eqref{lambdaf}, and the trivial solution is a stationary solution of \eqref{RDE1}. Thus the assumptions of Theorem \ref{local_diss_attractor} are satisfied. Since both $g\in C^3_b$ and $g(0)=0$, the fact that $C_{g^*}$ is sufficiently small is equivalent to choosing $C^*_g(\epsilon_0)$ in \eqref{gcond*} to be as small as possible by considering a smaller ball 
    $B(0,\epsilon)\subset B(0,\epsilon_0)$ if necessary to make $\|g\|_{\infty,B(0,\epsilon)}$ as small as possible, and then requiring $\|Dg(0)\|$ to be sufficiently small so that $\|Dg\|_{\infty,B(0,\epsilon)}$ is controlled to be as small as possible. In other words, if $\|Dg(0)\| < C_0$ for sufficiently small $C_0$ then the trivial solution of the extended system \eqref{RDE1new} on a sufficiently small ball $B(0,\epsilon)\subset B(0,\epsilon_0)$ is locally exponentially stable under the additional assumptions (${\textbf H}_{f}$)-(${\textbf H}_{g}$)-\eqref{f_power_rate}. That is, there exist two positive random variables $\epsilon_0 > \epsilon \geq R^*(\omega)>0$, $\alpha^*(\omega)>0$ and a positive constant $\mu$ such that if $\|y_0\|\leq R^*(\omega)$ then the solution $y_t$ of \eqref{RDE1new}, starting from $y_0$, satisfies $\|y_t\|\leq \epsilon < \epsilon_0$  and
			\begin{equation}\label{cutoff.eqs}
				\|y_t\| \leq \alpha^*(\omega) \exp(-\mu t),\quad \forall t\geq 0.
			\end{equation}
			We notice here that since $f^*=f$ and $g^*=g$ on $B(0,\epsilon_0)$ the quantities $\|Df(0)\|$,  $C^*_g$ defined for the equation \eqref{RDE1} coincide with their counterparts defined for \eqref{RDE1new}.
			
			For any initial value $\|y_0\| \leq R^*(\omega) \leq \epsilon < \epsilon_0$ the solution $y_t$ of  \eqref{RDE1new}, starting from $y_0$, satisfies $\|y_t\| \leq \epsilon$ for all $t\geq 0$, hence it is the solution of \eqref{RDE1} starting from $y_0$ because \eqref{RDE1} coincides with \eqref{RDE1new} for all those solutions. Therefore, for any initial value $y_0$ satisfying $\|y_0\| \leq R^*(\omega)$ the solution $y_t$ of \eqref{RDE1} starting from $y_0$ satisfies $\|y_t\|\leq \epsilon$  and \eqref{cutoff.eqs} for all $t\geq 0$. 
			This implies that the trivial solution of equation \eqref{RDE1} is exponential stable. 
		\end{proof}

 \section{Lyapunov stability of discrete  systems}
In this section, we study the local dynamics of a discrete system of the form 
		\begin{equation}\label{REuler}
			\begin{split}
				y^\Delta_0 &\in \R^d,\\
				y^\Delta_{t_{k+1}} &= y^\Delta_{t_k} + f(y^\Delta_{t_k}) \Delta + g(y^\Delta_{t_k})x_{t_k, t_{k+1}} + Dg(y^\Delta_{t_k})g(y^\Delta_{t_k})\X_{t_k,t_{k+1}},\quad k \in \N.
			\end{split}
        \end{equation}
under the regular grid $\Pi = \{t_k:=k\Delta\}_{k \in \N}$, $0<\Delta\leq 1$. Note that numerical approaches to study rough differential equations go back to \cite{lyons94}, \cite{davie}, \cite{frizdis} (see also \cite{lejay}, \cite{liu} for further details). The global dynamics of the discrete system \eqref{REuler} has been studied recently in \cite{duckloeden}, \cite{congduchong23}, which show that there is a similarity in asymptotic behavior of the continuous system \eqref{RDE1} and its discretization \eqref{REuler} in the sense that the existing random attractor of the discrete system \eqref{REuler} converges to the random attractor of the continuous system \eqref{RDE1} as the step size $\Delta$ tends to zero. A difficulty in dealing with the discrete system is that we can not apply the Doss-Sussmann technique, simply because it is difficult to control the solution growth in a smooth way for the discrete time set. One way to overcome this challenge is to couple the discrete system \eqref{REuler} with its unperturbed discrete system and control the difference of the two trajectories by applying the discrete sewing lemma. For this, an additional condition of global Lipschitz continuity of $f$ is required. \\
  
 Throughout this section, we will assume (${\textbf H}_{f}$), (${\textbf H}_{g}$), (${\textbf H}_{X}$) and further more:   
\begin{equation}\label{Lf}
        L_f:=\|Df\|_{\infty}<\infty.
\end{equation}
In particular, denote by $L_g$ the quantity
\begin{equation}\label{Lgdis}
    L_g:=\max\Bigg\{\|Dg\|_{\infty},\|D^2g\|_\infty,\left[\|g\|_\infty (\|Dg\|_\infty\vee \|D^2g\|_\infty\vee \|D^3g\|_\infty)\right]^{\frac{1}{2}},  \left(\|g\|^2_\infty  \|D^3g\|_\infty\right)^{\frac{1}{3}} \Bigg\}. 
\end{equation}
which is somehow similar to $C_g$ in \eqref{gcond.new}.\\

We now state definition of stationary solution to discrete-
time systems \eqref{REuler} analogous to that for continuous time systems.

    \begin{definition}
    A random point $\ta^\Delta(\cdot):\Omega \to \R^d$ is called a stationary solution of \eqref{REuler} if $\ta^\Delta(\theta_t\omega), t\in\Pi,$ is the solution of \eqref{REuler} starting from $\ta^\Delta(\omega)\in\R^d$ at time 0.
    \end{definition}
The concept of stability for the stationary solution is defined in a similar way to Definition \ref{Defstability} except the continuous time set $\R_+$ is replaced by the discrete time set $\Pi$.

Following Section \ref{contRDE}, we assume that $\ta^\Delta$ is a stationary solution to \eqref{REuler}. 
Recall from Section \ref{contRDE} that
    \begin{equation}\label{M}
      \langle y - \ta^\Delta(\omega), f(y)-f(\ta^\Delta(\omega)) \rangle \leq M(Df,\ta^\Delta(\omega),r) \|y-\ta^\Delta(\omega)\|^2,\quad \forall \|y-\ta^\Delta(\omega)\| < r,\ \forall \omega\in\Omega
    \end{equation}
where it follows from \eqref{Mdf} and \eqref{Lf} that
\[
|M(Df,\ta^\Delta(\omega),r)|=| \ell(f,\ta^\Delta(\omega)) + osc(Df(\ta^\Delta(\omega)))_r |\leq 3L_f. 
\]
From now on, we write $ M(\ta^\Delta(\omega))$ in short for simplicity. We would like to study the local stability of a stationary solution $\ta^\Delta(\omega)$ of \eqref{REuler}. 
We write $\ta^\Delta_t = \ta^\Delta(\theta_t \omega)$ in short. To obtain local stability, we estimate the difference between an arbitrary solution $y^\Delta$ of \eqref{REuler}  and $\ta^\Delta$. Assign
$h_t:= y^\Delta_t-\ta^\Delta_t$ and 
\[
    P_{s,t}: = [g(y^\Delta_s)x_{s,t} + Dg(y^\Delta_s)g(y^\Delta_s)\X_{s,t}]-[g(\ta^\Delta_s)x_{s,t} + Dg(\ta^\Delta_s)g(\ta^\Delta_s)\X_{s,t}].
\]
Then
    \begin{equation}\label{h1}
    h_{t_{k+1}} = h_{t_k} + [f(y^\Delta_{t_k})-f(\ta^\Delta_{t_k})]\Delta +P_{t_k, t_{k+1}}.
    \end{equation}
and it is easy to check that
\begin{equation}\label{Pvsh}
\|P_{s,t}\|\leq \left[\|Dg\|_\infty \|x_{s,t}\| +(\|g\|_\infty \|D^2g\|_\infty +\|Dg\|^2_\infty )\|\mathbb{X}_{s,t}\|\right].\|h_s\|.
\end{equation}

\begin{proposition}\label{hnew}
If $4C_pL_g\ltn \bx\rtn_{\tp,\Pi[a,b]}\leq \lambda<\dfrac{1}{2}$, the following estimate holds
        \begin{equation*}
		  \ltn h, R^h \rtn_{\tp,\Pi[a,b]}
		   \leq\Big[L_f (b-a)+2\lambda \|f(\ta^\Delta)\|_{\infty,[a,b]} (b-a)+(4C_p+1)L_g \ltn \bx\rtn_{\tp,\Pi[a,b]} \Big] \| h,R^h\|_{\tp,\Pi[a,b]}. 
	   \end{equation*}
\end{proposition}
\begin{proof} 
See proof in the Appendix.\\
\end{proof}
This motivates us to construct a discrete sequence of stopping times $\tau^{\Delta}_n (\lambda,\cS,[0,\infty))$ as presented in Subsection \ref{dis_greedy} based on a fixed $\lambda<\dfrac{1}{2}$ and
\[
	\mathcal{S} =\{\bw^{(1)}, \bw^{(2)}, \bw^{(3)}, \bw^{(4)}\},
\]
in which 
\begin{eqnarray*}
&&\bw^{(1)}_{s,t}=L_f(t-s), \beta_1 = 1;\qquad
 \bw^{(2)}_{s,t} = (4C_p+1)^p L_g^p \ltn x \rtn^p_{\tp,\Pi[s,t]}, \beta_2 = \frac{1}{p};\\ 
&&\bw^{(3)}_{s,t} =(4C_p+1)^pL_g^p \ltn\X \rtn^{q}_{\tq,\Pi[s,t]}, \beta_3 = \frac{1}{q};\qquad
\bw^{(4)}_{s,t} = 2\lambda  \|f(\ta^\Delta)\|_{\infty,\Pi[s,t]}(t-s),  \beta_4 =  1.
\end{eqnarray*}
Note that $\bw^{(4)}_{s,t}, a\leq s\leq t\leq b $ is a control since it  is continuous, superadditive and zero on the diagonal.

Throughout this section, we will write $\tau^{\Delta}_n=\tau^{\Delta}_n (\lambda,\cS,[0,\infty))$ for abbreviation without indicating $\cS$.
Observe that whenever $\|y^\Delta_{t_k}-a^\Delta_{t_k}\| <r$,  
		\begin{eqnarray}\label{eqRDE:02}
			\|h_{t_{k+1}}\|^2 &=& \|h_{t_k}\|^2 + \|f(y^\Delta_{t_k})-f(\ta^\Delta_{t_k})\|^2 \Delta^2+ \|P_{t_k,t_{k+1}}\|^2 \notag\\
            &&+2 \langle y^\Delta_{t_k}-\ta^\Delta_{t_k}, f(y^\Delta_{t_k})-f(\ta^\Delta_{t_k})\rangle \Delta  + 2 \langle h_{t_k}, P_{t_k,t_{k+1}} \rangle + 2 \langle f(y^\Delta_{t_k})-f(\ta^\Delta_{t_k}), P_{t_k,t_{k+1}} \rangle \Delta \notag\\
    	&\leq& \|h_{t_k}\|^2 \left[1+2  M(\ta^\Delta_{t_k} )\Delta +L_f^2 \Delta^2 \right]+ 2 \langle f(y^\Delta_{t_k})-f(\ta^\Delta_{t_k}), P_{t_k,t_{k+1}}\rangle \Delta \notag\\ 
        &&+ 2 \langle h_{t_k}, P_{t_k,t_{k+1}} \rangle +\| P_{t_k,t_{k+1}}\|^2.
		\end{eqnarray}
Consider the first interval $[\tau^\Delta_0,\tau^\Delta_1]$ where $\tau^\Delta_0 =0$. There are two possibilities. The trivial case is when 
\[ 
\Big[ L_f (\tau^\Delta_{1}-\tau^\Delta_0)+2\lambda \|f(\ta^\Delta)\|_{\infty,\Pi[\tau^\Delta_0,\tau^\Delta_1]} (\tau^\Delta_1-\tau^\Delta_0)+(4C_p+1)L_g\ltn \bx \rtn_{\tp,\Pi[\tau^\Delta_0,\tau^\Delta_{1}]} \Big] > \lambda,     
\] 
then $\tau^\Delta_0,\tau^\Delta_{1}$ are consecutive in $\Pi$, i.e. $\tau^\Delta_1-\tau^\Delta_0 = \Delta$. Then \eqref{Pvsh} and \eqref{eqRDE:02} yield
 \begin{eqnarray*}
				\|h_{\tau^\Delta_1}\|^2 &\leq&  \Big[ 1+2  M(\ta^\Delta_{\tau^\Delta_0 })\Delta +L_f^2 \Delta^2 +2(1+L_f\Delta)(L_g\|\bx_{\tau^\Delta_0,\tau^\Delta_1}\|+L^2_g\|\bx_{\tau^\Delta_0,\tau^\Delta_1}\|^2 )  \Big.\notag\\
                &&\Big. \hspace{4cm}+(L_g\|\bx_{\tau^\Delta_0,\tau^\Delta_1}\|+L^2_g\|\bx_{\tau^\Delta_0,\tau^\Delta_1}\|^2 )^2\Big]\|h_{\tau^\Delta_0}\|^2\\
                &\leq&   \exp \Big\{\Big[2M(\ta^\Delta_{\tau^\Delta_0 }+ L^2_f\Delta\Big](\tau^\Delta_1-\tau^\Delta_0)+3e^{6L_f\Delta}L_g\ltn \bx\rtn_{\tp,[\tau^\Delta_0,\tau^\Delta_1] }\Big\}\|h_{\tau^\Delta_0}\|^2 
\end{eqnarray*}
as $\|h_{\tau^\Delta_0}\|\leq r$ by assumption, where the last inequality comes from the Taylor expansion of the exponential function.
The other case is non-trivial and needs to be proved as follows.
\begin{proposition}\label{case1}
  Assume $\|h_{\tau^\Delta_{0}} \|\leq \frac{r}{16(1+C_p)} $.  If 
  \[
    \Big[L_f (\tau^\Delta_{1}-\tau^\Delta_0) +2\lambda \|f(\ta^\Delta)\|_{\infty,\Pi[\tau^\Delta_0,\tau^\Delta_1]} (\tau^\Delta_1-\tau^\Delta_0)+(4C_p+1) L_g\ltn \bx \rtn_{\tp,\Pi[\tau^\Delta_0,\tau^\Delta_{1}]}\Big] \leq \lambda
\] then 
    \begin{equation}\label{ind}
			\|h_{\tau^{\Delta}_1}\|^2
			\leq \exp \Big\{\Big[2\bar{M}_0 + L_f^2\Delta\Big](\tau^\Delta_1-\tau^\Delta_0)  +6(1+\lambda)(1+2C_p)^2e^{12\lambda+2\lambda^2}L_g\ltn \bx\rtn_{\tp,\Pi[\tau^\Delta_0,\tau^\Delta_{1}]} \Big\}\|h_{\tau^\Delta_0}\|^2.
     		  \end{equation} 
where $\bar{M}_0$ denotes the average of $M(\ta^\Delta_{t_k}),\ t_k\in [\tau^\Delta_0,\tau^\Delta_1)$.
\end{proposition}
The proof of Proposition \ref{case1} is provided in the Appendix. Hence in both cases, we have have just proved that
		\begin{eqnarray}
			\|h_{\tau^\Delta_1}\| 
          		&\leq&   \exp \Big\{\Big[\bar{M}_0+\frac{1}{2}L_f^2\Delta\Big](\tau^\Delta_1-\tau^\Delta_0)   +KL_g\ltn \bx\rtn_{\tp,[\tau^\Delta_0,\tau^\Delta_1]}\Big\}  \|h_{\tau^\Delta_0}\| \notag\\
          		&\leq& \exp \Big\{\Big[\bar{M}_0+\frac{1}{2}L_f^2\Delta\Big](\tau^\Delta_1-\tau^\Delta_0)  +KL_g \gamma^*  N^*(\gamma^*,\bx,[\tau^\Delta_0,\tau^\Delta_1])\Big\}  \|h_{\tau^\Delta_0}\|\notag
      		 \end{eqnarray}  
for 
\begin{equation}\label{constantK}
K:=5(1+\lambda)(1+2C_p)^2e^{12\lambda+2\lambda^2} \vee \frac{3}{2}e^{6L_f\Delta}  
\end{equation}
 and arbitrary $\gamma^* \in(0,1)$,  where the last estimate follows from Lemma \ref{infN}.  By induction, we therefore can prove that.  

\begin{proposition}\label{z}
If $\|h_{\tau^\Delta_{n}} \|\leq \frac{r}{16 (1+C_p)} $ then
		\begin{eqnarray} 
			\|h_{\tau^\Delta_{n+1}}\|
    		      &\leq& \exp \Big\{ \Big[\bar{M}_n+\frac{1}{2}L_f^2\Delta\Big](\tau^\Delta_{n+1}-\tau^\Delta_n)  + KL_g \gamma^*  N^*(\gamma^*,\bx,[\tau^\Delta_n,\tau^\Delta_{n+1}])\Big\}   \|h_{\tau^\Delta_n}\|
 	      \end{eqnarray} 
where $\bar{M}_n$ denotes the average of $M(\ta^\Delta_{t_k}),\ t_k\in [\tau^\Delta_n,\tau^\Delta_{n+1})$ and $\gamma^*\in(0,1)$. 
\end{proposition}

Note that \eqref{N*} yields the estimate for the number of the continuous stopping time $\tau_n(\lambda, \cS,[s,t])$ w.r.t. the set of control $\cS$ as
        \begin{eqnarray}\label{N_new}
			N^*(\lambda,\mathcal{S},[s,t])
			& <& 1 + \frac{4^{p-1}}{\lambda^p} \Big[L_f^p(t-s)^p +  (4C_p+1)^pL_g^p\ltn x \rtn^p_{\tp,\Pi[s,t]} + (4C_p+1)^pL_g^{p}\ltn \X \rtn^{q}_{\tq,\Pi[s,t]}\notag\\
            &&\qquad \qquad \qquad + 2^p\lambda^p  \|f(\ta^\Delta)\|^p_{\infty,\Pi[s,t]}(t-s) \Big]. \notag\\
			&<& 1 + \frac{4^{p-1}}{\lambda^p} \Big[L_f^p(t-s)^p +  (4C_p+1)^p L_g^p\ltn \bx \rtn^p_{\tp,\Pi[s,t]} +2^p\lambda^p  \|f(\ta^\Delta)\|^p_{\infty,\Pi[s,t]}(t-s)^p  \Big].\notag\\
		\end{eqnarray}
Now we state the main result of this section. 
\begin{theorem}\label{stable1}
Assume (${\textbf H}^+_{f}$) with \eqref{Lf}, (${\textbf H}_{g}$), (${\textbf H}_{X}$).  Assume further that   $\|\ta^\Delta(\cdot)\| \in L^p$ and 
      \begin{equation}\label{cond.}
      		-\eta:= \E \ell(f,\ta^\Delta(\cdot))  + \frac{1}{2}L_f^2\Delta <0.
      \end{equation}
If there exist $\lambda \in (0,\frac{1}{2}), \gamma^* \in (0,1)$ such that
            \begin{eqnarray}\label{discstabcri.gen} 
				\eta>KL_g \gamma^*\Big[\E N^*(\gamma^*,\bx(\cdot),[0,1])+ 2 \E N^*(\lambda,\cS,[0,1])\Big],
			\end{eqnarray}
where $K$ is given in \eqref{constantK} and numbers $N^*$s are estimated by \eqref{Nest}, \eqref{N_new}. Then the stationary solution $\ta^\Delta(\cdot)$ of \eqref{REuler} is exponentially stable almost surely.
\end{theorem}
\begin{proof}
First, choose $0<r<r_0$ small enough so that 
\[
        \eta-\E[osc(Df(\ta^\Delta(\cdot)))_r ]>KL_g \gamma^*\Big[\E N^*(\gamma^*,\bx(\cdot),[0,1])+ 2 \E N^*(\lambda,\cS,[0,1])\Big].
\]
Then \eqref{M} is satisfied in which 
\[
       -\left[\E M(Df,\ta^\Delta(\cdot),r)+\frac{1}{2}L_f^2\Delta\right]>  KL_g \gamma^*\Big[\E N^*(\gamma^*,\bx(\cdot),[0,1])+ 2 \E N^*(\lambda,\cS,[0,1])\Big].
\]
Using Proposition \ref{z} and the definition of $\bar{M}_m$ for $0\leq m\leq n$, we obtain by induction   
  		\begin{equation}\label{ydiscest_gen}
				\|h_{\tau_n^\Delta}\|
				 \leq \|h_0\| \exp \left\{ \frac{\sum_{i=0}^{\frac{\tau^\Delta_n}{\Delta}-1} \left[M(\ta^\Delta_{t_i}) +\frac{1}{2}L_f^2\Delta \right]}{\tau^\Delta_n/\Delta}\tau_n^\Delta  +K L_g \gamma^*\sum_{j =0}^{n-1} N^*(\gamma^*,\bx,[\tau^\Delta_j,\tau^\Delta_{j+1}])\right\}
			\end{equation}
provided that  $\|h^\Delta_{\tau_m^\Delta}\| \leq  \dfrac{r}{16(1+C_p)} $. By applying \eqref{Nsum} in Lemma \ref{Nsumin} to \eqref{ydiscest_gen}, we obtain
			\begin{eqnarray}\label{ydiscest2}
				\|h_{\tau_n^\Delta}\| &\leq& \|h_0\| \exp \left\{ \frac{\sum_{i=0}^{\frac{\tau^\Delta_n}{\Delta}-1} \alpha_i}{\tau^\Delta_n/\Delta}\tau_n^\Delta  +K L_g \gamma^* N^*(\gamma^*,\bx,[0,\tau^\Delta_{n}])  + KL_g\gamma^* n      \right\},
			\end{eqnarray}
where we write $\alpha_i=M(\ta^\Delta_{t_i}) + \frac{1}{2}L_f^2\Delta$ in short. Hence, to make sure that $\|h_{\tau_n^\Delta}\| \leq \frac{r}{16(1+C_p) }$ for any $n\in \N$ it suffices to chooses $y^\Delta_0$ such that
			\begin{eqnarray}\label{discR}
				&&\|y^\Delta_0 - a^\Delta(\omega)\| \leq R^\Delta(\omega) \notag\\
				&:=&   \frac{r}{16(1+C_p)} \inf_{n \geq 1} \Big[ \exp \Big\{-\frac{ \sum_{i=0}^{\frac{\tau^\Delta_n}{\Delta}-1}    \alpha_i }{\tau^\Delta_n/\Delta}\tau_n^\Delta   - KL_g \gamma^*N^*(\gamma^*,\bx(\omega),[0,\tau^\Delta_{n}]) - KL_g \gamma^* n \Big\}\Big] \notag\\
				&=& \frac{r}{16(1+C_p)} \inf_{n \geq 1} \left[ \exp \Big\{-\tau^\Delta_n \Big(\frac{\sum_{i=0}^{\frac{\tau^\Delta_n}{\Delta}-1}\alpha_i}{\tau^\Delta_n/\Delta} - KL_g \gamma^* \frac{N^*(\gamma^*,\bx(\omega),[0,\tau^\Delta_{n}])}{\tau_n^\Delta}   - KL_g\gamma^* \frac{n}{\tau_n^\Delta} \Big)\Big\} \right].\notag\\
			\end{eqnarray}
We show that $R>0$ a.s. under conditions  \eqref{cond.} and \eqref{discstabcri.gen}. Indeed, denote $\tau_k(\gamma^*,\bx(\omega),[0,\infty))$ the maximal stopping time that is less than $\tau^\Delta_{n} =\tau^\Delta_{n}(\lambda,\cS)$, then $N^*(\gamma^*,\bx(\omega), [0,\tau^\Delta_{n}]) \leq k+1$. By \eqref{infD} in Lemma \ref{infimumD}
			\begin{equation}\label{expN}
				\limsup \limits_{n \to \infty} \frac{N^*(\gamma^*,\bx(\omega), [0,\tau^\Delta_{n}])}{\tau_n^\Delta} \leq \limsup \limits_{k \to \infty} \frac{k+1}{\tau_k(\gamma^*,\bx(\omega),[0,\infty))} \leq \E N^*(\gamma^*,\bx(\cdot),[0,1]).
			\end{equation}
Mean while, applying \eqref{disinfD} and \eqref{infDS} for $\gamma := \lambda$ yields
			\[
			\limsup \limits_{n \to \infty} \frac{n}{\tau_n^\Delta(\lambda, \cS,[0,\infty))} \leq \frac{1}{\liminf \limits_{n \to \infty} \frac{\tau_n(\lambda,\cS,[0,\infty))}{2n} } \leq2 \E N^*(\lambda,\cS,[0,1])\quad \text{a.s.}
			\]
Thus for a given path $\bx(\omega)$ and a fixed parameter
\[
\eps< -\left[\E M(Df,\ta^\Delta(\cdot),r)+\frac{1}{2}L_f^2\Delta\right] - KL_g\gamma^* \E N^*(\gamma^*,\bx(\cdot),[0,1])- 2KL_g\gamma^* \E N^*(\lambda,\cS,[0,1])
\]
there exists $n(\omega)>0$ large enough such that
    \begin{align}
         &\frac{r}{16(1+C_p)} \inf_{n \geq n(\omega)} \left[ \exp \Big\{\tau^\Delta_n \Big(\frac{\sum_{i=0}^{\frac{\tau^\Delta_n}{\Delta}-1}\alpha_i}{\tau^\Delta_n/\Delta} - KL_g \gamma^* \frac{N^*(\gamma^*,\bx,[0,\tau^\Delta_{n}])}{\tau_n^\Delta}   - KL_g\gamma^* \frac{n}{\tau_n^\Delta} \Big)\Big\} \right].\notag\\
        & \geq\frac{r}{16(1+C_p)} \inf_{n \geq n(\omega)} e^{\eps \tau^\Delta_n} >0.       
    \end{align}
This implies that $R(\omega)>0$ a.s. These arguments together with similar arguments to the continuous case presented in the proof of Theorem \ref{local_diss_attractor} prove that there exist $\mu>0$, and positive random variable $\alpha(\omega)$ such that
 if $\|y_0-\ta^\Delta(\omega)\|\leq R(\omega)$ then
			\[
			\|y^\Delta_{\tau^\Delta_n}-\ta^{\Delta}_{\tau^\Delta_n}\| \leq \alpha(\omega) \exp (- \mu \tau^\Delta_n) \quad \hbox{for all}\quad n\geq 0,
			\]
Finally, for any $\tau^\Delta_n<t_k<\tau^\Delta_{n+1}$ for some $n$, \eqref{hcase1} implies
            \[
			\|y^\Delta_{t_k}-\ta^{\Delta}_{t_k}\| \leq 2 \alpha(\omega) \exp (- \mu \tau^\Delta_n) \leq  2 \alpha(\omega) e^{\mu(t_k-\tau^\Delta_n))} \exp (- \mu t_k)\leq  2 \alpha(\omega) e^{\mu\frac{\lambda}{L_f}} \exp (- \mu t_k).
			\]
This proves the exponential stability of $\ta^\Delta(\omega)$. 
\end{proof}

        \begin{remark}\label{specialcase}
\begin{enumerate}
\item We expect that condition \eqref{Lf} is negligible provided (${\textbf H}^+_{f}$), as we only consider the local dynamics in the vicinity of the discrete stationary solution $\ta^\Delta$. The only difference in computation is the estimate
 \[
\| f(y^\Delta_s) -f(\ta^\Delta_s) \| \leq \left\|\int_0^1 Df(\ta^\Delta_s + \chi h_s)d\chi \right\|\|h_s\| \leq C_f\Big[1+(\|\ta^\Delta_s\|+1)^\rho\Big] \|h_s\| =: C(f,\ta^\Delta_s)\|h_s\|
 \]
 as long as $\|h_s\|\leq 1$, where $ C(f,\ta^\Delta)$ is an integrable coefficient. In this case the simple term $L_f(t-s)$ should be replaced by $\int_s^t C(f,\ta^\Delta_u)du$, thus the formulae and the stability criterion might be more complex in the end. 
    \item Observe from the proof of Proposition \ref{hnew} that in case $\ta^\Delta$ is a non-random fixed point then $\ltn \ta^\Delta, R^{\ta^\Delta}\rtn_{\tp,\Pi[a,b]}=0$. The definition of $L_g$ can then be reduced to 
    \[
    L_g=\max\{\|Dg\|_\infty,(\|g\|_\infty\|D^2g\|_\infty)^{\frac{1}{2}}\}.
    \]
      
 \item Criterion \eqref{discstabcri.gen} for discrete systems is similar to criterion \eqref{Anegdef}  or \eqref{Anegdef_new}  for continuous systems, with $\kappa$ in the left hand side of \eqref{Anegdef_new} is reduced to $\ell$; the only difference is $N^*(\lambda,\cS,[0,1])$ in the right hand side of \eqref{discstabcri.gen}. In particular, all the controls in $\cS$ depend only on $\bx$ and $\|f(\ta^\Delta)\|_\infty$ that are computable in practice, hence no further information on $g$ than $L_g$ is needed.

\item In general  it is difficult to check \eqref{discstabcri.gen}, because the number $N^*(\lambda,\cS,[0,1])$ in the right hand side of \eqref{discstabcri.gen} might also depend on $L_g$ (as seen in the estimate \eqref{N_new}). A careful look at \eqref{discstabcri.gen} and \eqref{N_new} shows that if there exists an $L^{p}$- random variable which bounds $E\ltn \bz\rtn_{\infty,\Pi[0,1]}$, then the right hand side of \eqref{discstabcri.gen} does not depend on $\ta^\Delta$. This observation suggests us to consider two special cases for the dissipative drift as below.
 \end{enumerate} 
\end{remark}

\medskip
In the following, we assume $f$ is global dissipative in the sense of \eqref{globaldiss}, i.e. there exist constants $D_1, D_2>0$ so that
    \begin{equation*}
        \langle y, f(y) \rangle \leq  D_1- D_2\|y\|^2,\quad \forall y \in \R^d.
    \end{equation*} 
We modify the proof in \cite[Theorem \ 3.3 \& 4.2]{congduchong23} to prove that
\begin{proposition}\label{attractor}
   Assume (${\textbf H}^+_{f}$) with \eqref{Lf}, (${\textbf H}_{g}$), (${\textbf H}_{X}$). Assume more \eqref{globaldiss}. There exists $\Delta^*$ so that if $\Delta<\Delta^*$, there exists a pullback random attractor $\mathcal{A}^\Delta(\omega)$ to \eqref{REuler} which lies inside the ball $B(0,\mathcal{R}^{1/p}(\omega))$, where
    \begin{equation}\label{radius}
			\mathcal{R}=\mathcal{R}(\omega)
            = \Gamma+ \Gamma \sum_{k=0}^{\infty} e^{ \frac{-D_2}{8}k }\left(1+\ltn \theta_{-k} \bx(\omega)\rtn^{p(p+2)}_{\tp,[-1,2]} \right)
		\end{equation}
in which $\Gamma= \Gamma(L_f, \|f(0)\|, D_1,D_2, \|g\|_{C^3_b})$ is a constant dependent increasingly on $\|g\|_{C^3_b}$.
\end{proposition}
   \begin{proof}
       The proof follows that in \cite[Theorem\ 3.3 \& 4.2]{congduchong23} and \cite[Theorem \ 3.11]{duchong21} with a slight modification to make use assumption boundedness of $g$. For the benefit of the reader, we sketch the proof in the appendix.
   \end{proof}
\medskip

\begin{corollary}\label{rhs}
Under the assumptions of Proposition \ref{attractor}, it holds that 
    \[
    \E N^*(\gamma^*,\bx(\cdot),[0,1])+ 2 \E N^*(\lambda,\cS,[0,1]) \leq \bar{\Gamma} \left[1+ \E(\ltn \bx(\cdot)\rtn^{p(p+2)}_{\tp,[0,1]}) \right]
    \] 
for a certain constant $\bar{\Gamma}=\bar{\Gamma}(L_f, \|f(0)\|, D_1,D_2,\|g\|_{C^3_b},\lambda,\gamma^*) >0$ dependent increasingly on $\|g\|_{C^3_b}$.
\end{corollary}
\begin{proof}
Since stationary solution $\ta^\Delta$ belongs to $\mathcal{A}^\Delta$, Proposition \ref{attractor} implies 
    $$
        \|\ta^\Delta\|^p\leq \mathcal{R}.
    $$
Moreover, from \eqref{mu} and \eqref{h_inf} 
        \begin{eqnarray*}\label{sol.est}
			\|\ta^\Delta\|_{\infty,[0,T]}&\leq &2\|\ta^\Delta\|+ \Lambda(\bx,[0,1]))
		\end{eqnarray*}
where $\Lambda(\bx,[0,1])=\Gamma(1+ \ltn \bx \rtn^{(p+2)}_{\tp,[0,1]})$. 
By convex inequality, it follows that
     \begin{eqnarray*}
           \|f(\ta^\Delta)\|^p_{\infty,\Pi[0,1]}
          &\leq & \Gamma\mathcal{R}.
    \end{eqnarray*}
From \eqref{Nest} and \eqref{N_new}, there exists generic constant $\bar{\Gamma}=\bar{\Gamma}(L_f, \|f(0)\|, D_1,D_2,\|g\|_{C^3_b},\lambda,\gamma^*) >0$ such that
        \begin{eqnarray*}
			&& N^*(\gamma^*,\bx(\cdot),[0,1])+ 2 N^*(\lambda,\cS,[0,1])\\ 
			&<&  2 +\frac{1}{(\gamma^*)^p}\ltn \bx(\cdot) \rtn^p_{\tp,[0,1]}+ \frac{4^{p-1}}{\lambda^p} \Big[L_f^p +  (4C_p+1)^p L_g^p\ltn \bx (\cdot)\rtn^p_{\tp,\Pi[0,1]} +2^p\lambda^p  \|f(\ta^\Delta(\cdot))\|^p_{\infty,\Pi[s,t]} \Big].\notag\\
            &<&  \bar{\Gamma}\mathcal{R}.
		\end{eqnarray*}
By \eqref{radius}, it is easy to check that 
    \[
       \E(\mathcal{R}) \leq  \Gamma\left[1+ \dfrac{\E(\ltn \bx\rtn^{p(p+2)}_{\tp,[0,1]})}{1- e^{\frac{-D_2}{8}}} \right].
    \]
The proof is finished.
\end{proof}

As a consequence, we now derive a simpler stability criterion for the case of global dissipativity. 
\begin{theorem}\label{disccridiss}
Assume (${\textbf H}^+_{f}$) with \eqref{Lf} and \eqref{globaldiss}, (${\textbf H}_{g}$), (${\textbf H}_{X}$) such that \eqref{cond.} holds. If
 \begin{eqnarray}\label{discstabcri_gendis} 
				\eta> K\bar{\Gamma}  L_g  \left[1+ \E(\ltn \bx\rtn^{p(p+2)}_{\tp,[0,1]}) \right];
			\end{eqnarray}
where  $K$ is given in \eqref{constantK} and $\bar{\Gamma}$ is given in Corollary\ \ref{rhs}, then the stationary solution $\ta^\Delta(\cdot)$ of \eqref{REuler} is exponentially stable almost surely.
\end{theorem}
Criterion \eqref{discstabcri_gendis} is easier to check by fixing $\lambda, \gamma^*$ so that only $\bar{\Gamma} L_g$ in the right hand side depends increasingly on $L_g$.  
In the case of strict dissipativity, the following result concludes the stability for the singleton attractor. 
Note that the singleton attractor of the discrete system \eqref{REuler} converges to that of the continuous system as $\Delta\to 0$, as proved in \cite{congduchong23}. 

\begin{corollary}\label{strict}
Assume (${\textbf H}^+_{f}$) with \eqref{Lf}, (${\textbf H}_{g}$), (${\textbf H}_{X}$) and further that $f$ is strictly dissipative in the sense \eqref{strictdissipative}. There exist $\Delta^*>0$ and $C^*>0$,   such that for any $\Delta<\Delta^*$, any $L_g< C^*$, there exists a unique stationary solution to \eqref{REuler} which is also the singleton attractor. Moreover, the stationary solution is exponentially stable almost surely.
		\end{corollary}
\begin{proof}
The existence of of the random pullback attractor is given in Proposition \ref{radius} given that $\Delta$ small enough. The fact that the attractor is one point under sufficient small $L_g$ is proved in a similar way to the continuous case in \cite{duc21} and \cite{duchong21}
 where we make use of Proposition \ref{z} for $h$ being defined as the difference of two stationary solutions.
Thus, one may choose $\Delta^*$ depending on $f$, $C^*$ depending on $f,g$,  $\E(\ltn \bx\rtn^{p(p+2)}_{\tp,[0,1]})$ so that 
\eqref{cond.}, \eqref{discstabcri.gen}  are satisfied. The exponential stability is then a direct consequence of Theorem \ref{stable1}.
\end{proof}

\subsection*{Special case: the trivial solution}\label{sec_stab_discrete}
We close this section by considering the special case in which $a^\Delta$ is the trivial solution of the discrete system under assumptions (${\textbf H}^\eps_{f}$), (${\textbf H}^\eps_{g}$) as stated at the end of Section \ref{contRDE}.
First define
\[        L^*_f(\eps_0)=\|Df\|_{\infty, B(0,\eps_0)}
\]
and, by Remark \ref{specialcase} (1), also define
\[            L^*_g(\eps_0)=\max\{\|Dg\|_{\infty, B(0,\eps)},\|g\|_{\infty, B(0,\eps)}\|D^2g\|_{\infty, B(0,\eps)}\}.
 \]
We consider the discrete scheme associate with \eqref{RDE1new}
 	\begin{equation}\label{REuler_new}
			\begin{split}
				y^\Delta_0 &\in \R^d,\\
				y^\Delta_{t_{k+1}} &= y^\Delta_{t_k} + f^*(y^\Delta_{t_k}) \Delta + g^*(y^\Delta_{t_k})x_{t_k, t_{k+1}} + Dg^*(y^\Delta_{t_k})g^*(y^\Delta_{t_k})\X_{t_k,t_{k+1}},\quad k \in \N.
			\end{split}
        \end{equation}
The following result is similar to Theorem \ref{main1}.
		\begin{theorem}\label{main1_new}
			Assume  (${\textbf H}_X$), (${\textbf H}^\eps_f$) and (${\textbf H}^\eps_g$) hold with $f(0)=0, g(0) =0$. Then there exists $\Delta_0>0$, $C_0>0$ depending only on $f$ such that  for any $0<\Delta<\Delta_0$, $\|Dg(0)\|<C_0$,  
			the trivial solution of \eqref{REuler} is exponentially stable almost surely. 
		\end{theorem}
		\begin{proof}
Firstly consider \eqref{REuler_new} with the note that $c^*L^*_f(\eps_0), c^*L_g^*(\eps_0)$	play the role of $L_f,L_g$ in \eqref{cond.}, \eqref{discstabcri.gen}, where $c^*$ is determined in Lemma \ref{fefferman}. 

We fix $\Delta_0 < \frac{2\lambda_f}{(c^*L_f(\eps_0))^2}$, $\gamma^*=\dfrac{1}{2},\lambda=\dfrac{1}{3}$ and choose $C_0<1$ satisfy
\[
\eta=\lambda_f-\frac{(c^*L_f(\eps))^2\Delta_0}{2}> \dfrac{1}{2}Kc^*C_0\Big[\E N^*(\dfrac{1}{2},\bx,[0,1])+ 2 \E N^*(\dfrac{1}{3},\cS,[0,1])\Big].
\]   
This is possible, because we can estimate the two expectations due to \eqref{Nest} and \eqref{N*}
			\begin{eqnarray*}
				\E N^*(\frac{1}{2},\bx(\cdot),[0,1]) &\leq& 1+ 2^p \E\ltn \bx(\cdot) \rtn^p_{\tp,[0,1]} =: C_1;\\  
				\E N^*(\frac{1}{3},\cS,[0,1])
				&\leq& 1 + 3^p(c^*L_f(\eps_0))^p + (12C_pc^*)^p \E\ltn \bx(\cdot) \rtn^p_{\tp,[0,1]}  =: C_2 
			\end{eqnarray*}
			by the quantities which are independent of $g$. We then determine $C_0$ as follows
			\[
			C_0 := \min \Big\{1,\frac{2\eta}{c^*K (C_1+2C_2)} \Big\}.
			\]
If $\|Dg(0)\|<C_0$ there exists $\eps_1<\eps_0$ so that $\|Dg\|_{\infty,B(0,\eps_1)}<C_0$. Choose
	\begin{equation}\label{new_eps_2}
				\epsilon_2 \leq \min\{ \frac{C_0 }{1+ \|D^2g\|_{\infty,B(0,\epsilon_0)}},  \epsilon_1 \}. 
			\end{equation}
			Then $B(0,\epsilon_2)\subset B(0,\epsilon_0)$ and we have
			\begin{equation*}\label{eqn_main1}
				\|D^ig\|_{\infty,B(0,\epsilon_2)}\leq
				\|D^ig\|_{\infty,B(0,\epsilon_0)} \quad\hbox{for}\quad i=0,1,2,3.
			\end{equation*}
            Since $\|Dg\|_{\infty,B(0,\epsilon_2)}\leq C_0$ and $g(0)=0$ we have 
			$\|g\|_{\infty,B(0,\epsilon_2)}\leq C_0\epsilon_2$. Hence
			\begin{eqnarray}\label{est.tilde.C}
				\|g\|_{\infty,B(0,\epsilon_2)}  \|D^2g\|_{\infty,B(0,\epsilon_2)}\leq  C_0\epsilon_2 \|D^2g\|_{\infty,B(0,\epsilon_2)} < C^2_0;
			\end{eqnarray}
which implies $L^*_g(\eps_2)<C_0$. It is easily seen that with such choice of $\Delta_0$ and $C_0$ if $0<\Delta \leq \Delta_0$ and $0< \|Dg(0)\|< C_0$, the criterion  \eqref{discstabcri.gen}  is satisfied. hence Theorem \ref{stable1} is applicable to \eqref{REuler_new} to conclude the exponential stability of the trivial solution. Moreover, one may replace $r$ in \eqref{discR} by $\eps_2$ so that any solution $y^\Delta$ of \eqref{discR} starts inside $B(0, R^\Delta(\omega))$ where $R^\Delta(\omega)$ given in \eqref{discR} does not exit from the ball. Therefore, $y^\Delta$ is also the solution of \eqref{REuler}. This completes the proof. 
\end{proof}

		\section*{Acknowledgments}
		The authors would like to thank the anonymous referees for their useful comments which led to the improvements of the paper. This work is partly supported by the project NVCC01.10/24-25 of Vietnam Academy of Science and Technology.
        L. H. D. would like to thank Bj\"orn Schmalfu{\ss} for the invitation to give a talk in the Stochastics seminar at Friedrich-Schiller-Universit\"at Jena (Germany) in June 2023. 
		    
		\section{Appendix}\label{appen}
		\subsection{Rough paths}\label{roughpath}
		Let us briefly present the concept of rough paths in the simplest form, following \cite{frizhairer} and  \cite{lyons98}. 
		For any finite dimensional vector space $W$, denote by $C([a,b],W)$ the space of all continuous paths $y: [a,b] \to W$ equipped with the sup norm $\|\cdot\|_{\infty,[a,b]}$ given by $\|y\|_{\infty,[a,b]}=\sup_{t\in [a,b]} \|y_t\|$, 
		where $\|\cdot\|$ is the norm in $W$. We write $y_{s,t}:= y_t-y_s$. For $p\geq 1$, denote by $C^{p{\rm-var}}([a,b],W)\subset C([a,b],W)$ the space of all continuous paths $y:[a,b] \to W$ of finite $p$-variation $\ltn y\rtn_{\tp,[a,b]} :=\left(\sup_{\mathcal{P}([a,b])}\sum_{i=1}^n \|y_{t_i,t_{i+1}}\|^p\right)^{1/p} < \infty$, where the supremum is taken over the whole class of finite partitions of $[a,b]$. 
        
		Also for each $0<\alpha<1$, we denote by $C^{\alpha}([a,b],W)$ the space of H\"older continuous functions with exponent $\alpha$ on $[a,b]$ equipped with the norm
		\begin{equation}\label{holnorm}
			\|y\|_{\alpha,[a,b]}: = \|y_a\| + \ltn y\rtn_{\alpha,[a,b]},\quad \text{where} \quad \ltn y\rtn_{\alpha,[a,b]} :=\sup_{\substack{s,t\in [a,b],\ s<t}}\frac{\|y_{s,t}\|}{(t-s)^\alpha} < \infty.
		\end{equation}
		For $\alpha \in (\frac{1}{3},\frac{1}{2})$, a couple $\bx=(x,\X) \in \R^m \oplus (\R^m \otimes \R^m)$, where $x \in C^\alpha([a,b],\R^m)$ and 
		\begin{eqnarray*}
			 \X \in C^{2\alpha}([a,b]^2,\R^m \otimes  \R^m) 
			&:=& \left\{\X \in C([a,b]^2,\R^m \otimes  \R^m):  \sup_{\substack{s, t \in [a,b],\ s<t}} \frac{\|\X_{s,t}\|}{|t-s|^{2\alpha}} < \infty \right\}, 
		\end{eqnarray*}
		is called a {\it rough path} if it satisfies Chen's relation
		\begin{equation}\label{chen}
			\X_{s,t} - \X_{s,u} - \X_{u,t} = x_{s,u} \otimes  x_{u,t},\qquad \forall a \leq s \leq u \leq t \leq b. 
		\end{equation}
		We introduce the rough path semi-norm 
		\begin{eqnarray}\label{translated}
			\ltn \bx \rtn_{\alpha,[a,b]} &:=& \ltn x \rtn_{\alpha,[a,b]} + \ltn \X \rtn_{2\alpha,[a,b]^2}^{\frac{1}{2}},\\ 
			\text{where}\quad \ltn \X \rtn_{2\alpha,[a,b]^2}&:=& \sup_{s, t \in [a,b];s<t} \frac{\|\X_{s,t}\|}{|t-s|^{2\alpha}} < \infty. \notag 
		\end{eqnarray}
		Throughout this paper, we will fix parameters $\frac{1}{3}< \alpha < \nu \leq \frac{1}{2}$ and  $p = \frac{1}{\alpha}$ so that $C^\alpha([a,b],W) \subset C^{\tp}([a,b],W)$. We also set $q=\frac{p}{2}$ and consider the $\tp$ semi-norm 
		\begin{equation}\label{pvarnorm}
			\begin{split}
				\ltn \bx \rtn_{\tp,[a,b]} &:= \Big(\ltn x \rtn^p_{\tp,[a,b]} + \ltn \X \rtn_{\tq,[a,b]^2}^q\Big)^{\frac{1}{p}}, \\
				\ltn \X \rtn_{\tq,[a,b]^2} &:= \left(\sup_{\mathcal{P}([a,b])}\sum_{i=1}^n \|\X_{t_i,t_{i+1}}\|^q\right)^{1/q}, 
			\end{split}
		\end{equation}
		where the supremum is taken over the whole class of finite partitions $\mathcal{P}([a,b])$ of $[a,b]$. The reader is also referred to \cite[Chapter 10]{frizhairer} for a detailed construction of $\X$ of a multi-dimensional Gaussian process $X = (X_i)_{i = 1}^m$ with mutually independent components.  

  	\subsection{Probabilistic settings}\label{probset}
		Following \cite{duckloeden}, denote by $T^2_1(\R^m) = 1 \oplus \R^m \oplus (\R^m \otimes \R^m)$ the set with the tensor product
		\[
		(1,g^1,g^2) \otimes (1,h^1,h^2) = (1, g^1 + h^1, g^1 \otimes h^1 + g^2 +h^2),
		\]
		for all ${\bf g} =(1,g^1,g^2), {\bf h} = (1,h^1,h^2) \in T^2_1(\R^m)$.
		Then $(T^2_1(\R^m),\otimes)$ is a topological group with unit element ${\bf 1} = (1,0,0)$ and ${\bf g}^{-1} = (1,-g^1, g^1 \otimes g^1 - g^2)$. 
		
		Given $\alpha \in (\frac{1}{3},\nu)$, denote by $\cC^{0,\alpha}(I,T^2_1(\R^m))$ 
		the closure of $\cC^{\infty}(I,T^2_1(\R^m))$ in the H\"older space $\cC^{\alpha}(I,T^2_1(\R^m))$, and by 
		$\cC_0^{0,\alpha}(\R,T^2_1(\R^m))$ the space of all paths $ {\bf g}: \R\to T^2_1(\R^m))$ such that $ {\bf g}|_I \in \cC^{0,\alpha}(I, T^2_1(\R^m))$ for each compact interval $I\subset\R$ containing $0$. Then $\cC_0^{0,\alpha}(\R,T^2_1(\R^m))$ is equipped with the compact open topology given by the $\alpha$-H\"older norm \eqref{holnorm}, i.e  the topology generated by the metric 
		\[
		d_\alpha( {\bf g}, {\bf h}): = \sum_{k\geq 1} \frac{1}{2^k} (\| {\bf g}- {\bf h}\|_{\alpha,[-k,k]}\wedge 1). 
		\]
		
		Let us consider a stochastic process $\bar{\bX}$ defined on a probability space $(\bar{\Omega},\bar{\mathcal{F}},\bar{\bP})$  with realizations in $(\cC^{0,\alpha}_0(\R,T^2_1(\R^m)), \mathcal{F})$. Assume further that $\bar{\bX}$ has stationary increments. Assign $\Omega:=\cC_0^{0,\alpha}(\R,T^2_1(\R^m))$ and equip it with the Borel $\sigma$-algebra $\mathcal{F}$ and let $\bP$ be the law of $\bar{\bX}$. Denote by $\theta$ the {\it Wiener-type shift}
		\begin{equation}\label{shift}
			(\theta_t \omega)_\cdot = \omega_t^{-1}\otimes \omega_{t+\cdot},\forall t\in \R, \omega \in \cC^{0,\alpha}_0(\R,T^2_1(\R^m)),
		\end{equation}  
		and define the so-called {\it diagonal process}  $\bX: \R \times \Omega \to T^2_1(\R^m), \bX_t(\omega) = \omega_t$ for all $t\in \R, \omega \in \Omega$. Due to the stationarity of $\bar{\bX}$, it can be proved that $\theta$ is invariant under $\bP$, then forming a continuous (and thus measurable) dynamical system on $(\Omega, \mathcal{F},\bP)$ \cite[Theorem 5]{BRSch17}. 
	
		As pointed out in \cite[Remark 1]{duc21} and due to \cite[Corollary 9]{BRSch17}, the above construction is possible for $X_t$ to be a continuous, centered Gaussian process with stationary increments and independent components, satisfying: there exists for any $T>0$ a constant $C_T$ such that for all $p \geq \frac{1}{\bar{\nu}}$,
		$\E \|X_t- X_s\|^{p} \leq C_T |t-s|^{p\nu}$ for all $s,t \in [0,T]$. Then $\bX$ can be chosen to be the natural lift of $X$ in the sense of Friz-Victoir \cite[Chapter 15]{friz} with sample paths in the space $C_0^{0,\alpha}(\R,T^2_1(\R^m))$, for a certain $\alpha \in (0,\nu)$. 
		In particular, the Wiener shift \eqref{shift} implies that
		\begin{equation}\label{roughshift}
			\begin{split}
				\ltn \bx(\theta_h \omega) \rtn_{\tp,[s,t]} &= \ltn \bx(\omega) \rtn_{\tp,[s+h,t+h]}; \\
				N_{[s,t]}(\bx(\theta_h \omega)) &=  N_{[s+h,t+h]}(\bx(\omega)).		
			\end{split}
		\end{equation}
		
		As said above, in this paper, we need an assumption on ergodicity of $\theta$. 
		It is known (see \cite[Lemma 3]{GASch})  that if $X$ is a $m$-dimensional fractional Brownian motion with mutual independent components, we have 
		the ergodicity of $\theta$. 
		
		\begin{lemma}\label{ergodicity}
			Assume that $X= B^H$, then $\theta$ is ergodic.
		\end{lemma}
		\begin{proof}
			We sketch out a short proof here. For $H = \frac{1}{2}$, the canonical process w.r.t. the Wiener measure $\mP_{\frac{1}{2}}$ and Wiener shift $\theta^*_t \omega_{\cdot} = \omega_{t+\cdot}-\omega_{\cdot}$ on $\Omega^* = C^0_0(\R,\R)$ is ergodic. By \cite{GASch}, the Wiener shift $\eta_t x_\cdot = x_{t+\cdot}-x_t $ is ergodic on $\Omega^\prime = C^{0,\alpha}(\R,\R^m)$ w.r.t. $\mathbb{P}_H = B^H \mP_{\frac{1}{2}}$. Because of \cite[Theorem 10.4]{frizhairer}, there exists a full measure subset $\Omega_1 \subset\Omega^\prime$ such that $\omega = (1,x,\X) \in \Omega$ for any $x \in \Omega_1$. Moreover, by \cite[Theorem 15.42, 15.45]{friz}, one can choose this full measure subset $\Omega_1$ such that it satisfies the piece-wise linear approximations (mollifier approximation). 
			Then consider the natural lift $\mathcal{S}$ on smooth paths
			\begin{equation}
				\mathcal{S}(x)_{s,t} = (1,x_{s,t},\int_s^t x_{s,r}dx_r),
			\end{equation}
			which can be extended to $\Omega_1$ such that $\mathcal{S}: \Omega_1 \to \Omega$. One can now apply the arguments in \cite{BRSch17} to conclude that there exists a metric dynamical system $(\hat{\Omega},\hat{\mathcal{F}},\hat{\mathbb{P}},\hat{\theta})$ such that $\hat{\Omega} \subset \Omega$ and $\hat{\mathbb{P}}= \mathbb{P}^H \circ \mathcal{S}^{-1}$. Furthermore,
			\begin{equation}
				\theta_t \mathcal{S}(x) = \lim \limits_{n \to \infty} \hat{\theta}_t \mathcal{S}(x^{(n)}) = \lim \limits_{n \to \infty} \mathcal{S}(\eta_t x^{(n)}) = \mathcal{S}(\eta_t x).
			\end{equation}
			Since $\eta$ is ergodic, it follows from \cite[ Lemma 3]{GASch} that $\theta$ is also ergodic.
		\end{proof}	
		
		\subsection{Gubinelli's rough path integrals}\label{roughint}
		
		Following Gubinelli \cite{gubinelli}, a rough path integral can be defined for a continuous path $y \in C^\alpha([a,b],W)$ which is {\it controlled by} $x \in C^\alpha([a,b],\R^m)$ in the sense that, there exists a couple $(y^\prime,R^y)$ with $y^\prime \in C^\alpha([a,b],\cL(\R^m,W)), R^y \in C^{2\alpha}([a,b]^2,W)$ such that
		\begin{equation}\label{controlRP}
			y_{s,t} = y^\prime_s   x_{s,t} + R^y_{s,t},\qquad \forall a\leq s \leq t \leq b.
		\end{equation}
		$y^\prime$ is called the {\it Gubinelli derivative} of $y$.\\		
		Denote by $\cD^{2\alpha}_x([a,b])$ the space of all the couples $(y,y^\prime)$ controlled by $x$. Then for a fixed rough path $\bx = (x,\X)$ and any controlled rough path $(y,y^\prime) \in \cD^{2\alpha}_x ([a,b])$, the integral $\int_s^t y_u dx_u$ can be defined as the limit of the Darboux sum 
		\begin{equation*}
			\int_s^t y_u dx_u := \lim \limits_{|\Pi| \to 0} \sum_{[u,v] \in \Pi} \Big( y_{u}  \otimes  x_{u,v} + y^\prime_u   \X_{u,v} \Big)
		\end{equation*}
		where the limit is taken on all finite partitions $\Pi$ of $[a,b]$ with $|\Pi| := \displaystyle\max_{[u,v]\in \Pi} |v-u|$. Moreover, there exists a constant $C_p >1$ independent of $\bx$ and $(y,y^\prime)$ such that
		\begin{equation}\label{roughpvar}
				\Big\|\int_s^t y_u dx_u - y_s \otimes x_{s,t} - y^\prime_s \X_{s,t}\Big\|
				\leq  C_p \Big(\ltn x \rtn_{\tp,[s,t]} \ltn R^y \rtn_{\tq,[s,t]^2} + \ltn y^\prime\rtn_{\tp,[s,t]} \ltn \X \rtn_{\tq,[s,t]^2}\Big).
		\end{equation}

		\subsection{Stochastic integrals as rough integrals} \label{Ito}
		In general, a Gubinell rough integral $\int y dx$ is defined in the pathwise sense with respect to a driving path $x$, yet we can compare it to classical stochastic integrals in some special cases. Namely, let $B$ be a $m$-dimensional Brownian motion which is enhanced to an It\^o rough path $(B(\omega),\B^{\text{It\^o}}(\omega)) \in \cC^\alpha$ for any $\alpha \in (\frac{1}{3},\frac{1}{2})$ a.s. Assume $(y(\omega),y^\prime(\omega)) \in \cD^{2\alpha}_{B(\omega)}$ a.s. Then the rough integral 
		\[
		\int_I y_r dB^{\text{It\^o}}_r = \lim \limits_{|\Pi| \to 0} \sum_{[u,v] \in \Pi} \Big(y_u B_{u,v} + y^\prime_u \B^{\text{It\^o}}_{u,v}\Big) 
		\]
		exists a.s. If $y, y^\prime$ are adapted, then $\int_I y_r d^{\text{It\^o}}_r = \int_I y_r dB_r$ a.s. where the latter is the It\^o integral. The same conclusions also hold for $(B, \B^{\text{Strat}})$ and the corresponding Stratonovich integral (see \cite{frizhairer}).

      \subsection{Discrete rough paths}\label{discRP}
        Firstly, we recall some notation from \cite{congduchong23}. For a given function $y$ defined on finite set $\Pi[a,b] := \Pi\cap[a,b]$, where $a,b\in \Pi$, we define $R^y_{s,t}:= y_{s,t} - g(y_s)x_{s,t},\ s<t\in \Pi$ and introduce the following quantities
	\begin{equation*}
    \begin{split}
        \|y\|_{\infty,\Pi[a,b]} &: = \sup_{t_i\in \Pi[a,b]} \|y_{t_i}\|;\quad \| R^y\|_{\infty,\Pi[a,b]} :=  \sup_{s<t\in \Pi[a,b]}\|R^y_{s,t} \|; \\
		\ltn y\rtn_{\tp  ,\Pi[a,b]}  &:= \sup_{t_i^*\in\Pi[a,b],\; 0\leq i\leq r,\; t_0^*<t_1^*\ldots<t_r^*,\; r\leq n}\left(\sum_{i=0}^{r-1} \|y_{t_i^*}-y_{t_{i+1}^*}\|^p\right)^{1/p};\\
        \ltn R^y\rtn_{\tq,\Pi[a,b]} &: = \sup_{t_i^*\in\Pi[a,b],\; 0\leq i\leq r,\; t_0^*<t_1^*\ldots<t_r^*,\ r\leq n}\left(\sum_{i=0}^{r-1} |R^y_{t^*_i,t^*_{i+1}}|^p\right)^{1/p},
    \end{split}
		\end{equation*}
and  
       \begin{equation}\label{yRynorm}
       \begin{split}
          \ltn y,R^y\rtn_{\tp  ,\Pi[a,b]} &:=  \ltn y\rtn_{\tp,\Pi[a,b]} \vee \ltn R^y\rtn_{\tq ,\Pi[a,b]}  ;\\
            \| y,R^y\|_{\tp  ,\Pi[a,b]} &= \|y_a\|+\ltn y,R^y\rtn_{\tp  ,\Pi[a,b]}.            
       \end{split}
       \end{equation}  

		\subsection{Proofs}\label{proofs}
        \begin{proof}[{\bf Proposition \ref{solest}}] 
			We follow the proofs in  \cite[Proposition 2.1]{duc21} and \cite[Proposition 2.3]{duckloeden} with a small modification regarding to parameters.  In this proof we omit notation $[a,b]$ in superemum and $p-$var norms.
            
            $(i)$, First observe that $[g(\phi)]^\prime = Dg(\phi)g(\phi)$ so that 
			\[
			\|[g(\phi)]^\prime\|_\infty \leq \|Dg\|_\infty \|g\|_\infty \quad \text{and}\quad \ltn [g(\phi)]^\prime\rtn_{\tp} \leq (\|D^2g\|_\infty \|g\|_\infty + \|Dg\|_\infty^2) \ltn \phi \rtn_{\tp}.
			\]
		By the definition of Gubinelli derivative \eqref{controlRP}, since $\phi'_s=g(\phi_s)$ we have
        \begin{eqnarray*}
        R^{g(\phi)}_{s,t} &=& g(\phi)_{s,t}-
        Dg(\phi_s)g(\phi_s)x_{s,t}
        = \int_0^1 Dg(\theta \phi_t +(1-\theta)\phi_s)
        \phi_{s,t} d\theta - Dg(\phi_s)g(\phi_s)x_{s,t}\\
        &=&
         \int_0^1 Dg(\theta \phi_t +(1-\theta)\phi_s)
        (\phi'_s x_{s,t} + R^\phi_{s,t}) d\theta - Dg(\phi_s)g(\phi_s)x_{s,t}\\
        &=&
         \int_0^1 Dg(\theta \phi_t +(1-\theta)\phi_s)
        R^\phi_{s,t} d\theta
        +\int_0^1 \Big[Dg(\theta \phi_t +(1-\theta)\phi_s)-Dg(\phi_s)\Big] g(\phi_s)x_{s,t} d\theta\\
        &=&
         \int_0^1 Dg(\theta \phi_t +(1-\theta)\phi_s)
        R^\phi_{s,t} d\theta
     +   \int_0^1\Big[ \int_0^1 D^2g(\eta\theta\phi_t + (1-\eta\theta)\phi_s)\theta\phi_{s,t} d\eta\Big]g(\phi_s) x_{s,t} d\theta.
        \end{eqnarray*}
        This implies
        \[
        \|R^{g(\phi)}_{s,t}\| \leq
        \|Dg\|_\infty \|R^\phi_{s,t}\|
        + \frac{1}{2} \|D^2g\|_\infty \|g\|_\infty \|\phi_{s,t}\| \|x_{s,t}\|.
        \]
          Consequently, by Minkowski and Cauchy-Schwarz inequalities we have
          \[
          \ltn R^{g(\phi)}\rtn_{\tq} \leq \|Dg\|_\infty \ltn R^\phi \rtn_{\tq} + \frac{1}{2}\|D^2g\|_\infty  \|g\|_\infty \ltn x\rtn_{\tp}\ltn \phi\rtn_{\tp}.
          \]
As a result, it follows from $\phi_{s,t} = \int_s^t g(\phi_u)dx_u$ that 
			\begin{eqnarray}\label{phipvar} 
				\ltn \phi \rtn_{\tp} &\leq& \|g\|_\infty \ltn x\rtn_{\tp} + \|Dg\|_\infty \|g\|_\infty  \ltn \X \rtn_{\tq} \notag\\
				&&+\; C_p \Big\{\ltn x \rtn_{\tp} \ltn R^{g(\phi)} \rtn_{\tq} + \ltn \X \rtn_{\tq} \ltn [g(\phi)]^\prime \rtn_{\tp}  \Big\} \notag\\
				&\leq& \|g\|_\infty \ltn x\rtn_{\tp} + \|Dg\|_\infty \|g\|_\infty  \ltn \X \rtn_{\tq} + C_p \|Dg\|_\infty\ltn x \rtn_{\tp} \ltn R^\phi \rtn_{\tq} \\
				&& +\; C_p\Big[(\|D^2g\|_\infty \|g\|_\infty + \|Dg\|_\infty^2) \ltn \X \rtn_{\tq} + \frac{1}{2}\|D^2g\|_\infty \|g\|_\infty \ltn x \rtn_{\tp}^2\Big]\ltn \phi\rtn_{\tp}\notag; 
			\end{eqnarray}
			and similarly, since $R^\phi_{s,t} = \phi_{s,t} - \phi'_s x_{s,t}=\int_s^tg(\phi_u)dx_u - g(\phi_s)x_{s,t}$ we have
			\begin{eqnarray}\label{Rphiqvar} 
				\ltn R^\phi \rtn_{\tq} 	&\leq& \|Dg\|_\infty \|g\|_\infty  \ltn \X \rtn_{\tq} + C_p \|Dg\|_\infty\ltn x \rtn_{\tp} \ltn R^\phi \rtn_{\tq}\\
				&&+\; C_p\Big[(\|D^2g\|_\infty \|g\|_\infty + \|Dg\|_\infty^2) \ltn \X \rtn_{\tq} + \frac{1}{2}\|D^2g\|_\infty \|g\|_\infty \ltn x \rtn_{\tp}^2\Big]\ltn \phi\rtn_{\tp}.\notag
			\end{eqnarray}
			Then it follows from \eqref{phipvar} and \eqref{Rphiqvar} that
            \begin{eqnarray}\label{phi+R}
			&&	\ltn \phi \rtn_{\tp}\vee\ltn R^\phi \rtn_{\tq}\notag \\
            &\leq& \|g\|_{\infty} \ltn x\rtn_{\tp} +\|Dg\|_\infty \|g\|_\infty  \ltn \X\rtn_{\tq} \notag\\
                &&+  C_p\left[\|Dg\|_\infty\ltn x \rtn_{\tp} +(\|D^2g\|_\infty \|g\|_\infty + \|Dg\|_\infty^2) \ltn \X \rtn_{\tq} + \frac{1}{2}\|D^2g\|_\infty \|g\|_\infty \ltn x \rtn_{\tp}^2\right]\times\notag\\
                &&\times \left(\ltn \phi \rtn_{\tp}\vee\ltn R^\phi \rtn_{\tq}\right)\notag\\
                &\leq& \|g\|_{\infty} \ltn x\rtn_{\tp} +\|Dg\|_\infty \|g\|_\infty \ltn \X\rtn_{\tq} +  (\lambda+\dfrac{5}{2}\lambda^2) (\ltn \phi \rtn_{\tp}\vee\ltn R^\phi \rtn_{\tq})
            \end{eqnarray}
 Taking into account the fact that $\lambda = C_p C_g \ltn \bx \rtn_{\tp}\leq 1/8$, we have 
 \begin{eqnarray}
 \ltn \phi \rtn_{\tp}\vee \ltn R^\phi \rtn_{\tq}&\leq& \frac{1}{1-\lambda-\dfrac{5}{2}\lambda^2}(\|g\|_\infty \ltn x \rtn_{\tp}+\|Dg\|_\infty \|g\|_\infty \ltn \X\rtn_{\tq} )\notag\\
 &\leq& 2(\|g\|_\infty \ltn x \rtn_{\tp}+ C^2_g\ltn \X\rtn_{\tq} )\notag\\
 &\leq& 2(\|g\|_\infty \ltn x \rtn_{\tp}+ \lambda^2).
 \end{eqnarray}
Replacing the above estimate into \eqref{Rphiqvar} yields
    \begin{eqnarray*}
\ltn R^\phi \rtn_{\tq} 
    &\leq& 
     \lambda^2+ \lambda\ltn R^\phi \rtn_{\tq}+ C_p\Big[(\|D^2g\|_\infty \|g\|_\infty + \|Dg\|_\infty^2) \ltn \X \rtn_{\tq} \\
     && \hspace{4cm}+ \frac{1}{2}\|D^2g\|_\infty \|g\|_\infty \ltn x \rtn_{\tp}^2\Big] 2(\|g\|_\infty \ltn x \rtn_{\tp}+\lambda^2)  \\
     &\leq& \lambda^2+ \lambda\ltn R^\phi \rtn_{\tq}+ (\frac{5}{2}\lambda^2) 2 (\lambda+ \lambda^2)
    \end{eqnarray*}
which reduces to
    \[
      \ltn R^\phi \rtn_{\tq} \leq \frac{1}{1-\lambda} \Big[\lambda^2+5\lambda^2 (\lambda + \lambda^2)\Big]\notag \leq 2\lambda^2.
	\]
Hence we obtain \eqref{solest1a}.\\

 In case $g(0) = 0$, one rewrites 
			\[
			\|g(\phi)\|_{\infty,[s,t]} \leq \|g(0)\| + \|Dg\|_\infty \|\phi\|_{\infty,[s,t]} \leq \|Dg\|_\infty (\|\phi_s\| + \ltn \phi \rtn_{\tp,[s,t]}).
			\] 
Then \eqref{phi+R} can be rewritten as
        \begin{eqnarray*}
				\ltn \phi \rtn_{\tp}\vee\ltn R^\phi \rtn_{\tq} &\leq&(C_g \ltn x\rtn_{\tp} +C^2_g \ltn \X\rtn_{\tq})\|\phi_a\| +  (2\lambda+\dfrac{7}{2}\lambda^2) (\ltn \phi \rtn_{\tq}\vee\ltn R^\phi \rtn_{\tq}).
            \end{eqnarray*}
This deduces to  
    \[
  \ltn \phi \rtn_{\tp} \leq  \ltn \phi \rtn_{\tp} \vee\ltn R^\phi \rtn_{\tq} \leq \frac{\lambda + \lambda^2}{1-2\lambda - \frac{7}{2}\lambda^2}\|\phi_a\| \leq 2\lambda \|\phi_a\|
    \]
and to
    \[
      \ltn R^\phi \rtn_{\tq} \leq \frac{1}{1-\lambda} \Big[\lambda^2\|\phi_a\| +\frac{7}{2}\lambda^2 \frac{\lambda+\lambda^2}{1-\lambda -\frac{7}{2}\lambda^2} \|\phi_a\|) \Big]\notag \leq 2\lambda^2\|\phi_a\|. 
	\]
Thus \eqref{solest1} is proved.\\
		
	$(ii)$, Next, observe that $[g(\bar{\phi})-g(\phi)]^\prime = Dg(\bar{\phi})g(\bar{\phi}) - Dg(\phi)g(\phi)$. Applying estimates in $(i)$ for $\phi, \bar{\phi}, R^\phi, R^{\bar{\phi}} $ we have
       \allowdisplaybreaks
	\begin{eqnarray*}
				&&\|[g(\bar{\phi})-g(\phi)]^\prime \|_\infty \leq(\|D^2g\|_\infty \|g\|_\infty + \|Dg\|^2_\infty) \|\bar{\phi}-\phi\|_\infty  ;\\
				&&\ltn [g(\bar{\phi})-g(\phi)]^\prime\rtn_{\tp}\\
                &&\hspace{1cm}\leq (\|D^2g\|_\infty \|g\|_\infty + \|Dg\|^2_\infty) \ltn\bar{\phi}-\phi \rtn_{\tp} \\
				&&\hspace{1cm}+\; (\|D^3g\|_\infty \|g\|_\infty + 3\|D^2g\|_\infty \|Dg\|_\infty)\frac{1}{2}( \ltn \bar\phi\rtn_{\tp}+ \ltn \phi\rtn_{\tp}) \|\bar{\phi}-\phi\|_\infty \\
				&&\hspace{1cm}\leq  2C^2_g  \ltn\bar{\phi}-\phi \rtn_{\tp}\\
                &&\hspace{1cm}+  2 (\|D^3g\|_\infty \|g\|_\infty + 3\|D^2g\|_\infty \|Dg\|_\infty) (\|g\|_\infty \ltn x\rtn_{\tp}+ \|Dg\|_\infty\|g\|_\infty \ltn \X\rtn_{\tq})\|\bar{\phi}-\phi\|_\infty \\
                &&\hspace{1cm}\leq  2C^2_g  \ltn\bar{\phi}-\phi \rtn_{\tp}+  8 (C^3_g  \ltn x\rtn_{\tp}+C^4_g \ltn \X\rtn_{\tq})\|\bar{\phi}-\phi\|_\infty \\
                &&\hspace{1cm}\leq  2C^2_g  \ltn\bar{\phi}-\phi \rtn_{\tp}+  8[ C^3_g  \ltn \bx\rtn_{\tp} + C^4_g  \ltn \bx\rtn^2_{\tp}]\|\bar{\phi}-\phi\|_\infty; \\
   			&&\ltn R^{g(\bar{\phi})-g(\phi)}\rtn_{\tq} \\
                &&\hspace{1cm}\leq \|Dg\|_\infty \ltn R^{\bar{\phi}-\phi}\rtn_{\tq} + \|D^2g\|_\infty \|g\|_\infty \ltn x\rtn_{\tp}  \ltn\bar{\phi}-\phi \rtn_{\tp} \\
				&&\hspace{1cm}+\; \Big[\|D^2g\|_\infty \ltn R^\phi\rtn_{\tq} + \|D^2g\|_\infty \|Dg\|_\infty \ltn x\rtn_{\tp} \ltn \bar{\phi} \rtn_{\tp} \\
				&&\hspace{1cm}  +\; \|D^3g\|_\infty \|g\|_\infty\ltn x\rtn_{\tp} (\ltn \bar\phi \rtn_{\tp} + \ltn \phi \rtn_{\tp}) \Big] \|\bar{\phi}-\phi\|_\infty\\
	             &&\hspace{1cm}\leq C_g \ltn R^{\bar{\phi}-\phi}\rtn_{\tq} +C^2_g \ltn x\rtn_{\tp} \ltn\bar{\phi}-\phi \rtn_{\tp}\\
                &&\hspace{1cm}+\Big( C_g +3 C^2_g\ltn x\rtn_{\tp} \Big) 2\Big[\|g\|_\infty \ltn x\rtn_{\tp}+C^2_g\ltn \X\rtn_{\tp}\Big]   \|\bar{\phi}-\phi\|_\infty  \notag\\
                 &&\hspace{1cm} \leq C_g \ltn R^{\bar{\phi}-\phi}\rtn_{\tq} +C^2_g \ltn \bx\rtn_{\tp} \ltn\bar{\phi}-\phi \rtn_{\tp}\notag\\
                 &&\hspace{1cm}+2[ C^2_g \ltn \bx\rtn_{\tp}  +4C^3_g \ltn \bx\rtn^2_{\tp} + 3C^4_g \ltn \bx\rtn^3_{\tp}  ] \|\bar{\phi}-\phi\|_\infty.
 		\end{eqnarray*}
			As a result, it follows from $\bar{\phi}_{s,t}-\phi_{s,t} = \int_s^t [g(\bar{\phi}_u)-g(\phi_u)]dx_u$ that
			\begin{eqnarray}\label{diffphipvar} 
				\ltn \bar{\phi}- \phi \rtn_{\tp} &\leq&  \|Dg\|_\infty \ltn x \rtn_{\tp}\|\bar{\phi} - \phi\|_\infty + (\|D^2g\|_\infty \|g\|_\infty + \|Dg\|^2_\infty) \ltn \X \rtn_{\tq} \|\bar{\phi} - \phi\|_\infty \notag\\
				&&+\; C_p \Big[\ltn x\rtn_{\tp} \ltn R^{g(\bar{\phi})-g(\phi)}\rtn_{\tq}+ \ltn \X \rtn_{\tq}   \ltn [g(\bar{\phi})-g(\phi)]^\prime\rtn_{\tp}\Big] \notag\\
                &\leq& \Big[C_g \ltn \bx\rtn_{\tp}+2C^2_g \ltn \bx\rtn^2_{\tp}\Big]\|\bar{\phi} - \phi\|_\infty\notag\\
                &&\hspace{0cm}+ C_pC_g \ltn \bx\rtn_{\tp} \Big[\ltn R^{\bar{\phi}-\phi}\rtn_{\tq}+ C_g \ltn \bx\rtn_{\tp} \ltn\bar{\phi}-\phi \rtn_{\tp} + \Big.\notag\\
                && \Big. +2\Big(C_g \ltn \bx\rtn_{\tp} +4C^2_g \ltn \bx\rtn^2_{\tp}+3C^3_g \ltn \bx\rtn^3_{\tp}      \Big)\| \bar{\phi}-\phi\|_\infty\Big] \notag\\
                 &&\hspace{0cm}+ 2C_pC^2_g \ltn \bx\rtn^2_{\tp} \Big[   \ltn\bar{\phi}-\phi \rtn_{\tp}  + \big( 4C_g \ltn \bx\rtn_{\tp} + 4C^2_g \ltn \bx\rtn^2_{\tp}  \Big)\|\bar{\phi}-\phi\|_\infty\Big] \notag\\
				&\leq& (\lambda+ 4\lambda^2+16\lambda^3+14\lambda^4)\|\bar{\phi} - \phi\|_\infty+(\lambda+ 3\lambda^2) \Big(\ltn\bar{\phi}-\phi \rtn_{\tp} \vee \ltn R^{\bar{\phi}-\phi}\rtn_{\tq}\Big) \notag\\
				&\leq& \lambda(1+4\lambda+16\lambda^2 + 14\lambda^3)  \|\bar{\phi}_a - \phi_a\| \notag\\
                &&+(2\lambda+7\lambda^2+16\lambda^3+14\lambda^4)  ( \ltn\bar{\phi}-\phi \rtn_{\tp}\vee  \ltn R^{\bar{\phi}-\phi}\rtn_{\tq} ).
			\end{eqnarray} 
It is easily seen that $\ltn R^{\bar{\phi}- \phi} \rtn_{\tq}$ furnishes the same estimate. 
Consequently,
\[
 \Big( \ltn\bar{\phi}-\phi \rtn_{\tp}\vee  \ltn R^{\bar{\phi}-\phi}\rtn_{\tq} \Big) \leq \frac{\lambda(1+4\lambda+16\lambda^2 + 14\lambda^3) }{1-(2\lambda+7\lambda^2+16\lambda^3+14\lambda^4)} \|\bar{\phi}_a - \phi_a\|\leq 4
\lambda\|\bar{\phi}_a - \phi_a\|
\]
		\begin{equation}\label{phidiff_new}
			\ltn \bar{\phi}- \phi \rtn_{\tp}  \leq 4 \lambda \|\bar{\phi}_a - \phi_a\| \quad \text{and}\quad 	\ltn R^{\bar{\phi}- \phi} \rtn_{\tq} \leq  4 \lambda \|\bar{\phi}_a - \phi_a\|.
			\end{equation}
			One then uses the similar arguments to \cite[Proposition 2.1]{duc21} to prove \eqref{solest3}.		\\    
\end{proof}
        
\begin{proof}[{\bf  Proposition \ref{solestdiff}}]
    $i$,	The proof for the first estimate is simple, since one can write $\teta_t - \eta_t$ in the form 
		\begin{equation}\label{etadiff}
		\teta_t - \eta_t = \int_0^t [g(\tphi_u(\bx,\tz_t)) - g(\phi_u(\bx,z_t))] dx_u.
		\end{equation}
The first estimate in \eqref{etadiff} is then a direct consequence of \eqref{phidiff_new}.

      $ii$, The proof simply goes line by line with \cite[Proposition 4]{duc21} but using the estimates in the proof of Proposition \ref{solest}. By a direct computation, one can choose $K(\lambda):= 256 \lambda$. 
      Due to its similarity and length, it will be omitted here. \\
\end{proof}

\begin{proof}[{\bf Proof of Proposition \ref{hnew}}]
First, by repeating the computations in \cite[Section \ 3, Proposition 5.1]{congduchong23} we obtain
    \begin{eqnarray*}
    &&\|(\delta P)_{s,u,t}\| := \|P_{s,t} - P_{s,u}-P_{u,t} \| \\
    &&\leq \left\| \int_0^1 \left[ Dg(y^\Delta_s+\eta y^\Delta_{s,u})R^{y^\Delta}_{s,u}- Dg(\ta^\Delta_s+\eta \ta^\Delta_{s,u})R^{\ta^\Delta}_{s,u}\right]  d\eta \right\|\|x_{u,t}\|\\
    &&+\left\| \int_0^1 \left[Dg(y^\Delta_s+\eta y^\Delta_{s,u})g(y^\Delta_s)-Dg(y^\Delta_s)g(y^\Delta_s)+Dg(\ta^\Delta_s+\eta \ta^\Delta_{s,u}) g(\ta^\Delta_s)-Dg(\ta^\Delta_s)g(\ta^\Delta_s)\right] d\eta \right\|\times \\
    &&\qquad \times \|x_{s,u}\otimes x_{u,t}\| \\
    &&+ \left\| Dg(y^\Delta_s)g(y^\Delta_s)-Dg(y^\Delta_u)g(y^\Delta_u)-Dg(\ta^\Delta_s)g(\ta^\Delta_s)+Dg(\ta^\Delta_u)g(\ta^\Delta_u)\right\|\|\X_{u,t}\| \\
    &&=: A \|x_{u,t}\|+B \|x_{s,u}\otimes x_{u,t}\|+C\| \X_{u,t}\|,
    \end{eqnarray*}
 in which $A, B, C$ in the right hand side can be estimated as follows by using the Mean value theorem. Namely, put
\[
L_g^* = \max\Bigg\{\|Dg\|_{\infty} \|D^2g\|_\infty,  \|g\|_\infty  \|D^3g\|_\infty)\Bigg\}^{\frac{1}{2}},
\]
we obtain the estimates
         \allowdisplaybreaks
    \begin{eqnarray*}
        A&=& \left\| \int_0^1 \left[ Dg(y^\Delta_s+\eta y^\Delta_{s,u})R^{y^\Delta-\ta^\Delta}_{s,u}- \Big(Dg(y^\Delta_s+\eta y^\Delta_{s,u})-Dg(\ta^\Delta_s+\eta \ta^\Delta_{s,u})\Big)R^{\ta^\Delta}_{s,u}\right]  d\eta \right\|\\
        &\leq&  \|Dg\|_\infty  R^{y^\Delta-\ta^\Delta}_{s,u} +{\color{black} {\|D^2g\|_\infty  \|R^{\ta^\Delta}_{s,u}\| }} \|h\|_{\infty,[s,u]}  \\
        &\leq&  \|Dg\|_\infty  \ltn R^{y^\Delta-\ta^\Delta}\rtn_{\tq,\Pi[s,u]} +\|D^2g\|_\infty  \ltn R^{\ta^\Delta}\rtn_{\tq,\Pi[s,u]}  \|h\|_{\infty,[s,u]} \\
        &\leq& \|Dg\|_\infty  \ltn  R^{h}\rtn_{\tq,\Pi[s,u]} + \|D^2g\|_\infty \|h\|_{\infty,[s,u]}\ltn R^{\ta^\Delta}\rtn_{\tq,\Pi[s,u]};
    \end{eqnarray*}
    \begin{eqnarray*}
        B&= & \left\| \int_0^1\eta \left[\int_0^1 D^2g(y^\Delta_s+\xi\eta y^\Delta_{s,u}) y^\Delta_{s,u}  g(y^\Delta_s)d\xi- \int_0^1  D^2g(\ta^\Delta_s+\xi\eta \ta^\Delta_{s,u})\ta^\Delta_{s,u} g(\ta^\Delta_s) d\xi \right] d\eta \right\|\\
        &= & \left\| \int_0^1 \eta\left[\int_0^1 D^2g(y^\Delta_s+\xi\eta y^\Delta_{s,u}) (y^\Delta_{s,u}-\ta^\Delta_{s,u})  g(y^\Delta_s)d\xi \right.\right.\\
        &&+ \left. \left.\int_0^1  D^2g(y^\Delta_s+\xi\eta y^\Delta_{s,u})\ta^\Delta_{s,u} \left(g(y^\Delta_s)-g(\ta^\Delta_s) \right)d\xi  \right.\right.\\
        &&+ \left. \left.\int_0^1  \left(D^2g(y^\Delta_s+\xi\eta y^\Delta_{s,u})-D^2g(\ta^\Delta_s+\xi\eta \ta^\Delta_{s,u})\right)\ta^\Delta_{s,u} g(\ta^\Delta_s) d\xi \right] d\eta \right\|\\
        &\leq&\dfrac{1}{2}\|g\|_\infty.\|D^2g\|_\infty\ltn h\rtn_{\tp,\Pi[s,u]}    +  \dfrac{1}{2} ( \|D^2g\|_\infty.\|Dg\|_\infty  +\|D^3g\|_\infty.\|g\|_\infty) \|h\|_{\infty,\Pi[s,u]}\ltn \ta^\Delta\rtn_{\tp,\Pi[s,u]}  \\
         &\leq&   \dfrac{1}{2} L^2_g\ltn h\rtn_{\tp,\Pi[s,u]}    +  (L^*_g)^2 \|h\|_{\infty,\Pi[s,u]}\ltn \ta^\Delta\rtn_{\tp,\Pi[s,u]}  ;
          \end{eqnarray*}
and finally, by assigning $Q(\cdot):= Dg(\cdot)g(\cdot)$,
    \begin{eqnarray*}
        C &=& \left\|\int_0^1 DQ (y^\Delta_s+\eta y^\Delta_{s,u})y^\Delta_{s,u} d\eta -\int_0^1 DQ (\ta^\Delta_s+\eta \ta^\Delta_{s,u})a^\Delta_{s,u} d\eta   \right\|\\
         &=& \left\|\int_0^1 DQ (y^\Delta_s+\eta y^\Delta_{s,u})(y^\Delta_{s,u}-\ta^\Delta_{s,u}) d\eta +\int_0^1\Big[DQ (y^\Delta_s+\eta y^\Delta_{s,u})- DQ (\ta^\Delta_s+\eta \ta^\Delta_{s,u})\Big]\ta^\Delta_{s,u} d\eta   \right\|\\
          &\leq& (\|D^2g\|_\infty\|g\|_\infty.+\|Dg\|^2_\infty) \ltn h\rtn_{\tp,\Pi[s,u]}+\|D^2Q\|_\infty \|h\|_{\infty,[\Pi[s,u]]}\ltn \ta^\Delta\rtn_{\tp,\Pi[s,u]}\\
        &\leq& (\|D^2g\|_\infty\|g\|_\infty.+\|Dg\|^2_\infty) \ltn h\rtn_{\tp,\Pi[s,u]}\\
        &&+(3\|Dg\|_\infty.\|D^2g\|_\infty+\|g\|_\infty.\|D^3g\|_\infty) \|h\|_{\infty,\Pi[s,u]} \ltn \ta^\Delta\rtn_{\tp,\Pi[s,u]}\\
         &\leq&  2L^2_g \ltn h\rtn_{\tp,\Pi[s,u]} + 4(L^*_g)^2\|h\|_{\infty,\Pi[s,u]}\ltn \ta^\Delta\rtn_{\tp,\Pi[s,u]}  .
    \end{eqnarray*}
These estimates lead to
    \begin{eqnarray}\label{omegaP}
    &&\|(\delta P)_{s,u,t} \| \\
    &\leq& L_g \ltn x\rtn_{\tp,\Pi[s,t]} \ltn  R^h \rtn_{\tq,\Pi[s,t]}  + L_g^2\Big(\dfrac{1}{2}\ltn x\rtn^2_{\tp,\Pi[s,t]}+ 2  \ltn \X\rtn_{\tq,\Pi[s,t]}\Big) \ltn h \rtn_{\tp,\Pi[s,t]  }\notag\\
    &&+\left[\|D^2_g\|_\infty \ltn x\rtn_{\tp,\Pi[s,t]}  \ltn R^{\ta^\Delta} \rtn_{\tq,\Pi[s,t]}\right.\notag\\
    &&\qquad \left.+ (L^*_g)^2\Big( \ltn x\rtn^2_{\tp,\Pi[s,t]} +4 \ltn \X\rtn_{\tq,\Pi[s,t]}  \Big) \ltn \ta^\Delta \rtn_{\tp,\Pi[s,t]} \right] \|h\|_{\infty,\Pi[s,t]}  =: W_{s,t}(\delta P). \notag
    \end{eqnarray} 
 Thus, \cite[Corollary 2.4]{congduchong23} implies
		 \begin{eqnarray}
			\ltn h-P \rtn_{\tp,\Pi[a,b]}& \leq& L_f\|h\|_{\infty,\Pi[a,b]} (b-a) +  C_p W_{a,b}(\delta P);\notag
		\end{eqnarray}	
and by \eqref{Pvsh}
	\begin{eqnarray*}\label{h_new}
		  \ltn h \rtn_{\tp,\Pi[a,b]}
          &\leq& \ltn P \rtn_{\tp,\Pi[a,b]}+ L_f\|h\|_{\infty,\Pi[a,b]} (b-a) +  C_p W_{a,b}(\delta P) \notag\\
          &\leq&  \Big[ L_f(b-a) +L_g\ltn \bx\rtn_{\tp,\Pi[a,b]}+ 2 L^2_g\ltn \bx\rtn^2_{\tp,\Pi[a,b]}\Big]\|h\|_{\infty,\Pi[a,b]}  +  C_p W_{a,b}(\delta P) \notag\\
           &\leq& \Big[ L_f(b-a) +L_g\ltn \bx\rtn_{\tp,\Pi[a,b]}+ 2 L^2_g\ltn \bx\rtn^2_{\tp,\Pi[a,b]}\Big]\|h\|_{\infty,\Pi[a,b]} \notag\\
           &&+ C_p\Big[L_g\ltn x\rtn_{\tp,\Pi[a,b]}+ L^2_g\Big(\dfrac{1}{2}\ltn x\rtn^2_{\tp,\Pi[a,b]}+ 2  \ltn \X\rtn_{\tq,\Pi[a,b]} \Big)  \Big] \ltn h,R^h\rtn_{\tp,\Pi[a,b]}    \notag\\
            &&+C_p \Big[ \|D^2_g\|_\infty\ltn x\rtn_{\tp,\Pi[a,b]}\ltn R^{\ta^\Delta}\rtn_{\tp,\Pi[a,b]}+ \Big.\notag\\
              &&\hspace{2cm} \Big. + (L^*_g)^2\Big(\ltn x\rtn^2_{\tp,\Pi[a,b]}+ 4  \ltn \X\rtn_{\tq,\Pi[a,b]} \Big)\ltn \ta^\Delta\rtn_{\tp,\Pi[a,b]} \Big]   \|h\|_{\infty,\Pi[a,b]} \notag\\ 
 	   &\leq&\left[L_f (b-a)+4C_pL_g \ltn \bx\rtn_{\tp,\Pi[a,b]}+C_p\|D^2_g\|_\infty \ltn x\rtn_{\tp,\Pi[a,b]} \ltn R^{\ta^\Delta}\rtn_{\tp,\Pi[a,b]}\right.  \notag\\
           && \hspace{0cm}\left. +C_p(L^*_g)^2( \ltn x\rtn^2_{\tp,\Pi[a,b]} +4\ltn \X\rtn_{\tp,\Pi[a,b]})\ltn \ta^\Delta\rtn_{\tp,\Pi[a,b]}\right]  \times \| h,R^h\|_{\tp,\Pi[a,b]}.
	\end{eqnarray*}
The same estimate applies for $R^h$.  Therefore,
     \begin{eqnarray*}
		  &&\ltn h,R^h \rtn_{\tp,\Pi[a,b]}\\
          &\leq&\left[L_f (b-a)+4C_pL_g \ltn \bx\rtn_{\tp,\Pi[a,b]}+C_p\|D^2_g\|_\infty \ltn x\rtn_{\tp,\Pi[a,b]} \ltn R^{\ta^\Delta}\rtn_{\tp,\Pi[a,b]}\right.  \notag\\
           && \hspace{0cm}\left. +C_p(L^*_g)^2( \ltn x\rtn^2_{\tp,\Pi[a,b]} +4\ltn \X\rtn_{\tp,\Pi[a,b]})\ltn \ta^\Delta\rtn_{\tp,\Pi[a,b]}\right]  \times \| h,R^h\|_{\tp,\Pi[a,b]}.
	\end{eqnarray*}
Now we modify the estimate in formula (3.7) in \cite[Theorem\ 3.3]{congduchong23}) to have
		\begin{eqnarray*}
			&&\ltn \ta^\Delta,R^{\ta^\Delta} \rtn_{\tp,\Pi[a,b]}\\
            &\leq& \|f(\ta^\Delta)\|_{\infty,\Pi[a,b]}(b-a) +  \|g\|_\infty \ltn x \rtn_{\tp,\Pi[a,b]}+ \|g\|_\infty \|Dg\|_\infty\ltn \X \rtn_{\tq,\Pi[a,b]} \notag\\
			&&+C_p\ltn \ta^\Delta,R^{\ta^\Delta} \rtn_{\tp,\Pi[a,b]} \left[\|Dg\|_\infty \ltn x \rtn_{\tp,\Pi[a,b]}+ \|D^2g\|_\infty\|g\|_\infty \ltn x \rtn^2_{\tp,\Pi[a,b]} \right.\\
            &&\hspace{5cm}\left. +(\|D^2g\|_\infty\|g\|_\infty + \|Dg\|^2_\infty)\ltn \X \rtn_{\tq,\Pi[a,b]}\right]\\
            &\leq& \|f(\ta^\Delta)\|_{\infty,\Pi[a,b]}(b-a)+  \|g\|_\infty \ltn x \rtn_{\tp,\Pi[a,b]}+ \|g\|_\infty \|Dg\|_\infty\ltn \X \rtn_{\tq,\Pi[a,b]} \notag\\
            &&+ 4C_p L_g \ltn \bx\rtn_{\tp,\Pi[a,b]}\ltn \ta^\Delta,R^{\ta^\Delta} \rtn_{\tp,\Pi[a,b]},
		\end{eqnarray*}
which implies
		   \begin{eqnarray*}
			&&\ltn \ta^\Delta,R^{\ta^\Delta} \rtn_{\tp,\Pi[a,b]}\notag\\
            &\leq& \dfrac{1}{1-\lambda}\left[\|f(\ta^\Delta)\|_{\infty,\Pi[a,b]}(b-a)  +  \|g\|_\infty \ltn x \rtn_{\tp,\Pi[a,b]}+  \|Dg\|_\infty \|g\|_\infty\ltn \X \rtn_{\tq,\Pi[a,b]} \right]. 
		\end{eqnarray*}
Thus,
\begin{eqnarray}\label{aDelta}
   && C_p \|D^2g\|_\infty  \ltn x\rtn_{\tp,\Pi[a,b]}\ltn R^{\ta^\Delta} \rtn_{\tp,\Pi[a,b]}\notag\\
   &&+C_p(L^*_g)^2( \ltn x\rtn^2_{\tp,\Pi[a,b]} +4\ltn \X\rtn_{\tp,\Pi[a,b]})\ltn \ta^\Delta\rtn_{\tp,\Pi[a,b]}  \notag\\
    &\leq& \dfrac{ C_p}{1-\lambda}\left[\|f(\ta^\Delta)\|_{\infty,\Pi[a,b]} (b-a) + \|g\|_\infty \ltn x \rtn_{\tp,\Pi[a,b]}+  \|Dg\|_\infty \|g\|_\infty\ltn \X \rtn_{\tq,\Pi[a,b]}\right] \times \notag\\
    && \hspace{1cm} \times\left[\|D^2g\|_\infty  \ltn x\rtn_{\tp,\Pi[a,b]} +  (L^*_g)^2( \ltn x\rtn^2_{\tp,\Pi[a,b]} +4\ltn \X\rtn_{\tp,\Pi[a,b]})\right] \notag\\
    &\leq& \dfrac{C_p }{1-\lambda}\left[\|f(\ta^\Delta)\|_{\infty,\Pi[a,b]} (b-a) + L_g \ltn \bx \rtn_{\tp,\Pi[a,b]}\right] 4L_g \ltn \bx \rtn_{\tp,\Pi[a,b]}
    \end{eqnarray}
and hence
   \begin{eqnarray*}
		  &&\ltn h,R^h \rtn_{\tp,\Pi[a,b]}\\
		   &\leq&\Big[L_f (b-a)+4C_pL_g \ltn \bx\rtn_{\tp,\Pi[a,b]}+ 2\lambda \|f(\ta^\Delta)\|_{\infty,\Pi[a,b]} (b-a)\\
           &&\hspace{7cm}+\dfrac{\lambda}{1-\lambda} L_g \ltn \bx \rtn_{\tp,\Pi[a,b]}\Big] \| h,R^h\|_{\tp,\Pi[a,b]}\\
		     &\leq&\Big[L_f (b-a)+2\lambda \|f(\ta^\Delta)\|_{\infty,\Pi[a,b]}(b-a)+(4C_p+1)L_g \ltn \bx\rtn_{\tp,\Pi[a,b]} \Big] \| h,R^h\|_{\tp,\Pi[a,b]}.
	\end{eqnarray*}
The proved is completed.\\      
\end{proof}

\begin{proof}[{\bf Proof of Proposition \ref{case1}}] Firstly, by assumption
 \[
    \Big[L_f (\tau^\Delta_{1}-\tau^\Delta_0) +2\lambda \|f(\ta^\Delta)\|_{\infty,\Pi[\tau^\Delta_0,\tau^\Delta_1]} (\tau^\Delta_1-\tau^\Delta_0)+(4C_p+1) L_g\ltn \bx \rtn_{\tp,\Pi[\tau^\Delta_0,\tau^\Delta_{1}]}\Big] \leq \lambda
\]
we apply Proposition \ref{hnew} to obtain
	   \begin{equation}\label{hcase1}
     			\ltn h,R^h \rtn_{\tp,\Pi[\tau^\Delta_0,\tau^\Delta_1]} \leq \frac{\lambda}{1-\lambda} \|h_{\tau^\Delta_0}\|, \qquad
                 \| h,R^h \|_{\tp,\Pi[\tau^\Delta_0,\tau^\Delta_1]}  \leq \frac{1}{1-\lambda}  \|h_{\tau^\Delta_0}\| \leq 2 \|h_{\tau^\Delta_0}\|.
        \end{equation}
and by \ref{aDelta}
 \begin{eqnarray}\label{a2}
	    &&C_p \|D^2g\|_\infty  \ltn x\rtn_{\tp,\Pi[a,b]}\ltn R^{\ta^\Delta} \rtn_{\tp,\Pi[a,b]}\notag\\
        &&+C_p(L^*_g)^2( \ltn x\rtn^2_{\tp,\Pi[a,b]} +4\ltn \X\rtn_{\tp,\Pi[a,b]})\ltn \ta^\Delta\rtn_{\tp,\Pi[a,b]}\notag\\
        &\leq& 8C_pL_g \ltn \bx \rtn_{\tp,\Pi[a,b]}
		\end{eqnarray}      
 Assign $G_{s,t} :=  2 \langle h_s, P_{s,t} \rangle + \| P_{s,t}\|^2$. Since $M(\cdot)$ is bounded by $3L_f$, \eqref{eqRDE:02} implies that whenever $\|h_{t^\Delta_k}\|<r,$
	\begin{eqnarray*}
			\|h_{t_{k+1}}\|^2 &\leq& \|h_{t_k}\|^2 \exp \{2\Delta M(\ta^\Delta_{t_k})+ L_f^2\Delta^2\}+ 2L_f\|h\|_{\infty,\Pi[\tau^\Delta_0,\tau^\Delta_1]}\|P\|_{\infty,\Pi[\tau^\Delta_0,\tau^\Delta_1]} \Delta + G_{t_k,t_{k+1}}.
		\end{eqnarray*}
Now by induction and using the assumption that $L_f (\tau^\Delta_1-\tau^\Delta_0)<\lambda$ we obtain for $t_n\in [\tau^\Delta_0,\tau^\Delta_1)$ the estimates
        \allowdisplaybreaks
        \begin{eqnarray}\label{eqRDE:05}
			\|h_{t_n}\|^2 
			&\leq& \exp\left\{\Delta\displaystyle\sum_{k=0}^{n-1}\left[2M(\ta^\Delta_{t_k})+ L_f^2\Delta\right]\right \} \|h_{\tau^\Delta_0}\|^2 +  \sum_{k=0}^{n-1} \exp\left\{\Delta\displaystyle\sum_{j=k+1}^{n-1}\left[2M(\ta^\Delta_{t_j})+ L_f^2\Delta\right]\right \} G_{t_k,t_{k+1}}  \notag \\
            &&+ \sum_{k=0}^{n-1} \exp\left\{\Delta\displaystyle\sum_{j=k+1}^{n-1}\left[2M(\ta^\Delta_{t_j})+ L_f^2\Delta\right]\right \} 2L_f \Delta  \|h\|_{\infty,\Pi[\tau^\Delta_0,\tau^\Delta_1]}\|P\|_{\infty,\Pi[\tau^\Delta_0,\tau^\Delta_1]}  \notag\\
            &\leq& \exp\left\{\Delta\displaystyle\sum_{k=0}^{n-1}\left[2M(\ta^\Delta_{t_k})+ L_f^2\Delta\right]\right \}  \|h_{\tau^\Delta_0}\|^2 + \sum_{k=0}^{n-1} \exp\left\{\Delta\displaystyle\sum_{j=k+1}^{n-1}\left[2M(\ta^\Delta_{t_j})+ L_f^2\Delta\right]\right \} G_{t_k,t_{k+1}}  \notag\\
            &&+ 2L_f n\Delta  \|h\|_{\infty,\Pi[\tau^\Delta_0,\tau^\Delta_1]}\|P\|_{\infty,\Pi[\tau^\Delta_0,\tau^\Delta_1]} \exp\{(6L_f\Delta+ L^2_f\Delta^2)n\} \notag\\
            &\leq&  \exp\left\{\Delta\displaystyle\sum_{k=0}^{n-1}\left[2M(\ta^\Delta_{t_k})+ L_f^2\Delta\right]\right \}  \|h_{\tau^\Delta_0}\|^2 +  \sum_{k=0}^{n-1} \exp\left\{\Delta\displaystyle\sum_{j=k+1}^{n-1}\left[2M(\ta^\Delta_{t_j})+ L_f^2\Delta\right]\right \} G_{t_k,t_{k+1}}  \notag\\
                &&+ 2\lambda \|h\|_{\infty,\Pi[\tau^\Delta_0,\tau^\Delta_1]}\|P\|_{\infty,\Pi[\tau^\Delta_0,\tau^\Delta_1]} \exp \{(6L_f\Delta+ L^2_f\Delta^2)n\}   
		\end{eqnarray}
provided that $\|h_{t^\Delta_k}\|<r, 0\leq k<n$. Write in short $\beta_j := 2M(\ta^\Delta_{t_j})\Delta+ L_f^2\Delta^2$ and define $F_{s,t} = \exp \{\sum_{j=\frac{s}{\Delta}+1}^n \beta_j\} G_{s,t}$ for every $s<t\in \Pi[\tau^\Delta_0,\tau^\Delta_1]$. Repeating the calculations in \cite[Theorem 4.2 ]{congduchong23}, for $s\leq u\leq t$ yields
\begin{equation*}
			\|(\delta G)_{s,u,t}\| \leq
			2L_f(t-s)\|P\|_{\infty,[s,t]}  \|h\|_{\infty,\Pi[s,t]}+(3+2C_p)W_{s,t}(\delta P)  \left(\|P\|_{\infty,[s,t]} + \|h\|_{\infty,\Pi[s,t]}\right) =: W_{s,t}(\delta G),
		\end{equation*}
in which $W(\delta P)$ is given in \eqref{omegaP}, and
\begin{eqnarray*}
   \| (\delta F)_{s,u,t}\| &=& \left\| \exp \{\sum_{j=\frac{s}{\Delta}+1}^n \beta_j\} G_{s,t} - \exp \{\sum_{j=\frac{s}{\Delta}+1}^n \beta_j\} G_{s,u} -\exp \{\sum_{j=\frac{u}{\Delta}+1}^n \beta_j\} G_{u,t} \right \| \\
   &=& \left\| \exp \{\sum_{j=\frac{s}{\Delta}+1}^n \beta_j\} \delta G_{s,u,t} + \exp \{\sum_{j=\frac{s}{\Delta}+1}^n \beta_j\} G_{u,t} -\exp \{\sum_{j=\frac{u}{\Delta}+1}^n \beta_j\} G_{u,t} \right \| \\
   &\leq& \exp \{\sum_{j=\frac{s}{\Delta}+1}^n \beta_j\} \| \delta G_{s,u,t}\| + \exp \{\sum_{j=\frac{u}{\Delta}+1}^n \beta_j\} \Big(\exp \{\sum_{j=\frac{s}{\Delta}+1}^{\frac{u}{\Delta}} \beta_j\} -1\Big)\|G_{u,t}\|\\
  &\leq& \exp \{\sum_{j=\frac{s}{\Delta}+1}^n \beta_j\} \| \delta G_{s,u,t}\| + \exp \{\sum_{j=\frac{u}{\Delta}+1}^n \beta_j\} \Big(\sum_{j=\frac{s}{\Delta}+1}^{\frac{u}{\Delta}} \beta_j \Big)\exp \{\sum_{j=\frac{s}{\Delta}+1}^{\frac{u}{\Delta}} \beta_j\}\|G_{u,t}\|\\
  &\leq& \exp \{\sum_{j=\frac{s}{\Delta}+1}^n \beta_j\} \Big[ W_{s,t}(\delta G) + \|G\|_{\infty,\Pi[s,t]} \sum_{j=\frac{s}{\Delta}+1}^{\frac{u}{\Delta}} \beta_j \Big]\\
 &\leq& \exp \{\sum_{j=\frac{s}{\Delta}+1}^n \beta_j\} \Big[  \| \delta G_{s,u,t}\| + \|G\|_{\infty,\Pi[s,t]} (6L_f + L_f^2 \Delta)(u-s) \Big]\\
  &\leq& e^{6\lambda+\lambda^2} \Big[ W_{s,t}(\delta G) + \|G\|_{\infty,\Pi[s,t]} (6L_f + L_f^2 \Delta)(t-s) \Big] =:e^{6\lambda+\lambda^2} W_{s,t}(\delta F),
\end{eqnarray*}
where $W_{s,t}(\delta F)$ in the last line is a control. 	
We can now apply the discrete sewing lemma \cite{davie}, \cite[Lemma\ 2.2]{congduchong23} to \eqref{eqRDE:05} to obtain
		\begin{eqnarray*}
			\|h_{t_n}\|^2&\leq&  \exp\left\{\Delta\displaystyle\sum_{k=0}^{n-1}\left[2M(\ta^\Delta_{t_k})+ L_f^2\Delta\right]\right \} \|h_{\tau^\Delta_0}\|^2 \\
            &&+ e^{6\lambda+\lambda^2} \left[2\lambda \|h\|_{\infty,\Pi[\tau^\Delta_0,\tau^\Delta_1]}\|P\|_{\infty,\Pi[\tau^\Delta_0,\tau^\Delta_1]}+\|G_{\tau^\Delta_0,\tau^\Delta_1} \|+ C_pW_{\tau^\Delta_0,\tau^\Delta_1}(\delta F)\right]\\
            &\leq&  \exp\left\{\Delta\displaystyle\sum_{k=0}^{n-1}\left[2M(\ta^\Delta_{t_k})+ L_f^2\Delta\right]\right \} \|h_{\tau^\Delta_0}\|^2 \\
            &&+ e^{6\lambda+\lambda^2} \left[2\lambda \|h\|_{\infty,\Pi[\tau^\Delta_0,\tau^\Delta_1]}\|P\|_{\infty,\Pi[\tau^\Delta_0,\tau^\Delta_1]}+\|G_{\tau^\Delta_0,\tau^\Delta_1} \|\right.\\
            &&\hspace{4cm} \left. + C_p\left(W_{\tau^\Delta_0,\tau^\Delta_1}(\delta G)+ (6\lambda+\lambda^2) \| G\|_{\infty,[\tau^\Delta_0,\tau^\Delta_1]}\right) \right]\\
            &\leq&  \exp\left\{\Delta\displaystyle\sum_{k=0}^{n-1}\left[2M(\ta^\Delta_{t_k})+ L_f^2\Delta\right]\right \} \|h_{\tau^\Delta_0}\|^2 \\
           &&+  e^{6\lambda+\lambda^2}\left[ 2\lambda\|h\|_{\infty,\Pi[\tau^\Delta_0,\tau^\Delta_1]}\|P\|_{\infty,\Pi[\tau^\Delta_0,\tau^\Delta_1]} \right.\\
           &&\hspace{2cm}\left. +\|G\|_{\infty,[\tau^\Delta_0,\tau^\Delta_1]} + C_p W_{\tau^\Delta_0,\tau^\Delta_1}(\delta G)+C_p\|G\|_{\infty,[\tau^\Delta_0,\tau^\Delta_1]}\right]\\
             &\leq&  \exp\left\{\Delta\displaystyle\sum_{k=0}^{n-1}\left[2M(\ta^\Delta_{t_k})+ L_f^2\Delta\right]\right \} \|h_{\tau^\Delta_0}\|^2 \\
           &&+  e^{6\lambda+\lambda^2}\left[ 2\lambda\|h\|_{\infty,\Pi[\tau^\Delta_0,\tau^\Delta_1]}\|P\|_{\infty,\Pi[\tau^\Delta_0,\tau^\Delta_1]} +(1+C_p)\|G\|_{\infty,[\tau^\Delta_0,\tau^\Delta_1]} \right]\\
           &&+e^{6\lambda+\lambda^2} C_p \left[2L_f(\tau^\Delta_1-\tau^\Delta_0)\|P\|_{\infty,\Pi[\tau^\Delta_0,\tau^\Delta_1]}  \|h\|_{\infty,\Pi[\tau^\Delta_0,\tau^\Delta_1]}+\right.\\
           &&\hspace{2cm}\left.+(3+2C_p)W_{\tau^\Delta_0,\tau^\Delta_1}(\delta P)  \left(\|P\|_{\infty,\Pi[\tau^\Delta_0,\tau^\Delta_1]} + \|h\|_{\infty,\Pi[\tau^\Delta_0,\tau^\Delta_1]}\right)\right]\\
             &\leq&  \exp\left\{\Delta\displaystyle\sum_{k=0}^{n-1}\left[2M(\ta^\Delta_{t_k})+ L_f^2\Delta\right]\right \} \|h_{\tau^\Delta_0}\|^2 \\
           &&+  e^{6\lambda+\lambda^2}\left[ 2\lambda\|h\|_{\infty,\Pi[\tau^\Delta_0,\tau^\Delta_1]}\|P\|_{\infty,\Pi[\tau^\Delta_0,\tau^\Delta_1]} +(1+C_p)\|G\|_{\infty,[\tau^\Delta_0,\tau^\Delta_1]} \right]\\
           &&+e^{6\lambda+\lambda^2} C_p \left[2\lambda \|P\|_{\infty,\Pi[\tau^\Delta_0,\tau^\Delta_1]}  \|h\|_{\infty,\Pi[\tau^\Delta_0,\tau^\Delta_1]}+\right.\\
           &&\hspace{2cm}\left.+(3+2C_p)W_{\tau^\Delta_0,\tau^\Delta_1}(\delta P)  \left(\|P\|_{\infty,\Pi[\tau^\Delta_0,\tau^\Delta_1]} + \|h\|_{\infty,\Pi[\tau^\Delta_0,\tau^\Delta_1]}\right)\right]\\
            &\leq&  \exp\left\{\Delta\displaystyle\sum_{k=0}^{n-1}\left[2M(\ta^\Delta_{t_k})+ L_f^2\Delta\right]\right \} \|h_{\tau^\Delta_0}\|^2 \\
           &&+  e^{6\lambda+\lambda^2}\left[ 2\lambda(1+C_p)\|h\|_{\infty,\Pi[\tau^\Delta_0,\tau^\Delta_1]}\|P\|_{\infty,\Pi[\tau^\Delta_0,\tau^\Delta_1]} +(1+C_p)\|G\|_{\infty,[\tau^\Delta_0,\tau^\Delta_1]} \right]\\
           &&+e^{6\lambda+\lambda^2} C_p(3+2C_p)W_{\tau^\Delta_0,\tau^\Delta_1}(\delta P)  \left(\|P\|_{\infty,\Pi[\tau^\Delta_0,\tau^\Delta_1]} + \|h\|_{\infty,\Pi[\tau^\Delta_0,\tau^\Delta_1]}\right).
		\end{eqnarray*}
 Because of \eqref{Pvsh}, 
		\begin{eqnarray*}
             \|P\|_{\infty,\Pi[\tau^\Delta_0,\tau^\Delta_1]} &\leq&\dfrac{5}{4}  L_g\ltn \bx\rtn_{\tp,[\tau^\Delta_0,\tau^\Delta_1]} \|h\|_{\infty,\Pi[\tau^\Delta_0,\tau^\Delta_1]}\\
             \| G\|_{\infty,[\tau^\Delta_0,\tau^\Delta_1]} &\leq& 2\|h\|_{\infty,\Pi[\tau^\Delta_0,\tau^\Delta_1]} \|P\|_{\infty,\Pi[\tau^\Delta_0,\tau^\Delta_1]} +\|P\|^2_{\infty,\Pi[\tau^\Delta_0,\tau^\Delta_1]} \\
            &\leq&  3L_g\ltn \bx\rtn_{\tp,[\tau^\Delta_0,\tau^\Delta_1]} \|h\|^2_{\infty,\Pi[\tau^\Delta_0,\tau^\Delta_1]};
            \end{eqnarray*}	
and of \eqref{omegaP} we have
        \begin{eqnarray*}   	  
			 W_{\tau^\Delta_0,\tau^\Delta_1}(\delta P) &\leq&2L^*_g\ltn \bx\rtn_{\tp,\Pi[\tau^\Delta_0,\tau^\Delta_1]} \ltn h,R^h\rtn_{\tp,\Pi[\tau^\Delta_0,\tau^\Delta_1]}\\
              &&+ 4L^*_g\ltn \bx\rtn_{\tp,\Pi[\tau^\Delta_0,\tau^\Delta_1]} \ltn \ta^\Delta,R^{\ta^\Delta}\rtn_{\tp,\Pi[\tau^\Delta_0,\tau^\Delta_1]}  \| h\|_{\infty,\Pi[\tau^\Delta_0,\tau^\Delta_1]}\\
              &\leq& 10 L_g\ltn \bx\rtn_{\tp,\tau^\Delta_0,\tau^\Delta_1} \|h\|_{\tp,\Pi[\tau^\Delta_0,\tau^\Delta_1]}.
		\end{eqnarray*}	
due to \eqref{a2}. Combining with \eqref{omegaP}, we obtain
\begin{eqnarray}\label{eqRDE:08}
			\|h_{t_n}\|^2    
			 &\leq& \exp\left\{\Delta\displaystyle\sum_{k=0}^{n-1}\left[2M(\ta^\Delta_{t_k})+ L_f^2\Delta\right]\right \} \|h_{\tau^\Delta_0}\|^2 \notag\\
             &&+ e^{6\lambda+\lambda^2}\left(1+C_p\right)\left[\dfrac{5}{2}\lambda  L_g \ltn \bx \rtn_{\tp,[\tau^\Delta_0,\tau^\Delta_1]}+3L_g \ltn \bx \rtn_{\tp,[\tau^\Delta_0,\tau^\Delta_1]}   \right] \|h\|^2_{\infty,\Pi[\tau^\Delta_0,\tau^\Delta_1]} \notag\\
             &&+e^{6\lambda+\lambda^2} C_p(3+2C_p)
             \left[ 10L_g \ltn \bx \rtn_{\tp,[\tau^\Delta_0,\tau^\Delta_1]}   (1+ \dfrac{5}{4}L_g \ltn \bx \rtn_{\tp,[\tau^\Delta_0,\tau^\Delta_1]}) \right] \|h\|_{\tp,\Pi[\tau^\Delta_0,\tau^\Delta_1]}\notag\\
    	 &\leq& \exp\left\{\Delta\displaystyle\sum_{k=0}^{n-1}\left[2M(\ta^\Delta_{t_k})+ L_f^2\Delta\right]\right \} \|h_{\tau^\Delta_0}\|^2 \notag\\
             &&+ e^{6\lambda+\lambda^2}\left(1+C_p\right)(3+3\lambda)L_g \ltn \bx \rtn_{\tp,[\tau^\Delta_0,\tau^\Delta_1]}  \|h\|^2_{\tp,\Pi[\tau^\Delta_0,\tau^\Delta_1]} \notag\\
             &&+ e^{6\lambda+\lambda^2}C_p\left(3+2C_p\right)(1+\lambda)12L_g \ltn \bx \rtn_{\tp,[\tau^\Delta_0,\tau^\Delta_1]}  \|h\|^2_{\tp,\Pi[\tau^\Delta_0,\tau^\Delta_1]} \notag\\
    	 &\leq& \exp\left\{\Delta\displaystyle\sum_{k=0}^{n-1}\left[2M(\ta^\Delta_{t_k})+ L_f^2\Delta\right]\right \} \|h_{\tau^\Delta_0}\|^2 \notag\\
             &&+ e^{6\lambda+\lambda^2}\Big[3(1+\lambda)(1+C_p)+ 12(1+\lambda)(3C^2_p+3C_p)\Big] L_g \ltn \bx \rtn_{\tp,[\tau^\Delta_0,\tau^\Delta_1]}  \|h\|^2_{\tp,\Pi[\tau^\Delta_0,\tau^\Delta_1]} \notag\\
             &\leq&\exp\left\{\Delta\displaystyle\sum_{k=0}^{n-1}\left[2M(\ta^\Delta(\theta_{t_k}\omega))+ L_f^2\Delta\right]\right \} \|h_{\tau^\Delta_0}\|^2 \notag\\
             &&+ 10(1+\lambda)e^{6\lambda+\lambda^2} \left(1+2C_p\right)^2 L_g \ltn \bx \rtn_{\tp,[\tau^\Delta_0,\tau^\Delta_1]}  \|h_{\tau^\Delta_0}\|^2.
    	\end{eqnarray}
where we employ \eqref{hcase1} to the last inequality.
Hence \eqref{eqRDE:08} yields
	\begin{equation*}
			\|h_{t_n}\|^2\leq e^{6\lambda+\lambda^2}\left[1 +10(1+\lambda)(1+2C_p)^2 L_g\ltn \bx\rtn_{\tp,[\tau^\Delta_0,\tau^\Delta_1]} \right]\|h_{\tau^\Delta_0}\|^2 \leq 256(1+C_p)^2\|h_{\tau^\Delta_0}\|^2,
       \end{equation*} 				
which, due to assumption $\|h_{\tau^\Delta_{0}} \|\leq \frac{r}{16(1+C_p)}$,  ensures $\|h_{t_n}\| \leq r$ and \eqref{eqRDE:02} for all $t_n\in [{\tau^\Delta_0},{\tau^\Delta_1}]$. Take $t_n = \tau^\Delta_1$ in \eqref{eqRDE:08}, we obtain
		\begin{eqnarray}\label{ind}
			\|h_{\tau^{\Delta}_1}\|^2
			&\leq& \exp \Big\{(2\bar{M}_0 + L_f^2\Delta)(\tau^{\Delta}_1-\tau^{\Delta}_0) +10(1+\lambda)(1+2C_p)^2e^{12\lambda+2\lambda^2}L_g\ltn \bx\rtn_{\tp,\Pi[\tau^\Delta_0,\tau^\Delta_1]}\Big\}\|h_{\tau^{\Delta}_0}\|^2.\notag
     		  \end{eqnarray}  
which completes the proof.\\
 \end{proof}
  \medskip
\begin{proof}[{\bf Proof of Proposition \ref{attractor}}]
We use the coupling technique here. First, assume 
\[
0<	\Delta < \frac{1}{4} \wedge \frac{d}{2L_f^2} \wedge \frac{1}{2d}=:\Delta^*;
\]
and fix $T>0$ with an $n>0$ such that $ \dfrac{1}{2}<T:= n \Delta \leq 1 < (n+1) \Delta$. Then on $[0,T]$, we define 
\begin{equation}
\mu_{0} = y^\Delta_{0},\quad\mu_{t_{k+1}} = \mu_{t_k} + f(\mu_{t_k})(t_{k+1}-t_k),\quad t_k\in \Pi[0,T].
\end{equation}
Then
    \begin{eqnarray*} 
         \mu_{t_{k+1}}^2
         &\leq& \mu^2_{t_k} +[D_1-D_2\mu^2_{t_k}]\Delta + 2[L_f^2\mu^2_{t_k} + \|f(0)\|^2]\Delta^2\leq  e^{- \frac{D_2 \Delta}{2} } \mu^2_{t_k} + C(f)
   \end{eqnarray*}
    with $C(f)$ is a generic constant depending on $f$.
Therefore, 
    \begin{eqnarray}\label{mu}
      \|\mu_{T}\| 
      \leq  e^{- \frac{D_2 T}{2} }\|\mu_{0}\| + C(f), \quad 
      \|\mu\|_{\infty,[0,T]}
      \leq \|\mu_{0}\| + C(f).
    \end{eqnarray}
Moreover,
    \begin{equation}\label{mu1}
    \ltn \mu\rtn_{1{\rm-var},\Pi[0,T]} \leq  C(f)(\|\mu_{0}\| + 1).
    \end{equation}
Assign
 \[
         h_{t_k} := y^\Delta_{t_k}-\mu_{t_k},\quad t_k\in \Pi[0,T]
   \]
then
\begin{eqnarray*}
h_{t_{k+1}} &= &
h_{t_k}+[g(y_{t_k}) x_{t_k,t_{k+1}} + \partial g(y^\Delta_{t_k}) g(y^\Delta_{t_k})\X(t_k,t_{k+1})]+ [f(y^\Delta_{t_k}) -f(\mu_{t_k})  ]  (t_{k+1}-t_k) \\
&=:& h_{t_k}+ F_{t_k,t_{k+1}} +[f(y^\Delta_{t_k}) -f(\mu_{t_k})  ]  (t_{k+1}-t_k)
\end{eqnarray*}
in which $F_{s,t}: = g(y^\Delta_s)x_{s,t} + Dg(y^\Delta_s)g(y^\Delta_s)\X_{s,t}$.\\
We define an auxiliary quantity $R^h_{s,t}:= h_{s,t} - g(y_s)x_{s,t} = R^{y^\Delta}_{s,t} - \mu_{s,t}$. Then
\[
\|R^{y^\Delta}_{s,t} \| \leq \|R^h_{s,t} \| +\|\mu_{s,t}\|.
\]
Observe that $\|(\delta F)_{s,u,t}\| : = \|F_{s,t} - F_{s,u}-F_{u,t}\|$ satisfies
\[
\|(\delta F)_{s,u,t}\| \leq D(g) \left(\ltn x\rtn_{\tp,\Pi[s,t]}+\ltn \X\rtn_{\tq,\Pi[s,t]}\right),\; s\leq u\leq t
\]
in which $D(g)$ is a generic constant depending on $\|g\|_{C^3_b}$.
We modify formula (3.2) in \cite{congduchong23} to obtain for $s\leq u\leq t;\; s,u,t\in\Pi,$ (see more in \cite[Corollary 2.10]{duchong21})
\begin{eqnarray*}
    \|(\delta F)_{s,u,t}\| &\leq&  D(g) \left[ \left(\ltn R^{h}\rtn_{\tq,\Pi[s,t]} + \ltn\mu\rtn^{2/p}_{1{\rm-var},\Pi[s,t]}\right) \ltn x \rtn_{\tp,\Pi[s,t]} \right. \\
    && \hspace{3cm}\left.+( \ltn h \rtn_{\tp,\Pi[s,t]}+ \ltn \mu\rtn^{1/p}_{1{\rm-var},\Pi[s,t]}) \ltn \X \rtn_{\tq,\Pi[s,t]} \right]
    =:W(\delta F)_{s,t}
\end{eqnarray*}
which is a control.  Applying \cite[Corollary\ 2.4]{congduchong23} we have
\[
\| h_{s,t}-g(y_s)x_{s,t} - Dg(y_s)g(y_s)\X_{s,t}\| \leq L_f\|h\|_{\infty,\Pi[s,t]}(t-s) + C_p W(\delta F)_{s,t}.
\]
Then similar to the proof in Proposition \ref{hnew},
      \begin{eqnarray}\label{h}
 	     \ltn h,R^h\rtn_{\tp  ,\Pi[a,b]}
    &\leq&
         L_f\|h_a\|(b-a)+ D(g)  (\ltn x\rtn_{\tp,\Pi[a,b]}  +\ltn \X\rtn_{\tq,\Pi[a,b]}  ) \notag\\
         &&\hspace{-1cm}+\left[ L_f(b-a)+C_p D(g) \ltn x\rtn_{\tp,\Pi[a,b]} +C_pD(g) \ltn \X \rtn_{\tq,\Pi[a,b]}) \right] \ltn h, R^h\rtn_{\tp,\Pi[a,b]}\notag\\
         &&\hspace{-1cm}+C_pD(g)\left[ \ltn \mu\rtn^{2/p}_{1{\rm-var},\Pi[a,b]} \ltn x\rtn_{\tp,\Pi[a,b]}+ \ltn \mu\rtn^{1/p}_{1{\rm-var},\Pi[a,b]} \ltn \X\rtn_{\tq,\Pi[a,b]}  \right] .
     \end{eqnarray}
Now we repeat arguments in \cite[Theorem\ 3.3]{congduchong23}, to construct the same sequence of stopping times $\{\hat{\tau}_i\}$ 
on $[0,T]$ as presented at the end of Subsection 3.4, base on  $\gamma < \frac{1}{2}$, and set of controls $\mathcal{S}=\{\bw^{(1)},\bw^{(2)},\bw^{(3)}\}$ where $\bw^{(1)}_{s,t}=L_f(t-s), \beta_1 = 1$, $\bw^{(2)}_{s,t} = C_p^pD(g)^p\ltn x \rtn^p_{\tp,\Pi[s,t]}, \beta_2 = \frac{1}{p}$,  $\bw^{(3)}_{s,t} =  C_p^pD(g)^p\ltn \X \rtn^q_{\tq,\Pi[s,t]}, \beta_3 = \frac{1}{q}$. By \eqref{N_hat}, the number of $\{\hat{\tau}_i\}$ can be estimated as
\[ 
\hat{N}\leq  2 + \frac{2}{\gamma^p} 4^{p-1}\Big[L_f^pT^p +  C^p_pD(g)^p\ltn x \rtn^p_{\tp,\Pi[a,b]} + C^p_pD(g)^p\ltn \X \rtn^q_{\tq,\Pi[a,b]}\Big].
\] 
Hence 
 \begin{eqnarray*}
 \| h\|_{\infty  ,\Pi[0,T]}
	&\leq&
\|h_0\|e^{4 L_fT}+ D(g)  (\ltn x\rtn_{\tp,\Pi[a,b]}  +\ltn \X\rtn_{\tq,\Pi[a,b]}  ) \hat{N}\\
&&+ C_pD(g)\left[ \ltn \mu\rtn^{2/p}_{1{\rm-var},\Pi[a,b]} \ltn x\rtn_{\tq,\Pi[a,b]}+ \ltn \mu\rtn^{1/p}_{1{\rm-var},\Pi[a,b]} \ltn \X\rtn_{\tq,\Pi[a,b]}  \right] \\
&\leq&
 D(g)  (\ltn x\rtn_{\tp,\Pi[a,b]}  +\ltn \X\rtn_{\tq,\Pi[a,b]}  ) \hat{N}\\
&&+ C_pD(g)\left[ \ltn \mu\rtn^{2/p}_{1{\rm-var},\Pi[a,b]} \ltn x\rtn_{\tp,\Pi[a,b]}+ \ltn \mu\rtn^{1/p}_{1{\rm-var},\Pi[a,b]} \ltn \X\rtn_{\tq,\Pi[a,b]}  \right] 
\end{eqnarray*}
in which we use the sub-additivity of control
\[
C_pD(g)\left[ \ltn \mu\rtn^{2/p}_{1{\rm-var},\Pi[a,b]} \ltn x\rtn_{\tp,\Pi[a,b]}+ \ltn \mu\rtn^{1/p}_{1{\rm-var},\Pi[a,b]} \ltn \X\rtn_{\tq,\Pi[a,b]}  \right].
\] 
Using \eqref{mu1} and Young inequality, we obtain
 \begin{equation}\label{h_inf}
\| h\|_{\infty  ,\Pi[0,T]}\leq (e^{-\frac{D_2}{8}}-e^{\frac{D_2}{4}} )  \|\mu_0\| +  \Gamma(1+ \ltn \bx \rtn^{p+2}_{\tp,[0,T]})
\end{equation}
for some generic constant $\Gamma$ depends on $L_f, \|f(0)\|, D_1,D_2, \|g\|_ {C^3_b}$; moreover $\Gamma$ is an increasing function with respect to $\|g\|_ {C^3_b}$.  Next, we apply the convex inequality 
to conclude that
\begin{eqnarray*}
\|y^\Delta_T\| &\leq & \| h_T\| +\|\mu_T\|\\
	&\leq &  (e^{-\frac{D_2}{8}}-e^{-\frac{D_2}{4}} )  \|y^\Delta_0\| + \Gamma(1+ \ltn \bx \rtn^{p+2}_{\tp,[0,T]})  +e^{-\frac{D_2T}{2}}\|\mu_{0}\| +  C(f)\\
	&\leq &  e^{-\frac{D_2}{8}} \|y_0\| + \Gamma(1+ \ltn \bx \rtn^{p+2}_{\tp,[0,T]})\\
\|y_T\|^p &\leq & e^{-\frac{D_2}{8}} \|y_0\|^p + \Gamma(1+ \ltn \bx \rtn^{(p+2)p}_{\tp,[0,T]}).
\end{eqnarray*}
Similar estimate holds for norm of $y^\Delta$ on arbitrary $[kT, (k+1)T]$. The remaining part follows Theorem 4.2 and  Proposition 5.2 in \cite{congduchong23} step by step. 

\end{proof}

		\bibliographystyle{siamplain}
		
	\end{document}